\documentclass{article}
\usepackage[utf8x]{inputenc}
\usepackage{amsmath}
\usepackage{amssymb}
\usepackage{amsthm, thmtools}
\usepackage{graphicx}
\usepackage[colorinlistoftodos]{todonotes}
\usepackage[colorlinks=true, allcolors=blue]{hyperref}
\usepackage{caption}
\usepackage{subcaption}
\usepackage{epstopdf} %to include eps images
\usepackage{authblk}
\usepackage{bbm}
\usepackage{comment}
\usepackage{overpic}
\usepackage{rotating}

\usepackage[numbers,sort]{natbib}

\newtheorem{Thm}{Theorem}
\newtheorem{prop}[Thm]{Proposition}
\newtheorem{cor}[Thm]{Corollary}
\newtheorem{lemma}[Thm]{Lemma}
\theoremstyle{definition}
\newtheorem{Def}[Thm]{Definition}
\newtheorem{expl}{Example}
\newtheorem{rem}{Remark}

\newcommand{\rmd}{\mathrm{d}}
\newcommand{\dist}{\mathrm{dist}}

\newcommand{\rate}{\gamma}

\title{Synchronization and random attractors for reaction jump processes}

\author[b]{ Maximilian Engel\footnote{maximilian.engel@fu-berlin.de}}
\author[b]{Guillermo Olic\'on-M\'endez}
\author[a,c]{Nathalie Unger}
\author[a]{Stefanie Winkelmann\footnote{winkelmann@zib.de}}

\affil[a]{Zuse Institute Berlin, 14195 Berlin, Germany}
\affil[b]{Freie Universit\" at Berlin, Institut f\"ur Mathematik und Informatik, 14195 Berlin, Germany}
\affil[c]{Technische Universit\" at Berlin, 10623 Berlin, Germany}

\begin{document}

\maketitle

\begin{abstract} 
%\me{For the title we may also consider partial synchronization and/or random cycles/ random periodic orbits} \\
This work explores a synchronization-like phenomenon induced by common noise for continuous-time Markov jump processes given by chemical reaction networks. A corresponding random dynamical system is formulated in a two-step procedure, 
at first for the states of the embedded discrete-time Markov chain and then for the augmented Markov chain including also random jump times.  
We uncover a time-shifted synchronization in the sense that -- after some initial waiting time -- one trajectory exactly replicates another one with a certain time delay. Whether or not such a synchronization behaviour occurs depends on the combination of the initial states. We prove this partial time-shifted synchronization for the special setting of a birth-death process by analyzing the corresponding two-point motion of the embedded Markov chain and determine the structure of the associated random attractor. In this context, we also provide general results on existence and form of random attractors for discrete-time, discrete-space random dynamical systems. 
\end{abstract}



\section{Introduction}

Stochastic models of biochemical reaction dynamics are mostly based on the theory of Markov processes \cite{anderson2011continuous,mcquarrie1967stochastic}. A central role play \textit{reaction jump processes} which model a well-mixed reaction system as a continuous-time Markov process on a discrete state space. The state is given by the number of particles of each involved chemical species, and chemical reactions are modeled as stochastic events which induce jumps in the system's state that occur after exponentially distributed 
%waiting
sojourn times. 
The temporal evolution of the system's probability distribution is in this case characterized by the well-known \textit{chemical master equation} \cite{gillespie1992rigorous}. %Other stochastic modeling approaches for biochemical reaction dynamics include stochastic differential equations (SDEs) \cite{gillespie2000chemical}, 
Besides such reaction jump processes, there exist also modeling approaches using discrete-time Markov chains \cite{ko1991stochastic,huang2019}, stochastic differential equations (SDEs)  or ordinary differential equations (ODEs) \cite{gillespie2000chemical,kurtz1972relationship,kurtz1970solutions}, which approximate the dynamics on a macroscopic level in case of large population sizes, as well as hybrid model recombinations for multiscale reaction systems 
\cite{jahnke2011reduced,zeiser2008simulation,winkelmann2017hybrid,menz2012hybrid}. 
All these approaches for describing and analyzing stochastic phenomena within biochemical or other types of applied contexts have extensively been studied in the literature \cite{winkelmann2020,anderson2015stochastic,wilkinson2018stochastic}. 

\subsection{Background and related work}
The counterpart to stochastic processes within dynamical system theory is given by \textit{random dynamical systems (RDS)}. Here, the origin of uncertainties is considered somewhat differently. In simple terms, the system evolves according to deterministic %(transformation?) 
maps which are chosen randomly from a stochastic law. Formally speaking, an RDS $(\theta,\varphi)$ on a metric state space $\mathbb{X}$ (endowed with its Borel $\sigma$-algebra $\mathcal{B}(\mathbb{X})$) and discrete time set $\mathbb{T}=\mathbb{N}_0$ or $\mathbb{Z}$ consists of
\begin{itemize}
    \item a \textit{noise model}, given by a metric dynamical system $(\Omega,\mathcal{F},\mathbb{P},(\theta_n)_{n\in\mathbb{T}})$. By this we mean that $(\Omega, \mathcal{F},\mathbb{P})$ is a probability space and $\theta:=(\theta^n)_{n\in\mathbb{T}}$ is a family of measurable maps $\theta^n:\Omega\rightarrow\Omega$ for which $\theta^{n+m}=\theta^n\circ\theta^m$ for all $n,m\in\mathbb{T}$, and which is invariant with respect to $\mathbb{P}$ (or, $\mathbb{P}$ is $\theta$-invariant). This last statement means that $\theta^n_*\mathbb{P}(\cdot):=\mathbb{P}((\theta^n)^{-1}(\cdot))=\theta(\cdot)$, 
    
    \item a \textit{cocycle map} $\varphi:\mathbb{N}_0\times\Omega\times \mathbb{X}\rightarrow\mathbb{X}$, with $(n,\omega,x)\mapsto \varphi^n_{\omega}(x)$, which is measurable and satisfies the cocycle property over $\theta$, that is for all $x\in\mathbb{X}$, $n,m\in\mathbb{N}_0$, and $\omega\in\Omega$
    \begin{equation}
        \label{eq:cocycle}
        \varphi^0_\omega(x)=x, \qquad \varphi^{n+m}_\omega(x)=\varphi^n_{\theta^m\omega}\circ \varphi^m_{\omega}(x).
    \end{equation}
\end{itemize}
Notice that while the dynamics on the noise space $\Omega$ might be defined for both positive and negative times, this does not need to be the case for the cocycle map $\varphi$ which in our context will only be defined for times on $\mathbb{N}_0$. For a comprehensive theoretical background of RDS we refer to \cite{arnold1998}.

For a discrete-time system on a finite state space the maps are given by deterministic transitions matrices (containing only entries zero and one), and the expectation of the matrix-valued random variable of transitions maps
%/transformations 
agrees with the stochastic transition matrix of the corresponding Markov chain. The relation between such finite-state RDS and the related Markov chains has been studied 
%in the papers 
by F. Ye et al. \cite{ye2016,ye2019}. Among other things, it has been found that a given finite-state RDS induces a 
%single/
unique Markov chain, while one Markov chain might be compatible with several RDS \cite{ye2016}, as already discussed in a general context by Kifer \cite{Kifer86}.
In this sense, the RDS formulation may be seen as a more refined model of stochastic dynamics than the Markov chain: the former gives a precise description of the two-point motion, comparing trajectories with different initial conditions but driven by the same noise allowing for the analysis of \emph{random attractors} \cite{crauel2015}, whereas the latter characterizes the statistics of the one-point motion by means of the transition probabilities.
%Another description of stochastic phenomena.
%The mathematical setup is different 

While RDS representations of Markov chains (discrete in space and time) or SDEs (continuous in space and time) (see e.g.~\cite{arnold1998}) have been studied in the literature, an analogous investigation for continuous-time Markov processes on discrete state spaces is still missing. In the present work, we do a first step in this direction by formulating random dynamical systems corresponding to reaction jump processes as special types of continuous-time Markov processes. 
%We use the presentation of the reaction jump process by the RDS in order
Our goal is to study questions of synchronization: Given the same noise realization, will trajectories starting at different initial states approach each other in the course of time? Once they coincide at a certain time point, do they stay together forever? Numerical experiments have shown that two realizations of the reaction jump process with distinct starting points (but the same driving noise) may actually resemble each other after some time period in the sense that one of the trajectories appears to be a time-delayed replicate of the other. That is, after some random initial ``finding time", the two process realizations start to wander through the same sequence of states, with identical sojourn times in each of these states, but with a certain time lag with respect to each other. Whether or not this type of trajectory %imitation/
replication happens seems to depend in general on the combination of chosen initial states. By means of the RDS presentation of the dynamics, we provide an analytical explanation for this intriguing phenomenon of time-shifted synchronization and its dependency on the initial conditions.  

\subsection{Main results}
For our analysis, we use the fact that a (continuous-time) Markov jump process $(X(t))_{t\geq 0}$ has a discrete-time representation given by the 
%so called 
\textit{augmented Markov chain} \cite{sikorski2020augmented} which assigns to each discrete index $n$ the random time $T_n$ where the $n$th jump of the process occurs, as well as the state $X_n=X(T_n)$ entered by the process 
%by this jump/
at this jump time. The random sequence $(X_n)_{n\in \mathbb{N}_0}$ of states, called  \textit{embedded Markov chain}, is 
%actually 
a standard (discrete-time) Markov chain on a countable state space. % generally infinite state space (\gom{can call it simply a countable state space}).
In particular, we use an explicit recursive formula for this Markov chain which immediately yields the cocycle of an RDS. We show that the time-shifted synchronization of a time-homogeneous reaction jump process is equivalent to the ``normal" synchronization of the embedded Markov chain, for an appropriate subset of initial conditions. Given that the jump rate constants are time-independent, also the sojourn times within the states will agree once that the states do agree.  

In more detail, we focus on two main examples, providing several general insights on random attractors for discrete state spaces on the way: a simple birth-death process and the Schl\"ogl model with their monostable and bistable structures, respectively,
% a simple birth-death process with its monostable and the Schlögl model with its bistable structure 
detecting similarities and differences in the described synchronization behaviour. We obtain the following main results and insights:
\begin{itemize}
    \item For the embedded Markov chain of the birth-death process we prove \emph{partial synchronization} (and, by that, partial time-shifted synchronization for the reaction jump process) in the sense that common-noise trajectories with starting states of the same parity join each other in finite time while initial states with different parity lead to oscillations around each other (Proposition~\ref{PROP:synchronisation} and Corollary~\ref{COR:partial synchronization}).
    
    \item For general RDS corresponding with Markov chains, we relate different forms of random attractors (Theorem~\ref{THM:weakEquiv}) and give conditions for the existence of a (weak) random attractor (Theorem~\ref{THM:existenceWeakAttractor}). We verify these conditions for our examples (Proposition~\ref{PROP:stationary distribution}), prominently using the existence of a unique stationary distribution.
    
    \item We provide a main analytical result for the birth-death case, characterizing the random attractor as a pullback and forward attractor consisting of two random points with distance $1$ that form a random periodic orbit (Theorem~\ref{THM:strongAttractor}), by that also finding the structure of the corresponding sample measures, also called \emph{statistical equilibria} (Proposition~\ref{prop:samplemeasures}).
    
    \item We illustrate numerical insights that the weak attractor for the Schl\"{o}gl model has the same structure as for simple birth-death, apart from the fact that the distance of the two random points is not $1$, mirroring the bistability of the model.
\end{itemize}
Except for the general insights on random attractors on countable state spaces, most of our analytical results are, so far, restricted to the special case of a birth-death process since the absorbing property of the (thick) diagonal for the two-point motion can be used for this case. However, the general structure of the proof may well be extended to more general chemical reaction networks, also with multiple reactants and corresponding random periodic orbits.

Note that the works \cite{ye2016, ye2019, huang2019} mentioned earlier also deal with synchronization of RDSs for Markov chains, and in \cite{huang2019} even partial synchronization is considered. However, the latter approach is focused on linear cocycles for random networks, using the theory of Lyapunov exponents. Our proofs deploy an analysis of the two-point motion and its consequences for the random attractor, and do not require a linear interpretation,
being confronted with an infinite state space. 
Note that Newman's work on synchronization for RDS \cite{Newman2018, Newman2020} achieves general equivalent conditions for synchronization to occur, which can typically be verified via the \textit{maximal Lyapunov exponent} when the state space is a smooth manifold. For the class of examples considered in this work, the equivalence of these conditions, adjusted to the problem of partial synchronization, will automatically appear in a straight-forward manner.

\subsection{Structure of the paper}
The remainder of the paper is structured as follows. In Sec.~\ref{sec:RJP and RDS}, we introduce reaction jump processes and the corresponding augmented and embedded Markov chains, interpreting the latter as random dynamical systems, and relate their different forms of (partial) synchronization. Sec.~\ref{sec:synchro} is dedicated to proving partial synchronization of the birth-death chain, by using arguments on the two-point motion. In Sec.~\ref{sec:attractors}, we discuss general properties of weak, pullback and forward attractors for the discrete setting (Sec.~\ref{sec:attractors_gen}), show a general result on the existence of weak attractors including the setting of reaction problems (Sec.~\ref{sec:existence_weak_attr}), characterize the structure of this attractor for the birth-death case as a random periodic orbit (Sec.~\ref{sec:birth-death_attr}) and give illustrations of the two-point motion and the stationary, statistical behavior also for the more complicated Schl\"{o}gl model (Sec.~\ref{sec:Schloegl}). Finally, we provide a conclusion with outlook in Sec.~\ref{sec:conclusion}.

\section{Reaction jump processes and random dynamical systems}
\label{sec:RJP and RDS}

In this section, we introduce the reaction network under consideration and formulate the corresponding stochastic dynamics. At first (in Sec.~\ref{sec:RJP}), the pathwise formulation of the reaction jump process is given, including two exemplary reaction networks which will be extensively studied in this work. The related random dynamical systems will be formulated in Sec.~\ref{Sec:RDS} and Sec.~\ref{Sec:RDS2}.  

\subsection{The reaction jump process}\label{sec:RJP}

We consider the standard setting of well-mixed stochastic chemical reaction dynamics \cite{winkelmann2020}: 
There is a system of particles with $L\in \mathbb{N}$ different types/species $\mathcal{S}_1,\dotsc,\mathcal{S}_L$. The particles interact by $K\in \mathbb{N}$ chemical reactions $\mathcal{R}_1,\dotsc,\mathcal{R}_K$ given by
\[   \mathcal{R}_k : \quad \sum_{l=1}^L s_{lk} \mathcal{S}_l \to  \sum_{l=1}^L s'_{lk} \mathcal{S}_l, \]
where the stoichiometric coefficients $s_{kl},s'_{lk}$ are non-negative integers. The state of the system is given by a vector $x=(x_l)_{l=1,...,L}\in \mathbb{N}_0^L$ with $x_l$ counting the number of particles of species $l$. Each reaction induces a jump in the state of the form $x \mapsto x +\nu_k$, where $\nu_k=(\nu_{1k},...,\nu_{Lk})\in \mathbb{Z}^L$ is the state-change vector  given by $\nu_{lk}:=s'_{kl}-s_{lk}$. Given a state $x$, the reaction $\mathcal{R}_k$ takes place at rate $\alpha_k(x)$, where $\alpha_k:\mathbb{N}_0^L \to [0,\infty)$ is the corresponding propensity function.

The resulting reaction jump process (RJP) on $\mathbb{X}=\mathbb{N}_0^L$ has the path-wise representation 
\[ X(t) = X(0) + \sum_{k=1}^K \mathcal{U}_k\left(\int_0^t \alpha_k(X(s)) ds\right)\nu_k,  \]
where $\mathcal{U }_k$ are independent unit-rate Poisson processes. 
This process (and equivalently a more general Markov jump process) is fully characterized by the random jump times $T_n$, $n=1,2,...$, at which the jumps (here reactions) take place and the states $X_n:=X(T_n)$ that are entered at the jump times, namely by
\[ X(t) = X_n \quad \mbox{for} \; T_n \leq t < T_{n+1}, \]
with $T_0 = 0$ and $X_0=X(0)$. 
That is, we can consider a division of the Markov jump process into the process of \textit{jump times} $(T_n)_{n \in \mathbb{N}_0}$ with values in $[0, \infty)$ and the process of the states $(X_n)_{n \in \mathbb{N}_0}$ in $\mathbb{X}$ which is called the \textit{embedded Markov chain}. The discrete-time process $(X_n,T_n)_{n \in \mathbb{N}_0}$ is called the \textit{augmented Markov chain} \cite{sikorski2020augmented}.  

%Let $(r_n)_{n\in \mathbb{N}_0}$ and $(q_n)_{n \in \mathbb{N}_0}$ be sequences of independent (of each other and for different $n$) and uniformly distributed numbers in $[0,1]$, i.e.
%\begin{equation}\label{r_nq_n}
%    r_n,q_n \sim U(0,1) \quad \text{for all } n \in \mathbb{N}_0.
%\end{equation}
% It is well-known (see \cite{gillespie1976general,gillespie1977exact}) that there exist functions $\tau:\mathbb{X} \times [0,1] \to [0,\infty) $ and $\kappa:\mathbb{X} \times [0,1] \to \{1,...,K\}$ such that the jump times and the states of the RJP are recursively given by 
% \begin{align} 
%     T_{n+1} & =  T_n  + \tau(X_n,r_{n}), \label{T_n+1} \\
%     X_{n+1} & =  X_{n} + \nu_{ \kappa(X_n,q_{n})}, \label{X_n+1},
% \end{align}
% with $T_0=0$, $X_0=X(0)$. \Nathalie{The right hand sides of \eqref{T_n+1} and \eqref{X_n+1} depend on "$\omega$" and the left hand sides not}

The jump times and the states of the RJP are recursively given by \begin{align} 
    T_{n+1} & =  T_n  + \tau(X_n) \label{T_n+1} \\
    X_{n+1} & =  X_{n} + \nu_{ \kappa(X_n)} \label{X_n+1},
\end{align}
with $T_0=0$, $X_0=X(0)$, where $\tau(x)$ is an exponentially distributed random variable with mean $1/x$ and $\kappa(x)\in \{1,...,K\}$ is a random variable with point probabilities $\alpha_k(x)/\sum_{l=1}^K \alpha_l(x)$ for $k=1,...,K$. 
It is well known (see \cite{gillespie1976general,gillespie1977exact}) that $\tau$ and $\kappa$ can be realised by taking independent, uniformly distributed random numbers $r,q \sim U(0,1)$ and setting
\begin{equation}\label{tau}
    \tau(x,r) = \frac{1}{\sum_{k=1}^K\alpha_k(x)}\log \left(\frac{1}{r}\right),
\end{equation}
and  $\kappa(x,q)$ is the smallest integer satisfying 
\begin{equation} \label{kappa}
    \sum_{k=1}^{\kappa(x,q)} \alpha_k(x) > q\sum_{k=1}^K \alpha_k(x).
\end{equation}
Note that on a pathwise level $\kappa(x,q)$ depends on the order of reaction indices.

\begin{expl}[Birth-death process]\label{ex:birth-death}
As a basic example which will be analyzed in detail in Sec.~\ref{sec:birth-death_syn} and Sec.~\ref{sec:birth-death_attr} we consider the standard birth-death process of a single species $\mathcal{S}$ given by $K=2$ reactions
\begin{equation*}
    \mathcal{R}_1:  \emptyset  \stackrel{\rate_1}{\longrightarrow} \mathcal{S}, \quad
    \mathcal{R}_2: \mathcal{S} \stackrel{\rate_2}{\longrightarrow} \emptyset.
\end{equation*}
Here, $\rate_1,\rate_2>0$ are rate constants and the corresponding propensity functions are given by the \textit{law of mass action} as
\[ \alpha_1(x)=\rate_1, \quad \alpha_2(x)=\rate_2 x.\]
The state space of the resulting jump process is given by $\mathbb{X}=\mathbb{N}_0$. Consequently, also the state-change vectors $\nu_k$ are actually scalar and given by $\nu_1=1$ and $\nu_2=-1$. From \eqref{kappa} we can deduce that 
\begin{equation} \label{eq:kappa}
    \kappa(x,q)= \begin{cases} 1 & \textup{if } q<\frac{\rate_1}{\rate_1+\rate_2x},
\\ 2 & \textup{otherwise,}
\end{cases}
\end{equation}
which means that reaction $\mathcal{R}_1$ takes place with probability $\frac{\rate_1}{\rate_1+\rate_2x}$ given that the system is in state $x$, while $\mathcal{R}_2$ takes place with probability $1-\frac{\rate_1}{\rate_1+\rate_2x}$.

Under an appropriate scaling of the propensity functions $\alpha_1$ and $\alpha_2$, 
one may derive the corresponding \emph{reaction rate equation} governing the dynamics
of the concentration $C(t) = \frac{X^V(t)}{V}$ for the large volume limit
$V\to \infty$ (cf.~\cite{winkelmann2020}). This is an ordinary differential equation (ODE), given by
\begin{equation}\label{RRE:birth-death}
    \frac{\rmd C(t)}{\rmd t} = - \gamma_2 C(t)+ \gamma_1,
\end{equation}
with globally attracting equilibrium at $C = \frac{\gamma_1}{\gamma_2}$ (see Figure~\ref{fig:ODE}).
\end{expl}
%
%\me{We still have to coordinate this part with the later sections; take the following as preliminary...}
%
\begin{expl}[Schl\"ogl model] \label{ex:schloegl}
The other main example of this work is the Schlögl model, a chemical reaction network exhibiting bistability, cf.~e.g.~\cite{endres2017entropy,schlogl1972chemical,matheson1975stochastic} and see Sec.~\ref{sec:Schloegl} for a detailed discussion. 
%\me{Was there not also an article we discussed at the very beginning?} 
Again we only have one species $\mathcal{S}$ with the following reactions
\begin{equation*} 
\begin{array}{lcl}
\mathcal{R}_1:  \emptyset  \stackrel{\rate_1}{\longrightarrow} \mathcal{S}, & \quad &\mathcal{R}_2: \mathcal{S} \stackrel{\rate_2}{\longrightarrow} \emptyset,
\\\mathcal{R}_3:  2\mathcal{S}  \stackrel{\rate_3}{\longrightarrow} 3\mathcal{S}, &\quad
&\mathcal{R}_4: 3\mathcal{S} \stackrel{\rate_4}{\longrightarrow} 2\mathcal{S},
\end{array}
\end{equation*}
for rate constants $\rate_1, \rate_2, \rate_3, \rate_4 > 0$ and $x \in \mathbb{N}_0$.
The corresponding standard mass action propensities are given by
\begin{equation*} 
\begin{array}{lcl}
\alpha_1(x) = \rate_1, & \quad &\alpha_2(x)=\rate_2 x,,
\\\alpha_3(x) = \rate_3 x(x-1), &\quad
&\alpha_4(x)=\rate_4x(x-1)(x-2).
\end{array}
\end{equation*}
The Schlögl model has the state-change vectors $\nu_1, \nu_3=1$ and $\nu_2, \nu_4=-1$.

The corresponding reaction rate ODE is given by
\begin{equation}\label{RRE:Schloegl}
    \frac{\rmd C(t)}{\rmd t} = - \rate_4 C(t)^3 + \rate_3 C(t)^2 - \rate_2 C(t)+ \rate_1,
\end{equation}
which -- depending on the value of the reaction rates -- may exhibit an unstable equilibrium and two stable equilibria (see Figure~\ref{fig:ODE}).
\end{expl}
The two examples demonstrate two fundamentally different patterns in terms of the large-volume behavior of the process: whereas in Example~\ref{ex:birth-death} the reaction rate equation has one globally attracting equilibrium such that all
trajectories synchronize to the same concentration, in Example~\ref{ex:schloegl}
there are two (locally) attracting equilibria separated by an unstable equilibrium such that different ODE-trajectories may or may not synchronize depending on their initial conditions, see Figure~\ref{fig:ODE}.  %\sw{This text refers to trajectories of the deterministic dynamics (= ODE solutions), right?} \me{Yes, as made clear by the term reaction rate equation.}

This observation prompts interest in the synchronizing behavior of the underlying reaction jump process: Given the same random numbers but different initial states, will the jump times and states given by \eqref{T_n+1}-\eqref{X_n+1} approach each other in the course of time? Can we observe different synchronization behavior in Examples~\ref{ex:birth-death} and \ref{ex:schloegl}?
To give a systematic answer to these questions, we will analyze the reaction jump process in terms of the corresponding RDS which gives a natural approach to comparing trajectories with different initial conditions but driven by the same noise realizations.

\begin{figure}
    \centering
    \begin{subfigure}[b]{0.49\textwidth}
        \centering
        \includegraphics[width=\textwidth]{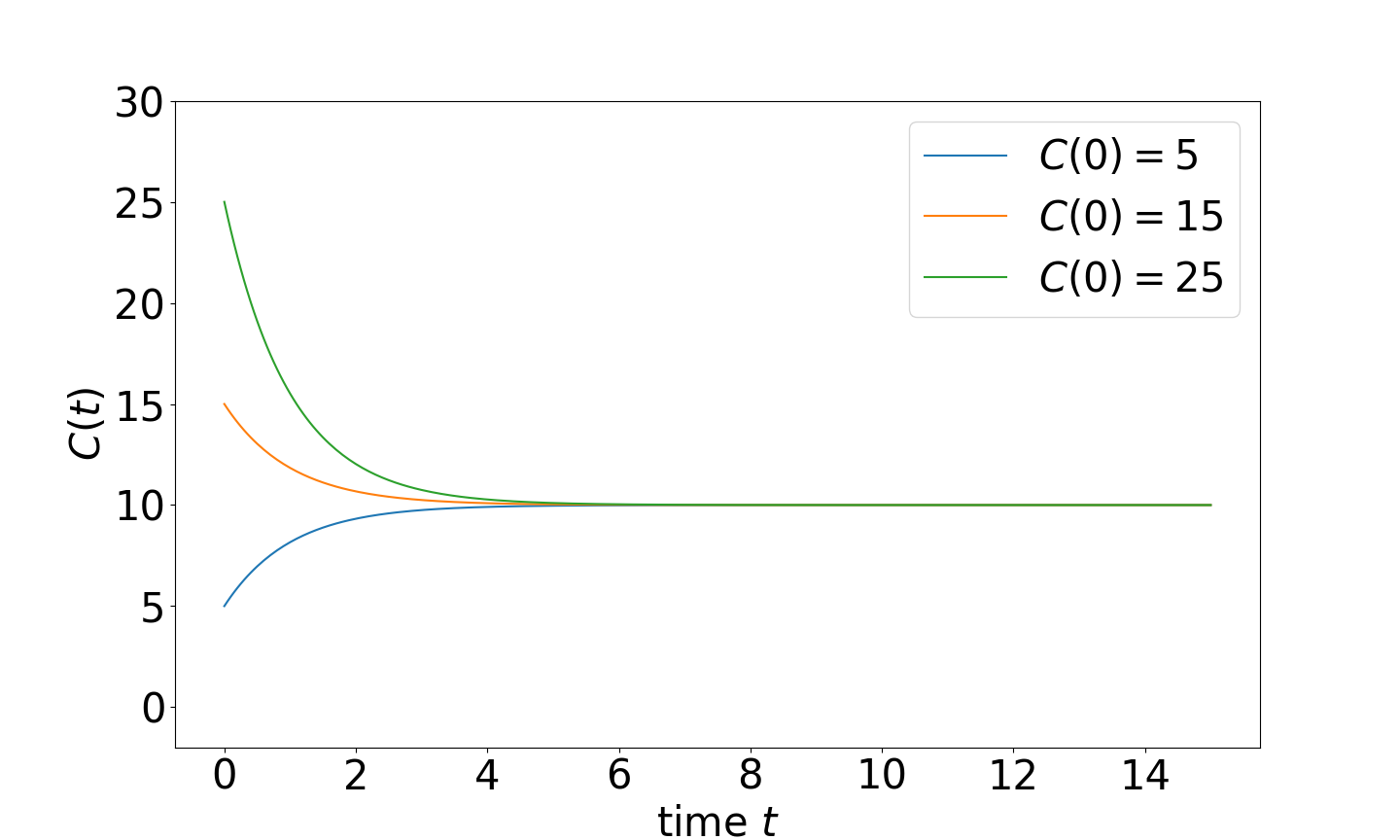}
        \caption{}
    \end{subfigure}
    \hfill
    \begin{subfigure}[b]{0.49\textwidth}
        \centering
        \includegraphics[width=\textwidth]{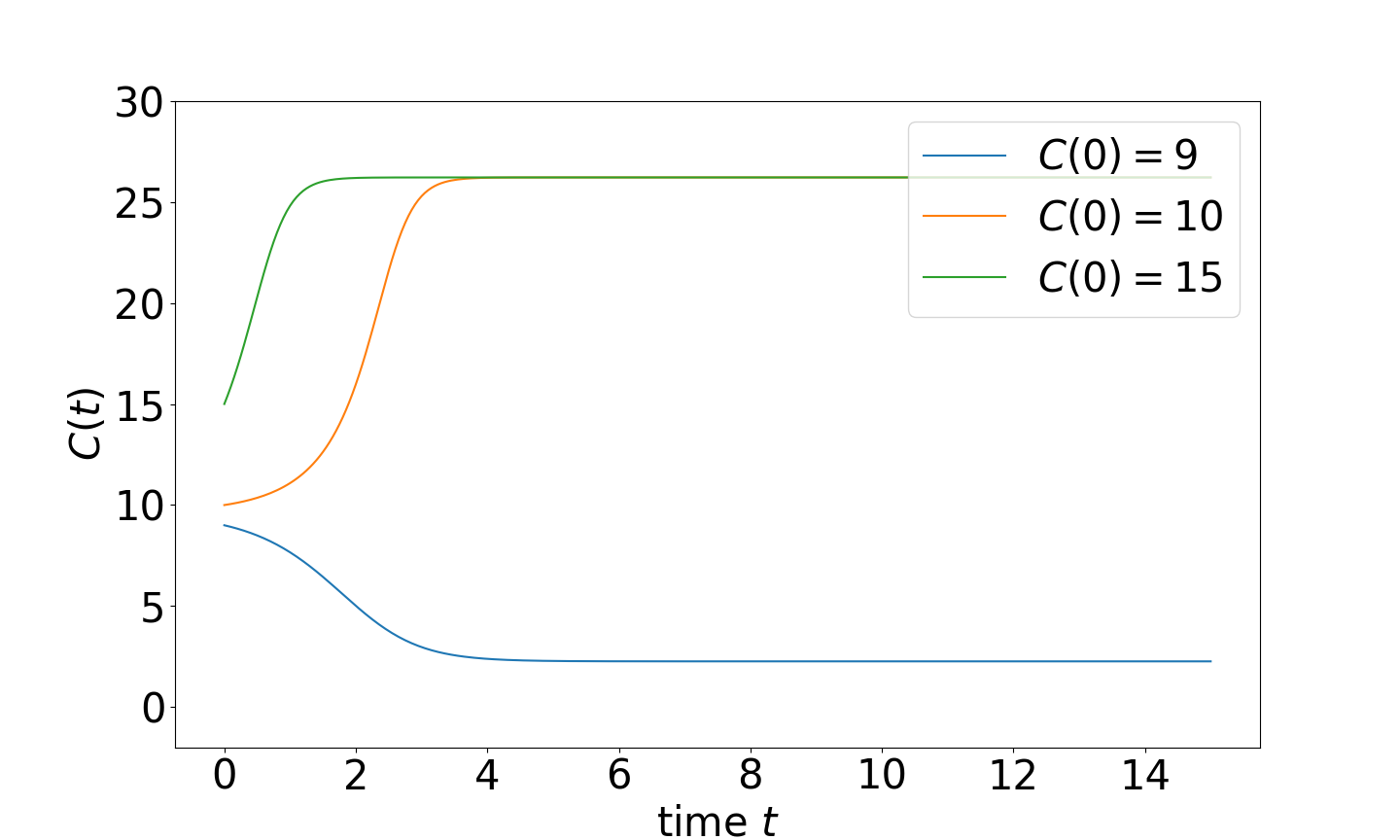}
        \caption{}
    \end{subfigure}
    \hfill
\caption{Solution of the reaction rate equation for (a) the birth-death process, see Eq.~\eqref{RRE:birth-death}, and (b) the Schlögl model, see Eq.~\eqref{RRE:Schloegl}, each for three different initial states $C(0)$. (a) Rate constants $\rate_1=10, \rate_2=1$. (b) Rate constants $\rate_1=6, \rate_2=3.5, \rate_3=0.4, \rate_4=0.0105$. Unstable equilibrium given by $c \approx 9.6201$.}
    \label{fig:ODE}
\end{figure}

\subsection{RDS formulation of the embedded Markov chain}\label{Sec:RDS}

At first, we formulate the setting of a RDS for the embedded Markov chain $(X_n)_{n\geq 0}$, given by $X_n:=X(T_n)$, as a discrete-time stochastic process. 
%For a comprehensive theoretical background of RDS we refer to \cite{arnold1998}.

The noise space $\mathcal{Q}_+$ of the RDS is chosen as
\begin{equation*}
    \mathcal{Q}_+ = \{q=(q_n)_{n \in \mathbb{N}_0} : q_n \in [0,1]\}. 
\end{equation*}
We endow $\mathcal{Q}_+$ with the Borel $\sigma$-algebra $\sigma(\mathcal{Q}_+)$ generated by its cylinder sets, and with the infinite product probability measure $\mathbb{P}=\lambda^{\mathbb{N}_0}$, where $\lambda$ denotes the Lebesgue measure on $[0,1]$. On this probability space $(\mathcal{Q}_+, \sigma(\mathcal{Q}_+), \mathbb{P})$ we define the shift map $\theta: \mathcal{Q}_+ \to \mathcal{Q}_+$ and its iterates by 
\begin{equation}
\label{shiftmap}
\theta(q_0,q_1,...) = (q_1,q_2,...), \qquad \theta^n:=\underbrace{\theta\circ\cdots\circ\theta}_{n \textup{ times }}
\end{equation}
Since $\theta$ is invariant with respect to $\mathbb{P}$, the tuple $(\mathcal{Q}_+,\sigma(\mathcal{Q}_+),\mathbb{P},(\theta)_{n\in\mathbb{N}_0})$ constitutes our underlying noise model.
%a \textit{metric dynamical system} \sw{metric? citation?}, since $\theta$ is $\mathbb{P}$-invariant (or $\mathbb{P}$ is $\theta$-invariant), i.e., $\mathbb{P}(\theta^{-1}(\cdot))=\mathbb{P}$. 
Throughout this work we will use interchangeably the short-hand notation
\begin{equation*}
  \mathbb{P}(S(q)) =  \mathbb{P}(\{q\in \mathcal{Q}_+:S(q) \textup{ holds} \}),
\end{equation*}
where $S(q)$ is a $q$-dependent statement. For any $q\in \mathcal{Q}_+$, consider the transition map $f_q:\mathbb{X} \to \mathbb{X}$ defined by
\begin{equation}\label{fq}
    f_q(x) := x + \nu_{ \kappa(x,q_0)},
\end{equation}
where $f_\cdot$ takes a whole sequence $q=(q_n)_{n \in \mathbb{N}_0}$ as an input but only evaluates the first entry of $q$, namely $q_0$, in $\kappa(x, \cdot)$ defined in \eqref{kappa}.
Therefore $f_{\theta^n q}(X_n)$ coincides with the right-hand side of the recursion \eqref{X_n+1}. Given a fixed order of the reactions, the transition map $f_q$ is unique. 
The RDS of the embedded Markov chain $(X_n)_{n\geq 0}$ is given by the tuple $(\theta,\varphi)$ with the cocycle map $\varphi:\mathbb{N}_0 \times \mathcal{Q}_+ \times \mathbb{X} \to \mathbb{X}$ defined by 
\begin{equation} \label{eq:iterates_cocycles}
    \varphi^n_q(x) 
    = \begin{cases} f_{\theta^{n-1}q}\circ\cdots\circ f_{q}(x)& \textup{if } n \geq 1,
\\ x & \textup{if } n =0.
\end{cases}
\end{equation}
%One can view this as a family of maps depending on $n\in\mathbb{N}$ and $q\in\mathcal{Q}_+$, so that we use the short-hand notation
%\begin{equation}\label{varphi_short}
 %  \varphi^n_{q}(x) := \varphi(n,q,x).  
%\end{equation}
Noting that for any $n\in\mathbb{N}_0$ we have that $\theta^nq = (\theta \circ ... \circ \theta) q =(q_n,q_{n+1}, ...)$ holds for the shift map $\theta$ given in \eqref{shiftmap}, it is straightforward to verify that the cocycle property \eqref{eq:cocycle} holds.
%\begin{equation}
%\label{eq:cocycle_prop}
%    \varphi^{n+m}_q(x)=(\varphi^n_{\theta^m q}\circ\varphi^m_{q})(x)
%\end{equation}
%for any $n,m\in\mathbb{N}$, $q\in\mathcal{Q}_+$, and $x\in\mathbb{X}$.

For each initial state $x \in \mathbb{X}$ and each $q\in \mathcal{Q}_+$ we obtain the orbit of states $(x_n)_{n \in \mathbb{N}_0} =(\varphi_{q}^n(x))_{n \in \mathbb{N}_0}$ from the \textit{random difference equation} 
\begin{equation} \label{orbit_x}
x_{n+1}= f_{\theta^nq}(x_n) ,  \: n \in \mathbb{N}_0 \quad \text{and} \quad x_0 = x \in \mathbb{X}.
\end{equation}
Given an initial state $x$, we have the relation
\begin{equation*}
   \mathbb{P}(X_n \in A|X_0 =x) = \mathbb{P}(\varphi^n_{q}(x) \in A)  
\end{equation*}
for any $A\in \mathcal{B}(\mathbb{X})$.
%Noting that $\theta^nq = (\theta \circ ... \circ \theta) q =(q_n,q_{n+1}, ...)$ holds for the shift map $\theta$ given in \eqref{shiftmap}, we obtain the \textit{cocyle property} in the following form: 
%\begin{equation}
%\label{eq:cocycle_prop}
%    \varphi^{n+m}_q(x)=(\varphi^n_{\theta^m q}\circ\varphi^m_{q})(x)
%\end{equation}
%for any $n,m\in\mathbb{N}$, $q\in\mathcal{Q}_+$, and $x\in\mathbb{X}$.
Note that by virtue of a fixed order of reaction indices assumed for \eqref{kappa} and of the explicit recursion formula~\eqref{X_n+1}, the Markov chain defines our RDS in a unique way; this is generally not the case as there may be different versions of a Markov chain purely characterized by its transition probabilities; for a general discussion see also \cite{Kifer86}.

\subsection{RDS formulation of the augmented Markov chain}\label{Sec:RDS2}

The RDS $(\theta,\varphi)$ introduced before captures only the states $X_n$ entered by the reaction jump process at the jump times $T_n$. In the following, we formulate another RDS which takes account also of the jump times $T_n$ by considering the augmented Markov chain, see \eqref{T_n+1}-\eqref{X_n+1}.
%{\color{purple}***}While the RDS for the embedded Markov chain $(X_n)_{n\in \mathbb{N}_0}$ could be formulated independently of the jump times, a separate RDS formulation for the process $(T_n)_{n\in \mathbb{N}_0}$ is not possible, because the dynamics of $(T_n)_{n\in \mathbb{N}_0}$ depend on the states $(X_n)_{n\in \mathbb{N}_0}$. In fact, we need to consider the augmented Markov chain $(X_n,T_n)_{n\in \mathbb{N}_0}$ and formulate the RDS for this combined process.{\color{purple}***} \gom{Maybe we can remove this paragraph?}   

The state space of the augmented Markov chain is given by $\mathbb{X} \times [0,\infty)$ with the $\sigma$-algebra given by $\mathcal{P}(\mathbb{X}) \otimes \mathcal{B}([0,\infty))$, where $\mathcal{P}(\mathbb{X})$ denotes the power set of $\mathbb{X}$. 
%Regarding the process of jump times $(T_n)_{n \in \mathbb{N}_0}$, we consider the same noise space $(\mathcal{Q}_+, \sigma(\mathcal{Q}_+), \lambda^{\mathbb{N}_0})$ with the same shift map $\theta$ as for the embedded Markov chain. 
In analogy to $f_q$ defined in \eqref{fq}, we consider for any $r \in \mathcal{Q}_+$ the mapping $g_r:\mathbb{X} \times [0,\infty) \to [0,\infty)$ with 
\begin{equation}
\label{g_r}
    g_r(x,t) : = t + \tau(x,r_0)
\end{equation}
for $\tau$ given in \eqref{tau}, such that $T_{n+1} = g_{\theta^n r}(X_n,T_n)$ in \eqref{T_n+1}. However, in contrast to $f_q$, this mapping depends not only on $t \in [0,\infty)$ but also on the state $x \in \mathbb{X}$. 
So, the corresponding cocycle map has to depend on state and time, as well as on $r \in \mathcal{Q}_+$ and $q \in \mathcal{Q}_+$. 
Therefore, we introduce the product noise space 
\begin{equation*}
\Omega_+=\mathcal{Q}_+ \times \mathcal{Q}_+
 = \{\omega=(\omega_n)_{n\in \mathbb{N}_0}: \omega_n =(q_n,r_n), q_n,r_n \in [0,1] \}
\end{equation*}
endowed with the product $\sigma$-algebra $\sigma(\mathcal{Q}_+) \otimes \sigma(\mathcal{Q}_+)$ and the product measure $\mathbb{P}_{\Omega_+} = \lambda^{\mathbb{N}_0}\otimes \lambda^{\mathbb{N}_0}$. By abusing the notation the corresponding shift map $\theta$ acts on both entries of a $\omega \in {\Omega_+}$:
\begin{equation*}
    \theta \omega = \theta (q_n,r_n)_{n\in \mathbb{N}_0} = (q_{n+1},r_{n+1})_{n\in \mathbb{N}_0}.
\end{equation*}
The transformation/time-one mapping for the augmented Markov chain is given by $h_\omega : \mathbb{X} \times [0,\infty)  \to \mathbb{X} \times [0,\infty)$ with
\begin{equation}\label{eq:h}
 h_{\omega}(x,t)  :=  (f_{q}(x),g_{r}(x,t)), 
\end{equation}
where $f_q$ and $g_r$ are given in \eqref{fq} and \eqref{g_r}.
The cocycle map $\psi: \mathbb{N}_0 \times \Omega \times \mathbb{X} \times [0,\infty) \to \mathbb{X} \times [0,\infty)$ is given by 
\begin{equation*}
%\label{eq:augmentedRDS}
       \psi^n_\omega(x,t) = \begin{cases} h_{\theta^{n-1}\omega}\circ\cdots\circ h_{\omega}(x,t)& \textup{if } n \geq 1,
\\ (x,t) & \textup{if } n =0,
\end{cases}
\end{equation*}
and fulfills the cocycle property \eqref{eq:cocycle}. 
We obtain
\begin{equation*}
    \mathbb{P}_{\Omega_+}\Big((X_n,T_n)\in A \;\big\vert\; (X_0,T_0)=(x,t)\Big) = \mathbb{P}_{\Omega_+}(\psi_{\omega}^n(x,t)\in A)
\end{equation*}
for $A \in \mathcal{P}(\mathbb{X})\otimes \mathcal{B}\big([0,\infty)\big)$ and a given initial state $x$ and starting time $t$.

We note that the first component of $\psi_{\omega}^{n}$ coincides with the cocycle map $\varphi_{q}^n$ of the embedded Markov chain, i.e. we have  $(\psi_{\omega}^{n}(x,t))_1= \varphi_{q}^n(x)$ for $\varphi_{q}^n$ given in \eqref{eq:iterates_cocycles} and $q_n=(\omega_n)_1$, while the second component $(\psi_{\omega}^n(x,t))_2$ referring to the time points cannot be considered separately as a cocycle.

\paragraph{Continuous-time process realizations}
By means of the RDS $\psi_{\omega}^{n}$ of the augmented Markov chain, we can introduce a version of the continuous-time  Markov jump process $(X(t))_{t\geq 0}$ starting at time $t_0=0$ in $X(0)=x_0$ by
\begin{equation}\label{eq:Phi}
    \Phi^t_\omega(x_0) :=\varphi^n_q(x_0) \quad \mbox{for} \quad (\psi^{n}_\omega (x_0,0))_2 \leq t < (\psi^{n+1}_\omega (x_0,0))_2.
\end{equation}
% \begin{equation}%\label{eq:Phi}
%     X(t,\omega,x_0) :=\varphi^n_q(x_0) \quad \mbox{for} \quad (\psi^{n}_\omega (x_0,0))_2 \leq t < (\psi^{n+1}_\omega (x_0,0))_2.
% \end{equation}
for $\omega=(q,r)$. This is helpful for illustrating the dynamics: In Fig~\ref{fig:Phi}, common noise realizations of the continuous-time birth-death process given in Example~\ref{ex:birth-death} are depicted for different initial states $x_0\neq y_0$. As we can see, the two realizations seem to approach each other -- with a certain time-delay -- given that the difference $x_0-y_0$ of the initial states is even (see Fig~\ref{fig:Phi}(a)), while this is not the case for an odd difference $x_0-y_0$ (see Fig~\ref{fig:Phi}(b)). This observation motivates to formulate and analyze the synchronization behavior of random dynamical systems for the reaction systems under consideration, which we will do in the following section. 

%
%\begin{rem} \sw{Too long for a remark. Just skip Remark?}
Importantly, we note that $\Phi^t_\omega(x_0)$, as given in~\eqref{eq:Phi}, does not satisfy the cocycle property and, hence, the continuous-time Markov jump process is itself not an RDS in this formulation. The easiest way to observe this is that there are instances of $\Phi^t_\omega(x_0) = \Phi^t_\omega(y_0)$ but $\Phi^{t+s}_\omega(x_0) \neq \Phi^{t+s}_\omega(y_0)$ for some $t,s >0$, $x_0\neq y_0$. However, the information from our RDS analysis of the augmented Markov chain is insightful in terms of characterizing the continuous-time process, as illustrated in Figure~\ref{fig:Phi} and the following results on time-shifted synchronization.
We additionally emphasize that our construction illustrates an intriguing lack of commutativity in the following sense: the Markov jump process admits a version that corresponds to the augmented Markov chain which directly induces an RDS. This RDS can be related back to the original process via~\eqref{eq:Phi} giving a version of the Markov jump process which, however, does not satisfy the cocycle property and is therefore not part of a continuous-time RDS itself. In summary, the RDS structure lies in the space-time version of the reaction rate process, revealing also relevant information about this process as we will see in the following. 
%\sw{I dont understand this last half sentence.}
%\end{rem}

\begin{figure}
    \centering
    \begin{subfigure}[b]{0.49\textwidth}
        \centering
        \includegraphics[width=\textwidth]{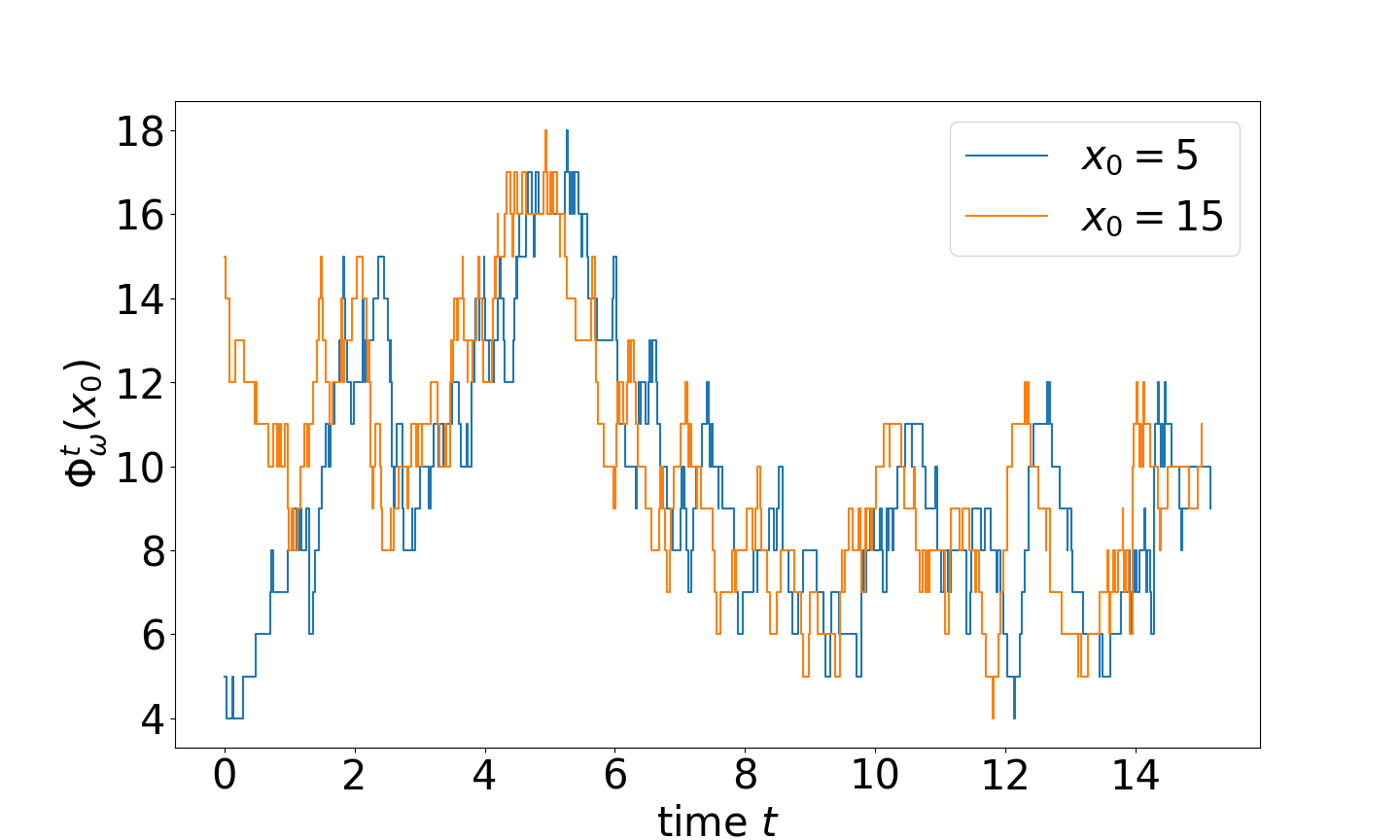}
        \caption{}
    \end{subfigure}
    \hfill
    \begin{subfigure}[b]{0.49\textwidth}
        \centering
        \includegraphics[width=\textwidth]{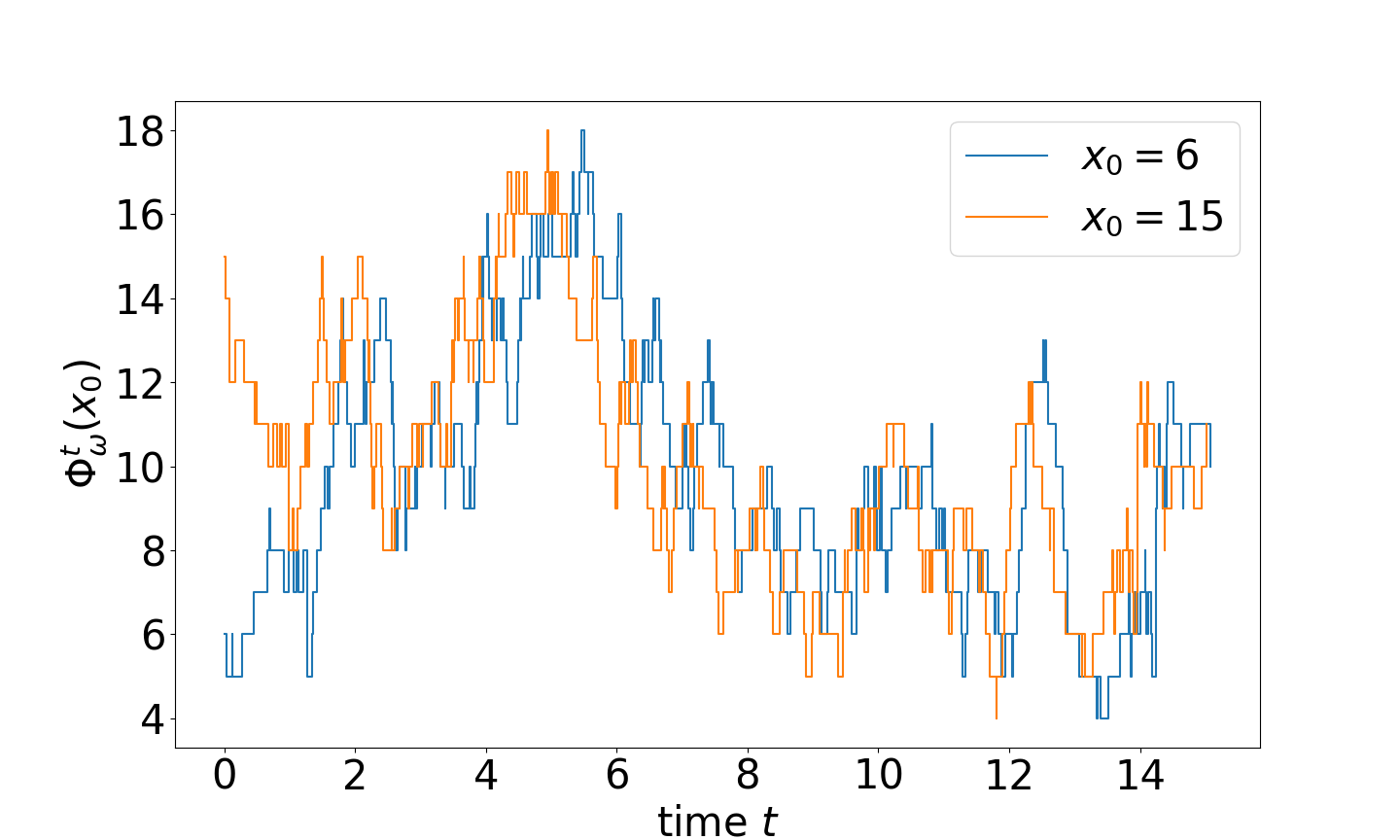}
        \caption{}
    \end{subfigure}
    \hfill
\caption{Continuous-time realizations $\Phi^t_\omega(x_0)$ given in \eqref{eq:Phi} for the birth-death process of Example~\ref{ex:birth-death}, driven by the same noise $\omega$. Realizations for two initial values with (a) even distance and (b) odd distance. In (a), the orange trajectory seems to become a time-delayed copy of the blue one, while this is not the case in (b). The rate constants are chosen as $\rate_1=10$, $\rate_2=1$.  }
    \label{fig:Phi}
\end{figure}

\section{Synchronization of reaction jump processes}
\label{sec:synchro}

In the following, we introduce the terms \textit{synchronization} and 
\textit{partial synchronization} for the random dynamical systems under consideration.
%together with different variantions of it \gom{In the end we didn't really introduce different variations of it, did we?}. 
We analyze the synchronizing properties of the birth-death process given in Example~\ref{ex:birth-death} as well as of the Schlögl model of Example~\ref{ex:schloegl}.

\subsection{General formulation} \label{sec:syn}

Analogously to \cite{huang2019}, we say that an RDS $(\theta,\varphi)$ on $\mathbb{X}=\mathbb{N}_0^L$ is \emph{synchronizing in $S\subset\mathbb{X}$} (or, simply, \emph{synchronizing} when $S=\mathbb{X}$) if 
%we define \textit{synchronization} of the RDS $(\theta, \varphi)$ in the following way:
for every two different initial states $x,y \in S$ and $\mathbb{P}$-a.e. $q \in \mathcal{Q}_+$ there exists a number $n_0\equiv n_0(x,y, q) \in \mathbb{N}_0$ such that
\begin{equation}  \label{synch}
    \varphi_{q}^{n_0}(x) = \varphi_{q}^{n_0}(y).
\end{equation}
It follows from the cocycle property that if \eqref{synch} holds for some $n_0\in\mathbb{N}$, it is true for any other $n\geq n_0$. We say that the RDS $(\theta,\varphi)$ is \textit{partially synchronizing} if there exists a partition $\xi=\{W_0, \dotsc, W_{p-1}\}$ of $\mathbb{X}$ such that %for all $p \in \mathbb{N}$ components $W_i$, Eq.~\eqref{synch} holds for any pair of states $x,y \in W_i$.
$(\theta,\varphi)$ is synchronizing in each $W_i\in\xi$.

For a fixed $q \in \mathcal{Q_+}$ and two different initial states $x,y \in \mathbb{X}$ the process $(\varphi_{q}^n(x), \varphi_{q}^n(y))_{n \in \mathbb{N}_0}$ in the product space $\mathbb{X}^2$ is called the \textit{two-point motion}. 
Let $\Delta$ denote the diagonal in $\mathbb{X}^2$, i.e.
\begin{equation}\label{def:diagonal}
    \Delta:=\{(x, y) \in \mathbb{X}^2 : x= y\}.
\end{equation}
Hence the RDS is synchronizing if and only if the two-point motion reaches the diagonal $\Delta$ at a time index $n_0(x,y, q)$.
%\gom{(I removed the comment "and remaining there" because this is a trivial fact of having a RDS).} %and staying there for all $n \geq n_0(x,y, q)$.

%\gom{(The following remark may be removed if you consider so, but this justifies that one can ask in the definition for the full probability set to depend on $x,y$ or not. Also it may or not depend on the partition).}

\begin{rem}
    \label{RMK:2pointSynch implies Whole}
        Since $\mathbb{X}$ is discrete, in order to show that an RDS $(\theta,\varphi)$ is synchronizing in $S$ it suffices to show that for every two initial states $x,y\in S$ there is a full probability measurable set $\mathcal{Q}_{x,y}$ in which \eqref{synch} holds. Indeed, for each $x,y\in\mathbb{X}$ consider the measurable sets
        \[
          \mathcal{Q}_{x,y}=\bigcup_{n=1}^{\infty}\lbrace q\in\mathcal{Q}_+ : \varphi^n_q(x)=\varphi^n_q(y)\rbrace.
        \]
and assume that $\mathbb{P}(\mathcal{Q}_{x,y})=1$. We can thus take $\mathcal{Q}_{S}=\bigcap_{x,y\in  S}\mathcal{Q}_{x,y}$ and \eqref{synch} holds for every $x,y\in \mathbb{X}$ and $q\in\mathcal{Q}_S$, where $\mathbb{P}(\mathcal{Q}_S)=1$. 

Furthermore, if an RDS is partially synchronizing then for each $W_i\in \xi$ we consider the corresponding sets $\mathcal{Q}_{W_i}$. By considering $\hat{Q}=\bigcap_{i=0}^{p-1}Q_{W_i}$, we can always assume without loss of generality that the set is the same for each element of the partition.
\end{rem}

\paragraph{Time-shifted synchronization}
%\gom{Do we actually need to define it? Maybe rather phenomenological description?}

We observe that the synchronization of the RDS $(\theta,\varphi)$ of the embedded Markov chain directly implies a \textit{time-shifted synchronization} of the RDS $(\theta,\psi)$ of the augmented Markov chain in the following sense.
%We define the \textit{time-shifted synchronization in $S\subset\mathbb{X}$} (or, simply, \emph{time-shifted synchronization} when $S=\mathbb{X}$) for the RDS $(\theta, \psi)$ associated to the augmented Markov chain \eqref{eq:augmentedRDS} as follows:
For every two different initial states $x,y \in S$, an initial time $t \in [0,\infty)$ and $\mathbb{P}_{\Omega_+}$-a.e. $\omega \in \Omega$ there exists a number $n_0:=n_0(x,y,t, \omega) \in \mathbb{N}_0$ and a value $R:=R(x,y, t, \omega)\geq 0$ such that
\begin{equation} 
\label{time-shifted_synch_1}
\begin{cases}
    \varphi_{q}^n(x)=\varphi_{q}^n(y)
    \\ \big|(\psi_{\omega}^n(x,t))_2 -(\psi_{\omega}^n(y,t))_2\big|=R
\end{cases}    
 \text{for all } n \geq n_0.
\end{equation}
This means that from a certain time point, the states $X_n$ of two realizations of the augmented Markov chain actually coincide, while for the jump times only the differences $T_{n+1}-T_n$ become the same.  Note that we are considering two different initial states $x,y \in S \subset \mathbb{X}$, but start with both at the same initial time $t \in [0, \infty)$.  Usually this initial time is $t=0$. If \eqref{time-shifted_synch_1} holds for all $x,y$ in each component of a partition of $\mathbb{X}$ we analogously speak of \textit{partial time-shifted synchronization}. 

%Furthermore, we define \textit{partial time-shifted synchronization} if \eqref{time-shifted_synch_1} holds for every $x$ and $y$ in each component of a partition of $\mathbb{X}$.

The fact that the (partial) time-shifted synchronization of the RDS $(\theta,\psi)$ follows from the (partial) synchronization of the RDS $(\theta,\varphi)$ of the related embedded Markov chain is because the first component $f_q(x)$ of the time-one mapping $h_\omega$ is independent of $t$, see \eqref{eq:h}. 
%We observe that in our particular setting the time-shifted synchronization  of the RDS $(\theta,\psi)$, referring to the augmented Markov chain, directly follows from the synchronization of the RDS $(\theta,\varphi)$ of the embedded Markov chain. This is due to the fact that the first component of the time-one mapping $h_\omega$ is independent of $t$, see \eqref{eq:h}. The same is true for partial synchronization of the RDS $(\theta,\varphi)$, which implies partial time-shifted synchronization of the RDS $(\theta,\psi)$.   
As a consequence, for analyzing the synchronization properties of the augmented Markov chain (and with it the synchronization properties of the continuous-time reaction jump process) it suffices to consider the corresponding embedded Markov chain. 

We proceed by analyzing the synchronization properties for the special case of the birth-death process. 

\subsection{Synchronization of the birth-death process} \label{sec:birth-death_syn}

We consider the embedded Markov chain of the birth-death process defined in Example~\ref{ex:birth-death}, which for simplicity we refer to it simply as \textit{birth-death chain} in the following\footnote{One should bear in mind that in many references (see e.g. \cite{durrett2019probability}) a birth-death chain is a general Markov chain in $\mathbb{N}_0$ where the only possible transitions from $x\in\mathbb{N}$ are $x+1$ or $x-1$. In this sense the embedded Markov chain of the Schlögl model is also a birth-death chain but we will distinguish the two by this choice of terminology.}. 
Its transition map $f_q:\mathbb{N}_0 \to \mathbb{N}_0$ can be chosen as
\begin{equation*}
    f_q(x)= \begin{cases} x+1 \quad \text{if } q_0< \frac{\rate_1}{\rate_1+\rate_2 x},
    \\x-1 \quad \text{otherwise,}
    \end{cases}
\end{equation*}
for $q \in \mathcal{Q}_+$, which directly follows from Eq.~\eqref{eq:kappa}. 
Recall that the choice of $f_q$ is not unique as long as the order of reaction indices it not fixed in \eqref{kappa}. A reordering of the reactions would lead to $f_q(x) = x-1$ if $q_0 <\frac{\rate_2 x}{\rate_1+\rate_2 x} $ and $f_q(x) = x+1$ otherwise, defining an RDS with the same statistics and, by symmetry, also the same topological properties that we will study throughout the rest of the paper.

Our goal is to show that the RDS of the birth-death chain partially synchronizes. 
For this purpose, 
%analyze the partial synchronization, 
we consider the two-point motion $(\varphi_{q}^n(x_0), \varphi_{q}^n(y_0))_{n \in \mathbb{N}_0}$ of the birth-death chain, which is depicted in  Figure~\ref{fig:two-point motion} for two different pairs of initial states.
As a first step, we will make clear through the transition probabilities of the two-point motion that the \textit{thick diagonal}
\begin{equation}\label{D}
	 \mathbb{D}:=\big\{ (x,y)\in \mathbb{N}_0^2 : y\in\lbrace x-1,x,x+1\rbrace \big\}
\end{equation}
is \textit{forward invariant} for the two-point motion, \textit{i.e.} $(\varphi^n_q(x),\varphi^n_q(y))\in \mathbb{D}$ for all $n\in\mathbb{N}$ if $(x,y)\in \mathbb{D}$; see Figure~\ref{fig:BDP_abstract} for an illustration. We prove partial synchronization for the birth-death chain via Lemma~\ref{Lemma_I}, Proposition~\ref{PROP:synchronisation}, and Corollary~\ref{COR:partial synchronization}.
%{\color{purple}***} In this example, the partition of $\mathbb{N}_0$ is given by $\xi=\{W_0, W_1\}$ with $W_0=\{0,2,4,\ldots\}$ and $W_1=\{1,3,5, \ldots\}$, as we will see in the following. {\color{purple}***} \gom{We could remove this}

\begin{figure}
    \centering
       \begin{subfigure}[b]{0.49\textwidth}
        \centering
        \includegraphics[width=\textwidth]{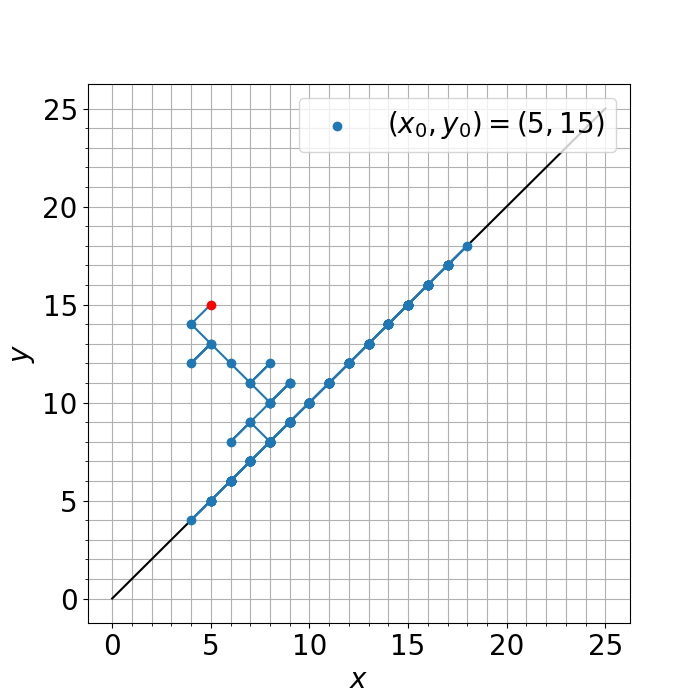}
        \caption{}
    \end{subfigure}
    \hfill
     \begin{subfigure}[b]{0.49\textwidth}
        \centering
        \includegraphics[width=\textwidth]{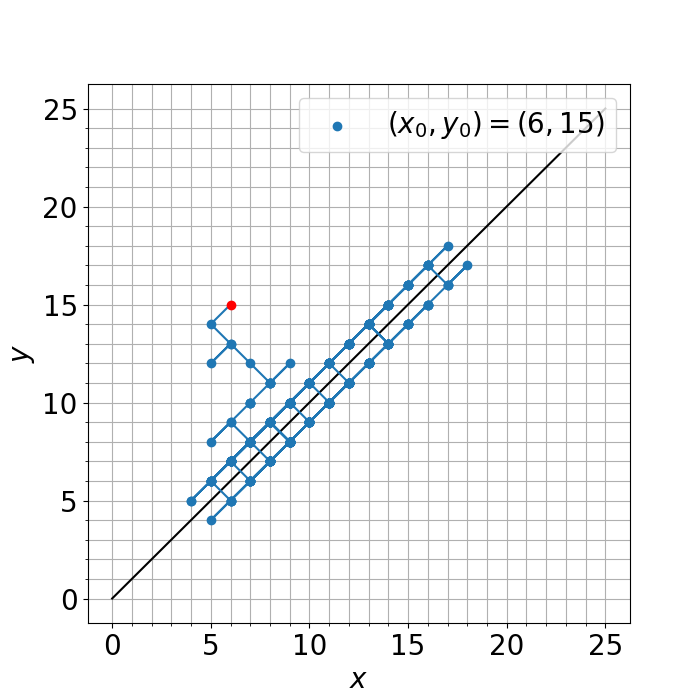}
        \caption{}
    \end{subfigure}
    \hfill
    \caption{Two-point motion $(\varphi_{q}^n(x_0), \varphi_{q}^n(y_0))_{n \in \mathbb{N}_0}$ of the birth-death process. The trajectory in (a) belongs to the realization shown in Figure~\ref{fig:Phi}(a), while (b) refers to Figure~\ref{fig:Phi}(b). The rate constants are chosen as $\rate_1=10$, $\rate_2=1$.
    }
    \label{fig:two-point motion}
\end{figure}

%(\gom{I rephrased the paragraph so that it was more consistent. I also changed $\delta$ below for $z$, just because $\delta$ sounds to me like a "small" number.})

%As a first step, we show that the two-point motion almost surely moves towards the diagonal $\Delta$ defined in \eqref{def:diagonal} and gets close to it in finite time. By \textit{close} we mean that it reaches the \textit{thick diagonal} $ \mathbb{D}$ given by
%\begin{equation}\label{D}
%	 \mathbb{D}:=\big\{ (x,y)\in \mathbb{N}_0^2 : y\in\lbrace x-1,x,x+1\rbrace \big\},
%\end{equation}

%As we will see in the following, the set $\mathbb{D}$ is absorbing (or forward invariant in terms of dynamical system theory). 

\paragraph{Transition probabilities}

Let $z$ be a variable which takes the  values $\nu_1=1$ or $\nu_2=-1$ of the corresponding state-change vectors. For the transition probabilities of the RDS of the embedded Markov chain we set
\begin{equation}
\label{P_delta}
   P_{z} (x) :=  \mathbb{P}\Big(\varphi_{q}^{n+1}(x_0)=x+z \;\Big|\;\varphi_{q}^n(x_0)=x\Big)
\end{equation}
for an arbitrary state $x\in \mathbb{N}_0$ and an initial state $x_0\in \mathbb{N}$. This probability is independent of $n$ because the process is time-homogeneous. 
Using again Eq.~\eqref{eq:kappa}, we obtain
\begin{align}
\label{eq:trans_prob}
    P_1(x) = \frac{\rate_1}{\rate_1+\rate_2x}, \quad P_{-1}(x) = \frac{\rate_2x}{\rate_1+\rate_2x}.
\end{align}

Just as for the one-point motion $(\varphi_q^n(x_0))_{n \in \mathbb{N}_0}$, we can also determine the transition probabilities for the Markovian dynamics of the two-point motion $(\varphi_{q}^n(x_0), \varphi_{q}^n(y_0))_{n \in \mathbb{N}_0}$. 
For $z_1,z_2\in \{1,-1\}$ set
\begin{align*}
%\label{P_delta,delta}
   P_{(z_1,z_2)} (x,y) :=  \mathbb{P}\Big(\big(&\varphi_{q}^{n+1}(x_0),\varphi_{q}^{n+1}(y_0)\big) 
   =(x+z_1,y+z_2) \nonumber\\
   &\Big|\big(\varphi_{q}^n(x_0), \varphi_{q}^n(y_0)\big)=(x,y)\Big)
\end{align*}
for any $(x,y) \in \mathbb{N}_0^2$. 
Given \eqref{eq:trans_prob} we can deduce that the transition probabilities are 
\begin{flalign}
   & \hspace{2cm} P_{(1,1)} (x,y)  = \min\{P_1(x),P_1(y)\},  & \llap{(yellow)} \label{transition_prob_1} \\
   &  \hspace{2cm}   P_{(-1,1)} (x,y)  =\max\{0,P_1(y)-P_1(x)\}, &\llap{(red)} \\
   & \hspace{2cm}   P_{(1,-1)} (x,y)  =\max\{0,P_1(x)-P_1(y)\}, &\llap{(blue)} \\
    & \hspace{2cm} P_{(-1,-1)} (x,y)  = 1-\max\{P_1(x),P_1(y)\}.  & \llap{(green)} \label{transition_prob_4}
\end{flalign}
The colors refer to the transitions given by the arrows in Figure~\ref{fig:BDP_abstract}.
Since we have $P_{(-1,1)}(x,y)=0$ for $x<y$ and $P_{(1,-1)}(x,y)=0$ for $x>y$, it follows inmediately that the set $\mathbb{D}$ is forward-invariant for the two-point motion. 

\paragraph{First hitting time of $\mathbb{D}$}

We will show that the absorbing set $\mathbb{D}$ is reached by the two-point motion $(\varphi_{q}^n(x_0), \varphi_{q}^n(y_0))_{n \in \mathbb{N}_0}$ of the birth-death process almost surely in finite time regardless of the starting point. Formally speaking, by considering
\[ \tau_{\mathbb{D}}(x_0,y_0,q):=\inf\lbrace n\geq 0: (\varphi_{q}^n(x_0), \varphi_{q}^n(y_0)) \in\mathbb{D} \rbrace\]
as the first hitting time of the thick diagonal $\mathbb{D}$, we show that 
\begin{equation*}
\mathbb{P}\big(\tau_\mathbb{D} (x_0,y_0,q) <\infty \big) =1
\end{equation*}
holds for all $x_0,y_0\in \mathbb{N}_0$. 
To do so, we first define for a given  $d\in\mathbb{Z}$
%\sw{$d$ will be the metric... Maybe no problem} 
the level set 
\begin{equation}\label{I_d}
	I_d := \left\{(x,y)\in \mathbb{N}_0^2:x-y=d \right\}
\end{equation}
and show in the following lemma that for $d\neq 0$ a process starting in $I_d$ will almost surely leave this set in finite time. 
%A schematic illustration of the sets $\mathbb{D}$ and $I_d$ is given in Figure~\ref{fig:BDP_abstract}. 

\begin{figure}
    \centering
    \begin{subfigure}[b]{0.49\textwidth}
        \centering
        \includegraphics[width=\textwidth]{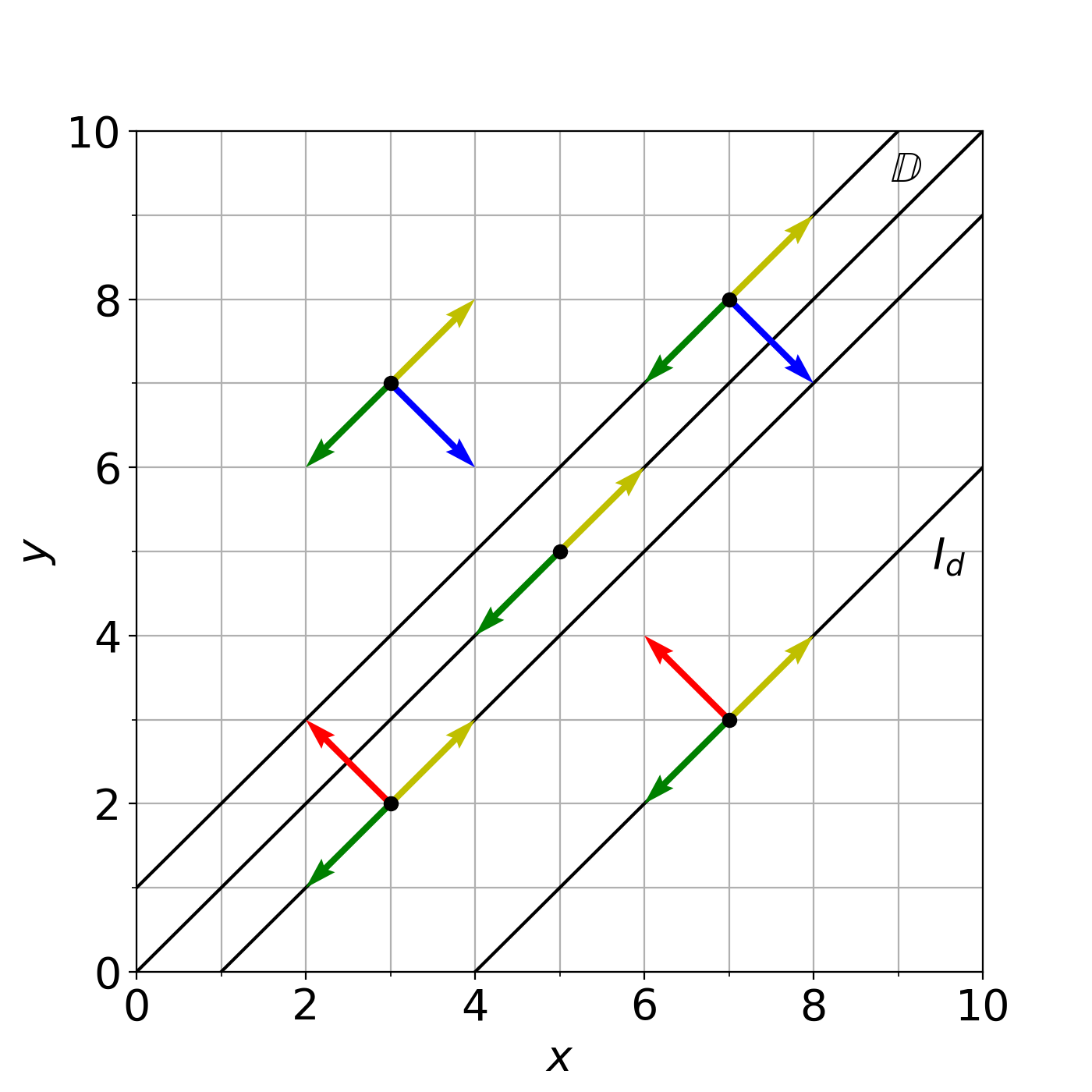}
    \end{subfigure}
    \hfill
    \caption{Schematic illustration for the transitions of the two-point motion for the birth-death process. The thick diagonal $\mathbb{D}$ is defined in \eqref{D}, the level sets $I_d$ are defined in \eqref{I_d}. The arrows indicate the directions in which the two-point motion can move, with the colors indicating the values of the corresponding probabilities as given in \eqref{transition_prob_1}-\eqref{transition_prob_4}.}
    \label{fig:BDP_abstract}
\end{figure}

\begin{lemma}\label{Lemma_I}
  Let $d\in\mathbb{Z}\setminus\lbrace 0\rbrace$ be given.	Then, for each initial state $(x_0,y_0)\in I_d$ the two-point motion $(\varphi_{q}^n(x_0), \varphi_{q}^n(y_0))_{n \in \mathbb{N}_0}$ of the birth-death chain exits $I_d$ $\mathbb{P}$-a.s., i.e.,
	\[ \mathbb{P}\big((\varphi_{q}^n(x_0), \varphi_{q}^n(y_0))\in I_d \; \forall n \geq 0 \big) =0 \quad \text{for all } (x_0,y_0) \in I_d.\]
\end{lemma}

\begin{proof}
Without loss of generality we only consider the case in which $d\geq 1$. For $x \in \mathbb{N}_0$ consider the state $(x+d,x)\in I_d$ and define 
\[ p_x := \mathbb{P}\big((\varphi_{q}^n(x+d), \varphi_{q}^n(x))\in I_d \; \forall n \geq 0 \big)\]	
as the probability for the two-point motion to stay forever on $I_d$ given that it starts in $(x+d,x)\in I_d$. 
By means of the law of total probability we have
\begin{equation*}
%\label{eq:total_probability}
	p_x = P_{(1,1)}(x+d,x)\cdot p_{x+1} + P_{(-1,-1)}(x+d,x) \cdot p_{x-1}
\end{equation*}
for $x\geq 1$, where $P_{(1,1)}(x+d,x)=\frac{1}{1+\alpha (x+d)}$ and $P_{(-1,-1)}(x+d,x)=\frac{\alpha x}{1+\alpha x}$ with $\alpha:=\frac{\rate_2}{\rate_1}$, see \eqref{transition_prob_1} and \eqref{transition_prob_4}. 
From this we can deduce the second-order difference equation
\begin{equation}\label{eq:recursion}
	p_{x+2} = (1+\alpha(x+d+1)) \left(p_{x+1} - \frac{\alpha (x+1)}{1+\alpha (x+1)} p_x \right)
\end{equation} 
for $x\geq 0$. Moreover, for $x=0$ we have $p_0=P_{(1,1)}(d,0)\cdot p_{1}$ such that 
\begin{equation}\label{eq:p0}
 p_1=(1+\alpha d)p_0.   
\end{equation}

Since $p_1$ is proportional to $p_0$, it follows inductively from Eq.~\eqref{eq:recursion} that $p_x$ is proportional to $p_0$ for all $x\in\mathbb{N}_0$. We prove by contradiction that the sequence $(p_x)_{x \in \mathbb{N}_0}$ of probabilities has to fulfill $p_x=0$ for all $x \in \mathbb{N}_0$.

Indeed, assume that $p_0>0$. 
From \eqref{eq:p0} we obtain
\begin{equation}
    \label{eq:initial_deviation}
    p_1-p_0=\left(1-\frac{1}{1+\alpha d}\right)p_1=\alpha d p_0>0.
\end{equation}
On the other hand, it follows immediately from \eqref{eq:recursion} that
    \[
        p_{x+2}-p_{x+1}=\alpha(x+d+1)p_{x+1}-\frac{\alpha(x+1)\left[ 1+\alpha(x+d+1)\right]}{1+\alpha (x+1)}p_x.
    \]
    Hence, by adding and subtracting $\frac{\alpha(x+1)\left[ 1+\alpha(x+d+1)\right]}{1+\alpha (x+1)}p_{x+1}$ we obtain
    \begin{align*}
        p_{x+2}-p_{x+1} & =\frac{\alpha d}{1+\alpha(x+1)}\cdot p_{x+1}+ \frac{\alpha(x+1)\left[ 1+\alpha(x+d+1)\right]}{1+\alpha (x+1)}\cdot (p_{x+1}-p_x)
        \\
        & \geq \alpha(x+1)\left(1+\frac{\alpha d}{1+\alpha(x+1)}\right)(p_{x+1}-p_x).
    \end{align*}
Let $u_{x}:=p_{x+1}-p_x$. Then, from the last inequality it follows that
\[
u_{x+1}>\alpha(x+1)u_x.
\]
By iterating the above inequality it follows that
\[
    u_x>\alpha^{x-1}\cdot x!\; u_0 \qquad x\in\mathbb{N}.
\]
Since $u_0>0$, see \eqref{eq:initial_deviation}, and since $\alpha^{x-1}\cdot x!\rightarrow \infty$ as $x\rightarrow\infty$ for any $\alpha>0$, this implies that $u_x\rightarrow\infty$. It then follows that $p_x\rightarrow\infty$, which is a contradiction since $p_x\in[0,1]$. In conclusion, we obtain $p_0=0$, and by proportionality, $p_x=0$ for all $x$. 
Noticing that all states in $I_d$ are of the form $(x+d,x)$ for some $x\in \mathbb{N}_0$ completes the proof. 
\end{proof}

By means of Lemma \ref{Lemma_I} we can now make the following central statement. 

\begin{prop}
\label{PROP:synchronisation}
	For each $(x,y)\in\mathbb{N}_0^2$, the two-point motion $(\varphi_{q}^n(x_0), \varphi_{q}^n(y_0))_{n \in \mathbb{N}_0}$ of the birth-death chain reaches the thick diagonal $\mathbb{D}$ almost surely in finite time.
\end{prop}

\begin{proof}
Let $(x_0,y_0)\in I_d$ for a given $d=x_0-y_0\neq 0$. According to Lemma~\ref{Lemma_I}, the two-point motion $(\varphi_{q}^n(x_0), \varphi_{q}^n(y_0))_{n \in \mathbb{N}_0}$ almost surely escapes from $I_d$ in finite time. Given the transition probabilities \eqref{transition_prob_1}-\eqref{transition_prob_4}, it can only end up in $I_{d-2}$ when $d\geq 1$ or in $I_{d+2}$ when $d\leq -1$. %for which the same argument applies. 
This happens a finite number of times until the process reaches $\mathbb{D}=I_{-1}\cup I_0 \cup I_1$. %The same is true for $x_0<y_0$ with $d=x_0-y_0<0$, because the whole system is symmetric and therefore Lemma~\ref{Lemma_I} also applies to $I_d$ with $d<0$. In this case of negative $d$, the two-point motion can move from $I_d$ to $I_{d+2}$, see again Figure~\ref{fig:BDP_abstract}. 
%As the set $\mathbb{D}$ is absorbing (forward invariant), the process will never leave it again. 
\end{proof}

\begin{cor}
\label{COR:partial synchronization}
The RDS $(\theta, \varphi)$ for the embedded Markov chain of the birth-death process partially synchronizes, and the partition is given by $\xi=\{W_0, W_1\}$ with $W_0=\{0,2,4,\ldots\}$ and $W_1=\{1,3,5, \ldots\}$.
\end{cor}

\begin{proof}
%\gom{(I found the proof here a little repetitive so I try to do it simpler here in colour)}

Let $W_0,W_1$ be as in the statement above. As observed before in Proposition~\ref{PROP:synchronisation}, for any fixed $(x,y)\in I_d$ with $d\neq 0$ the two-point motion escapes in finite time to $I_{d-2}$ if $d\geq 1$, or to $I_{d+2}$ when $d\leq -1$. Thus, it reaches $\Delta=I_0$ in finite time $\mathbb{P}$-a.s. if and only if $d$ is even, that is if $x,y\in W_i$ for some $i\in\lbrace 0,1\rbrace$. The result follows from Proposition~\ref{PROP:synchronisation} and Remark~\ref{RMK:2pointSynch implies Whole}.
% According to the transition probabilities given in  \eqref{transition_prob_1}-\eqref{transition_prob_4}, transitions of the two-point motion between the level sets $I_d$ are sorely of the form $I_d \to I_{d-2}$ for $d>0$,
% %the two-point motion will a.s. reach $I_{d-2}$ ones it escapes from $I_d$ for $d>0$, 
%  which means that the distance $|\varphi_{q}^n(x_0)- \varphi_{q}^n(y_0)|$ modulo $2$ does not change with $n$, $|\varphi_{q}^n(x_0), \varphi_{q}^n(y_0)|\bmod 2 = $ const. The diagonal $\Delta = I_0$, see \eqref{def:diagonal}, can thus only be reached when starting in $I_d$ for $d\in 2\mathbb{Z}$, i.e., for initial states $(x_0,y_0) \in \mathbb{N}_0^2$ with $|x_0-y_0|=2\mathbb{N}_0$. Using this insight and applying Proposition \ref{PROP:synchronisation}, we can follow that for the partition $\xi=\{W_0, W_1\}$ with $W_0=\{0,2,4,\ldots\}$ and $W_1=\{1,3,5, \ldots\}$ the two-point motion will reach the diagonal $\Delta=I_0$ a.s. in finite time when starting from $(x_0,y_0)\in W_i^2$ and will never leave it again. This is equivalent to partial synchronization as defined in Sec. \ref{sec:syn}. 
 \end{proof}

 Corollary~\ref{COR:partial synchronization} states that whenever $x,y\in W_i$ ($i=0,1$) we have for almost all $q \in \mathcal{Q}_+$ that $\#\varphi^n_{q}(\lbrace x,y\rbrace)=1$ for all $n$ sufficiently large, where $\#A$ denotes the cardinality of a set $A$. 
 More generally, for each finite (deterministic) set $K\subset W_i$ we obtain that for almost all $q \in \mathcal{Q}_+$ there is a $n_0(K, q)\in \mathbb{N}$ such that $\# \varphi^n_q(K)=1$ for all $n \geq n_0(K, q)$. This almost sure convergence implies the convergence in probability given by
    \begin{equation}
    \label{eq:detconvprob}
        \lim_{n\rightarrow \infty}\mathbb{P}(\# \varphi^n_q(K)\geq 2)=0.
    \end{equation}
%\me{In the following definition, $d(x, K_q)$ is not defined. Note that we introduce the distances later, see \eqref{eq:distance}.} \sw{I have put the definition of the metric here. }
This last statement can be extended to finite random sets as defined in the following Definition~\ref{DEF:randomset}. For this, let $d:\mathbb{X}\times \mathbb{X}\to [0,\infty)$ be the Euclidean distance on $\mathbb{X}$ and define
\begin{equation}\label{def:d(x,B)}
    d(x,B):= \inf_{y \in B}d(x,y)
\end{equation}
for non-empty sets $B\subset\mathbb{X}$.

\begin{Def}
\label{DEF:randomset}
Let $(\mathcal{Q},\sigma(\mathcal{Q}),\mathbb{P})$ be an arbitrary probability space and $\mathbb{X}=\mathbb{N}_0^L$. A mapping $K:\mathcal{Q}\rightarrow \mathcal{P}(\mathbb{X})$, denoted as $q \mapsto K_q$, is a \textit{random set} if the function $q \mapsto d(x, K_q)$ is measurable for each $x \in \mathbb{X}$.
\end{Def}

We say that a random set $K:\mathcal{Q}\rightarrow \mathcal{P}(\mathbb{X})$ is a \textit{finite random set} if $K_q$ is nonempty and finite for every $q\in\mathcal{Q}$. A finite random set in $\mathcal{P}(\mathbb{X})$ is contained in a deterministic finite set with high probability as indicated in the next proposition. %\me{This is a famous Lemma, also to be found in my lecture notes, and I believe it can also be found in Crauel's book on Random Probability measures on Polish Spaces. You can leave the proof if you want but please contextualize.} \sw{Better give a citation. We write a paper, not a textbook...  We can state the Proposition in order to be able to refer to it later, but for the proof we should only refer to the literature.}
\begin{prop}
    \label{PROP:randomfiniteInsideDeterministic}
        Let $K:\mathcal{Q}\rightarrow\mathcal{P}(\mathbb{X})$ be a finite random set. Then, for each $\varepsilon>0$ there is a finite set $F_\varepsilon\subset\mathcal{P}(\mathbb{X})$ such that
        \[
            \mathbb{P}\left( K_q\subset F_\varepsilon \right)\geq 1-\varepsilon.
        \]
\end{prop}
\begin{proof}
The statement is a particular case of a more general setting, see \cite[Proposition 3.15]{CrauelBook02}. 
%    Since $K_q$ is finite if and only if there is $m\in\mathbb{N}$ such that $K_q\subset F_m:=\{0,1,\ldots,m \}^L$, we consider the sets $\mathcal{Q}_m:=\{ q\in\mathcal{Q} : K_q\subset F_m\}$. Since for every $m$ we have that $\mathcal{Q}_m\subset \mathcal{Q}_{m+1}$, we get that
%    \[
%        1=\mathbb{P}(K_q \textup{ is finite})=\mathbb{P}\left( \bigcup_{m=1}^{\infty} \mathcal{Q}_m\right)=\lim_{m\rightarrow\infty}\mathbb{P}(\mathcal{Q}_m).
%    \]
%Consequently, given $\varepsilon>0$ there is $m_\varepsilon\in\mathbb{N}_0$ such that
%\[
%    \mathbb{P}\left(K_q\subset F_{m_\varepsilon}\right)=\mathbb{P}(\mathcal{Q}_m)\geq 1-\varepsilon.
%\]
%The result follows by taking $F_\varepsilon\equiv F_{m_\varepsilon}$.
\end{proof}

We can now generalize property \eqref{eq:detconvprob} for the birth-death chain for arbitrary random finite sets in the next Proposition.

%\me{This is a long proof for a corollary given the short proofs of all the statements before. Maybe rather call it a proposition itself?} \sw{I changed it to proposition.}

\begin{figure}
    \centering
    \begin{subfigure}[b]{0.69\textwidth}
        \centering
        \includegraphics[width=\textwidth]{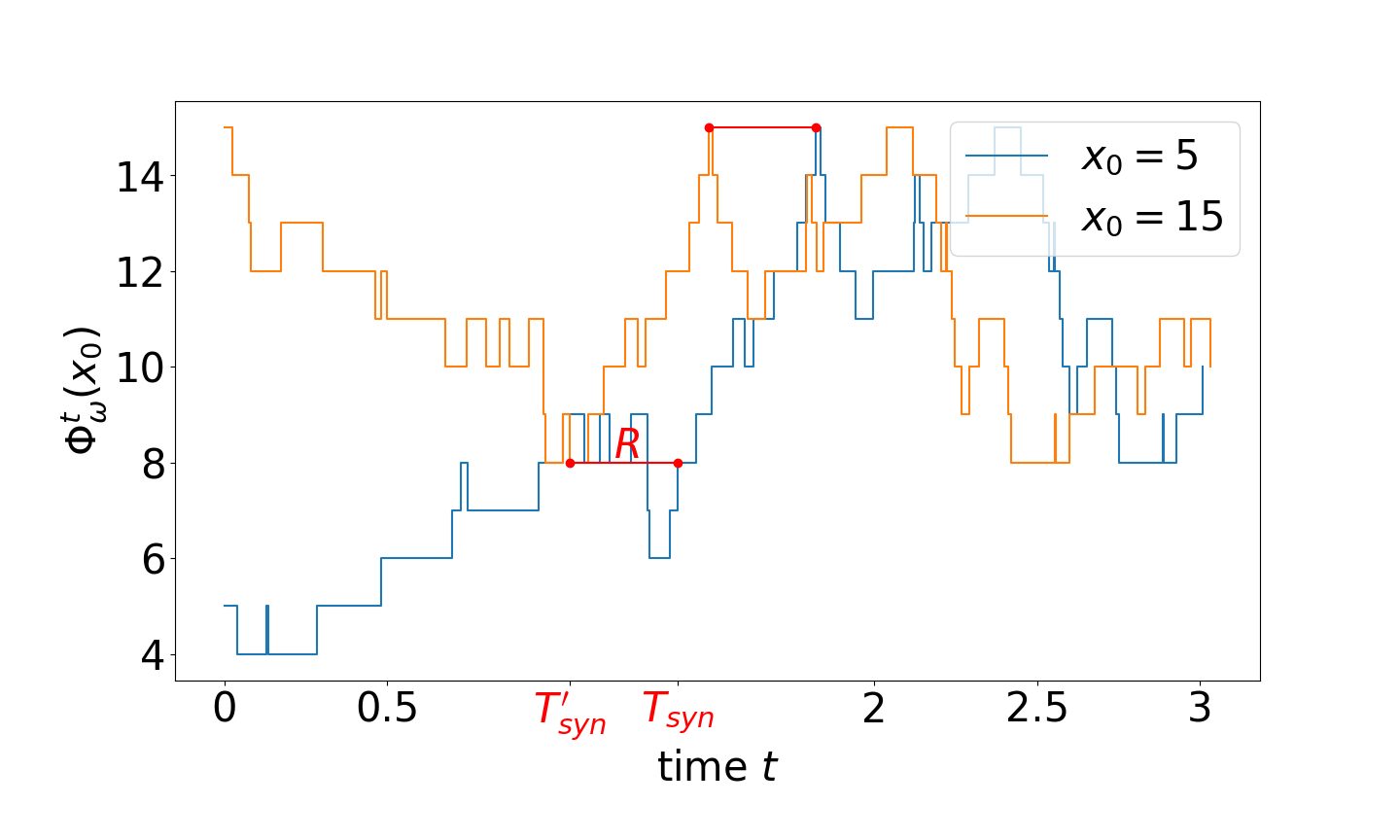}
    \end{subfigure}
    \hfill
    \caption{Time-shifted synchronization for the birth-death process. Extract of Figure~\ref{fig:Phi}(a), showing two realizations of $\Phi_\omega^t(x_0)$ for initial values with even distance. The time-shifted synchronization starts at time index $n_0=19$, where both realizations reach the state $x=8$. For the process starting in $x_0=5$ this happens at time $T_{\text{syn}}=(\psi_\omega^{n_0}(5,0))_2\approx 1.395$, while for the other realization the time point is given by $T'_{\text{syn}}=(\psi_\omega^{n_0}(15,0))_2\approx 1.064$. I.e., the time-shift is $R=T_{\text{syn}}-T'_{\text{syn}}\approx 0.331$ for these initial states.}
    \label{fig:timeshifted}
\end{figure}

\begin{prop}
\label{Prop:randomConvProb}
    Consider the setting of the birth-death chain. Let $K:\mathcal{Q}_+\rightarrow \mathcal{P}(\mathbb{N}_0)$ be a random finite set such that $K_q\subset W_i$ $\mathbb{P}$-a.s. for some $i\in\lbrace 0,1\rbrace$. Then 
    $\#\varphi^n_q(K_q)\rightarrow 1$ in probability.
\end{prop}

\begin{proof}
Since $K_q$ is nonempty $\mathbb{P}$-a.s., the set $\varphi^n_q(K_q)$ has at least one element for all $n\in\mathbb{N}$ $\mathbb{P}$-a.s..
On the other hand, let $\varepsilon>0$ be arbitrarily small and $F_\varepsilon\subset \mathbb{N}_0$ a finite set as in Proposition~\ref{PROP:randomfiniteInsideDeterministic}. Then, 
\[
    \mathbb{P}(\#\varphi^n_q(K_q)\geq 2) = 
    \mathbb{P}\left(\#\varphi^n_q(K_q)\geq 2 \ \land K_q\subset F_\varepsilon\right) + 
    \mathbb{P}\left(\#\varphi^n_q(K_q)\geq 2 \ \land K_q\not\subset F_\varepsilon\right).
\]
Note that for any $n\in\mathbb{N}_0$,
\[
    \lbrace q\in\mathcal{Q}_+ : \# \varphi^n_q(K_q)\geq 2 \ \land \  K_q\subset F_\varepsilon\rbrace 
       \subset 
    \lbrace q\in\mathcal{Q}_+ : \# \varphi^n_q(F_\varepsilon)\geq 2\rbrace,
\]
and thus we deduce together with Proposition~\ref{PROP:randomfiniteInsideDeterministic} that 
\[
    \mathbb{P}(\# \varphi^n_q(K_q)\geq 2) \leq \mathbb{P}(\# \varphi^n_q(F_\varepsilon)\geq 2) +\varepsilon.
\]
Due to \eqref{eq:detconvprob} and since $\varepsilon$ was arbitrarily small we conclude that
\[
    \lim_{n\rightarrow\infty}\mathbb{P}(\#\varphi^n_q(K_q)\geq 2)=0, 
\]
and the result follows.
\end{proof}

% \gom{I conjecture that the convergence $\mathbb{P}$-a.s. should hold as well, but this is not needed afterwards.}

% \Nathalie{A few words about the time-shifted synchronization of the birth-death process, refer to Figure \ref{fig:timeshifted}. Maybe some statistics about $R$ and $n_0$ or $(\psi_\omega^{n_0}(x,t))_2$!}

As noted in the end of Sec.~\ref{sec:syn}, the partial synchronization of the RDS $(\theta, \varphi)$ for the embedded Markov chain directly implies the partial time-shifted synchronization of the corresponding augmented Markov process, which can easily be seen from Eq.~\eqref{T_n+1}. Thus, the observations from Figure~\ref{fig:Phi} can now be confirmed/clarified: In Figure~\ref{fig:Phi}(a), we have $x_0\in W_0$ for both initial states, such that time-shifted synchronization as defined in  \eqref{time-shifted_synch_1} is guarantied by Corollary~\ref{COR:partial synchronization}, see Figure~\ref{fig:timeshifted} for a detailed look at the dynamics.  In contrast, the initial states chosen in Figure~\ref{fig:Phi}(b) are not in the same set $W_i$, and consequently, the trajectories do not synchronize. However, they are likely to stay close to each other because the corresponding two-point motion reaches the thick diagonal where the distance between states is not larger than one.

%\sw{Include an example where (full) synchronization is given? Maybe, we just need to add a reaction like $2\mathcal{S} \to \emptyset$ to the birth-death process? Then, not only transitions from even to odd or from odd to even are possible, but also from even to even or from odd to odd. }

\subsection{Synchronization for the Schlögl model}

In this section, we consider Example~\ref{ex:schloegl} as a variation of the birth-death process with additional bistable structure. 
We analyze the RDS $(\theta, \varphi)$ for the embedded Markov chain of the according reaction jump process. With the notation as in \eqref{P_delta}, the transition probabilities of the system are 
\begin{equation}
\label{eq:Schloegl_P}
    P_1(x) =  \frac{\rate_1+\rate_3x(x-1)}{\mu(x)}, \quad
    P_{-1}(x) =  \frac{\rate_2x+\rate_4 x(x-1)(x-2)}{\mu(x)}
\end{equation}
for an arbitrary state $x \geq 0$. %\gom{as a matter of fact, this makes sense for $x\geq 0$ and they give the formulas below}, where $\mu(x):=\rate_1+\rate_2x+\rate_3x(x-1)+\rate_4 x(x-1)(x-2)$. For the special cases $x=0,1$ we have $P_1(0)=1, P_{-1}(0)=0$ and $P_1(1)= \frac{\rate_1}{\rate_1+\rate_2},P_{-1}(1)=\frac{\rate_2}{\rate_1+\rate_2}$.
In Figure~\ref{fig:Schloegl_RJP}, two continuous-time realizations of the dynamics are shown. As in the birth-death process, one can observe a time-shifted synchronization when starting with even distance, see Figure~\ref{fig:Schloegl_RJP}(a). For an odd distance in the starting points, see Figure~\ref{fig:Schloegl_RJP}(b), the separation of the trajectories is even more significant than in the birth-death scenario.  

\begin{figure}
    \centering
    \begin{subfigure}[b]{0.49\textwidth}
        \centering
        \includegraphics[width=\textwidth]{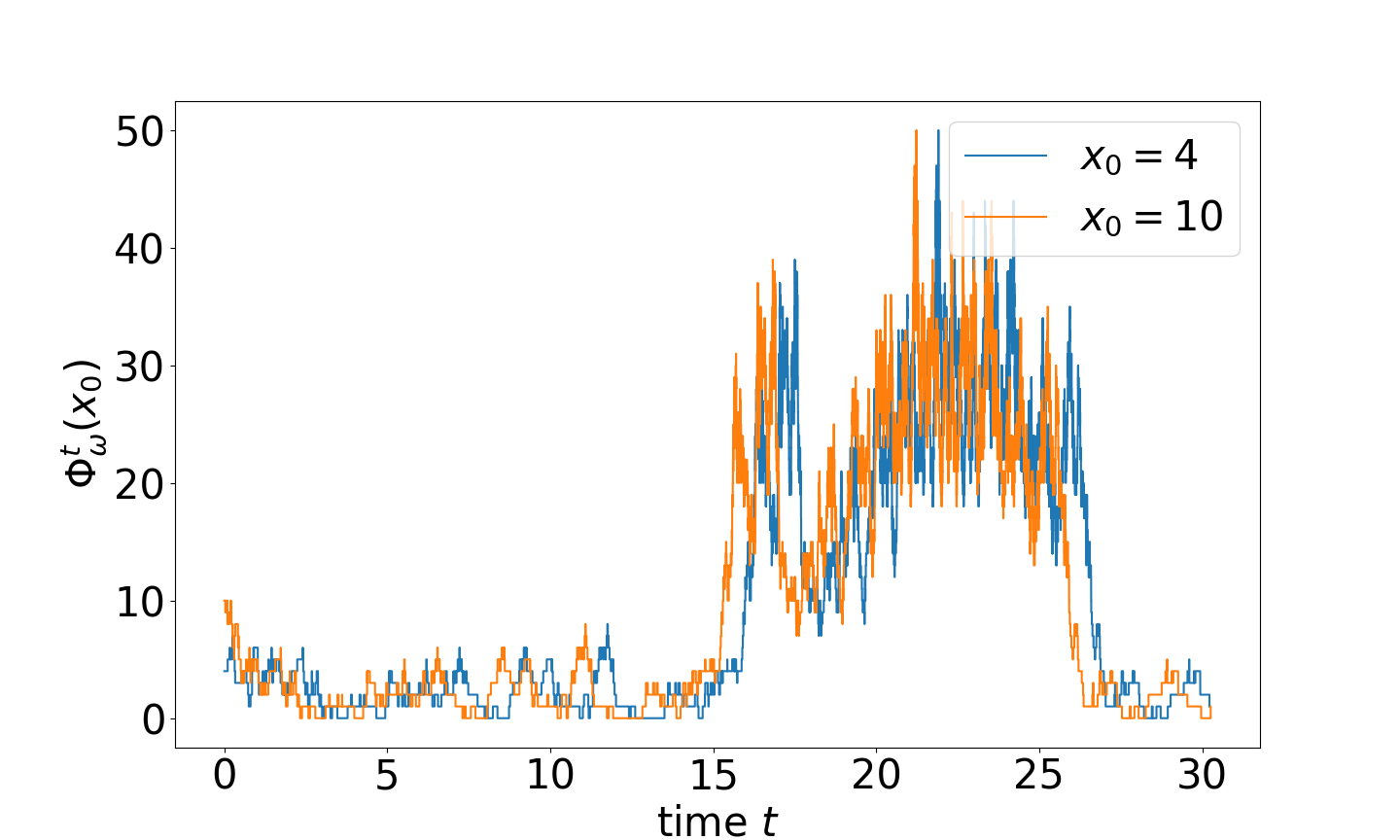}
        \caption{}
    \end{subfigure}
    \hfill
    \begin{subfigure}[b]{0.49\textwidth}
        \centering
        \includegraphics[width=\textwidth]{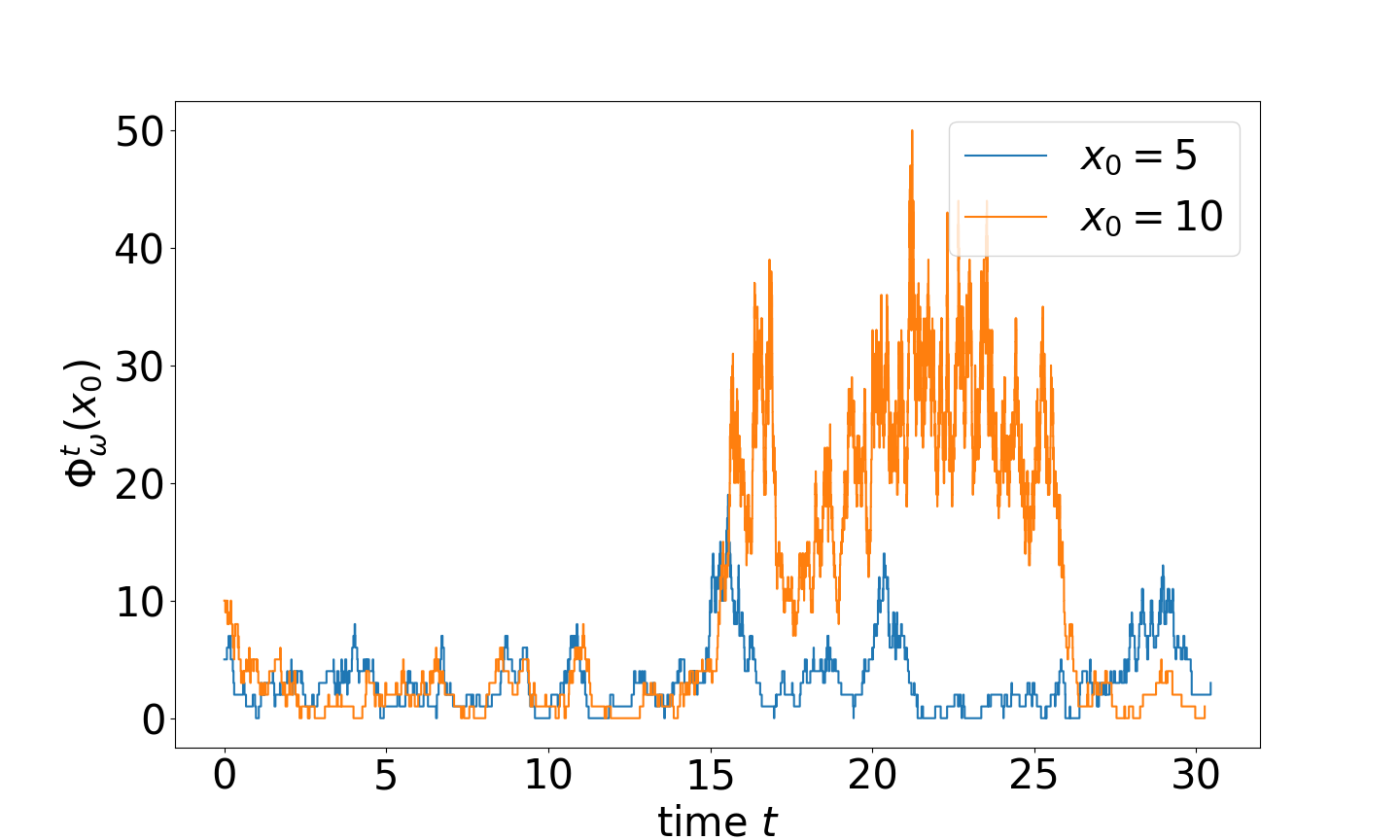}
        \caption{}
    \end{subfigure}
    \hfill
    \caption{Continuous-time realizations $\Phi^t_\omega(x_0)$ given in \eqref{eq:Phi} for the Schlögl model of Example~\ref{ex:schloegl}, driven by the same noise. Realizations for two initial values with (a) even distance and (b) odd distance. In (a), the orange trajectory becomes a time-delayed copy of the blue one, just as in Figure~\ref{fig:Phi}, while in (b), the trajectories clearly separate. The rate constants are chosen as $\rate_1=6$, $\rate_2=3.5, \rate_3=0.4, \rate_4=0.0105$.}
    \label{fig:Schloegl_RJP}
\end{figure}

The transition probabilities of the corresponding two-point motion are analogous to the ones for the birth-death process given in \eqref{transition_prob_1} - \eqref{transition_prob_4}, however with $P_1$ and $P_{-1}$ defined in \eqref{eq:Schloegl_P}. 
Figure~\ref{fig:Schloegl_2dim} shows two realizations of the two-point motion for different initial states. 
As opposed to the birth-death chain, the thick diagonal $\mathbb{D}$ is not absorbing in the Schl\"ogl model. This can be seen from the transition probabilities of the two-point motion which are non-zero also for directions pointing away from the thick diagonal (see Figure~\ref{fig:Schloegl_quiver}).

\begin{figure}
    \centering
    \begin{subfigure}[b]{0.49\textwidth}
        \centering
        \includegraphics[width=\textwidth]{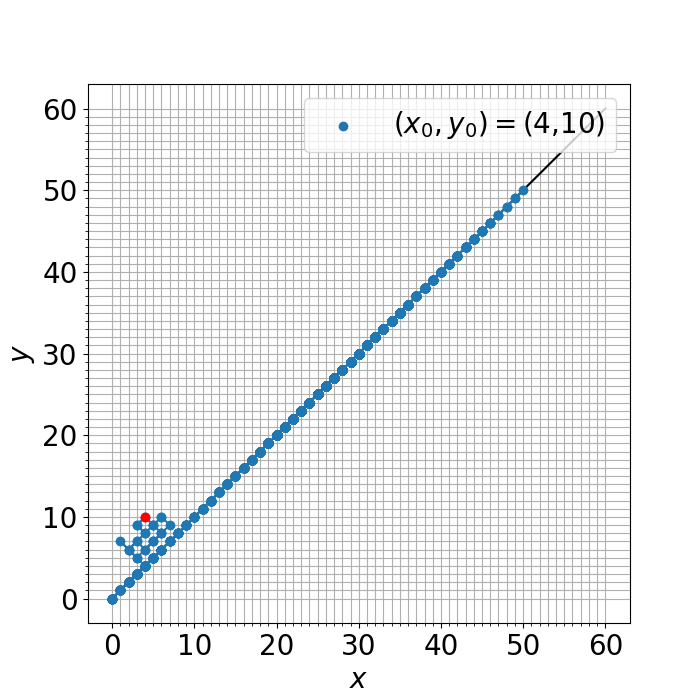}
        \caption{}
    \end{subfigure}
    \hfill
    \begin{subfigure}[b]{0.49\textwidth}
        \centering
        \includegraphics[width=\textwidth]{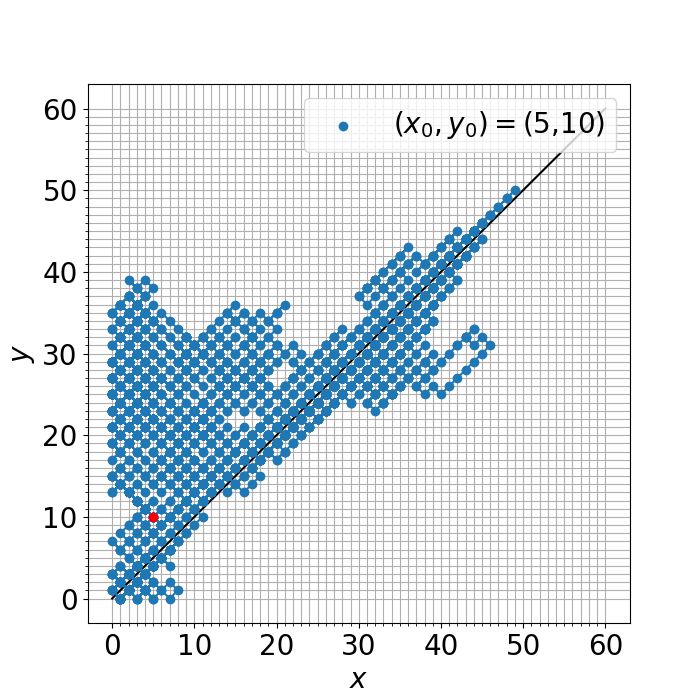}
        \caption{}
    \end{subfigure}
    \hfill
    \caption{Two-point motion $(\varphi_{q}^n(x_0), \varphi_{q}^n(y_0))_{n \in \mathbb{N}_0}$ of the Schlögl model. The trajectory in (a) belongs to the realization shown in Figure~\ref{fig:Schloegl_RJP}(a), while (b) refers to Figure~\ref{fig:Schloegl_RJP}(b). %$\rate_1=6$, $\rate_2=3.5, \rate_3=0.4, \rate_4=0.0105$.
    }
    \label{fig:Schloegl_2dim}
\end{figure}

A proof of the conjectured partial synchronization is much more involved than in the simple birth-death case, due to the described lack of monotonicity in the two-point dynamics towards the diagonal. Hence, the proof strategy of Proposition~\ref{PROP:synchronisation} and the following corollaries does not apply here.  We leave a suitably generalized proof with different means as an open problem for the future. 
From such a partial synchronization of the embedded Markov chain, one may conclude directly the time-shifted synchronization of the continuous-time realization $(\Phi^t_\omega(x))_{t\geq 0}$ for initial states with even distance as observed in Figure~\ref{fig:Schloegl_RJP}(a). 

%Again, the (conjectured) partial synchronization for the embedded Markov chain induces a time-shifted synchronization for the continuous-time realizations $\Phi^t_\omega(x), \Phi^t_\omega(y)$ if $x,y$ have even distance, and no such structure if this distance is odd; see Figure~\ref{fig:Schloegl_RJP} for a numerical exploration of the time-shifted sychnronization behaviour of the reaction jump process.

\begin{figure}
    \centering
    \begin{subfigure}[b]{0.49\textwidth}
        \centering
        \includegraphics[width=\textwidth]{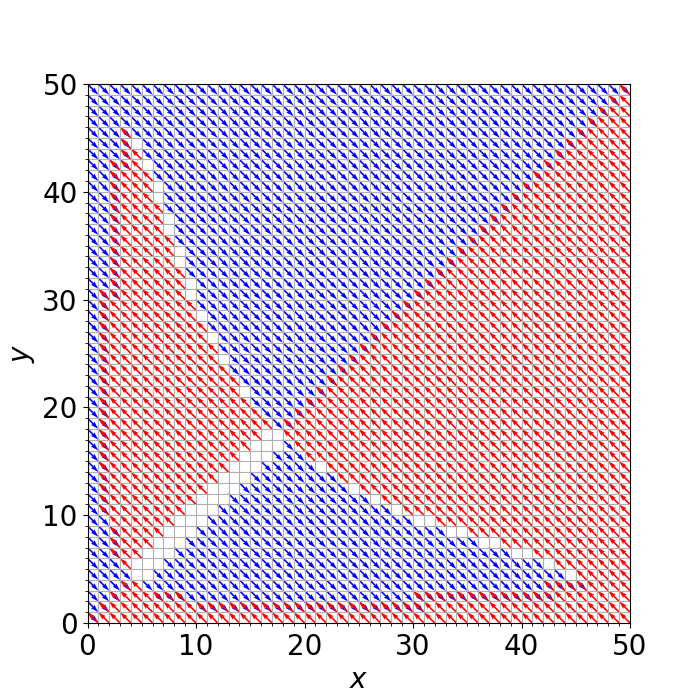}
    \end{subfigure}
    \hfill
    \caption{Possible transitions to the top left (red) and the bottom right (blue) for the Schl\"ogl model. Transitions parallel to the diagonal are not displayed. %$\rate_1=6$, $\rate_2=3.5, \rate_3=0.4, \rate_4=0.0105$. 
    }
    \label{fig:Schloegl_quiver}
\end{figure}

\section{Random attractors associated to (embedded) Markov chains}
\label{sec:attractors}
In this section we introduce different notions of random attractors for an RDS, providing insights into their relationships in the context of a discrete state  in Sec.~\ref{sec:attractors_gen}. As a matter of fact, we show in Sec.~\ref{sec:existence_weak_attr} that under very mild conditions an RDS induced by the random difference equation \eqref{orbit_x} admits a weak attractor. The partial synchronization for the embedded Markov chain of the birth-death process, which was proved in Section 3, will be useful to describe in full detail the structure of its attractor and the related sample measures in Sec.~\ref{sec:birth-death_attr}. We explore numerically such characteristics for the attractor of the embedded Markov chain of the Schlögl model in Sec.~\ref{sec:Schloegl}.
%\sw{Say shortly what will happen in the following subsections.}

% \Nathalie{How is synchronization (defined via almost sure convergence) in general related to the existence of random attractors? In the weak sense holds: existence of a weak attractor which is a random fixed point implies "weak" synchronization (defined via convergence in probability), see \cite{flandoli2016}. Is there something similar in the strong sense? Maybe I have asked that before..}
% \me{In fact, this is a little different: in \cite{flandoli2016}, synchronization is defined as the weak set attractor being a singleton and weak synchronizaton as the weak point attractor being a singleton. For us this difference is irrelevant as the set and point attractor are the same: furthermore, we have partial sycnhronization and the random attractor is not a singleton but a random periodic orbit.}
% \sw{Max/Guillermo: Please write a short intro on how this section is related to the one before. We have synchronization, now we analyze...}

\subsection{General properties of random attractors in discrete time and discrete space}
\label{sec:attractors_gen}

%\sw{Is this section for general discrete-time, discrete-space RDS? We need to make clear which results are true for discrete-time, discrete-space RDS in general and which only for those induced by reaction systems and why. Difference: for reaction systems the state space has a special structure ($\mathbb{N}_0^L$) and only finitely many transitions are possible (because we have a finite set of reactions)? Where exactly do we use this?}

In this and the following section (Sec.~\ref{sec:existence_weak_attr})
we consider a generalized setting as compared to Sec.~\ref{Sec:RDS}, since our insights on random attractors for the discrete case can be easily formulated in that generality.
More concretely, let $\mathcal{Q}$ be a probability space with measure $\mathbb P$ and let $\theta:  \mathcal{Q} \to \mathcal Q$ be an invertible and $\mathbb P$-ergodic map. Furthermore, for $\mathbb{X}=\mathbb{N}_0^L$ as before, let $\varphi: \mathbb N_0 \times \mathcal Q \times \mathbb X \to \mathbb X$ be a measurable cocycle (cf.~property~\eqref{eq:cocycle}) such that $(\theta, \varphi)$ is a random dynamical system. One of the characteristics of random attractors is that of \textit{invariance} as described in the following.

\begin{Def}
\label{DEF:invariance}
    Given a RDS $(\theta,\varphi)$, we say that a random set $A:\mathcal{Q}\rightarrow\mathcal{P}(\mathbb{X})$ (recall Definition~\ref{DEF:randomset}) is invariant (or $\varphi$-invariant) if
    \begin{equation}
    \label{eq:invariance}
        \varphi_{q}^n(A_q) = A_{\theta^n q}\qquad \textup{ for all } n\in\mathbb{N}, \ \mathbb{P}\textup{-a.s.}
    \end{equation}
\end{Def}

% The concept of a random attractor makes sense only when considering an invertible dynamical system on $\mathcal{Q_+}$. However, as we defined it so far (see Sec.~\ref{Sec:RDS}), the shift map is not invertible since for $q_0\neq \tilde{q}_0$ we have that
% \[
%     \theta(q_0,q_1,q_2,\ldots)=\theta(\tilde{q}_0,q_1,q_2)=(q_1,q_2,\ldots).
% \]
% We can come around this inconvenience by redefining our noise space as
% \[\mathcal{Q}:=\lbrace q=(q_n)_{n\in\mathbb{Z}} : q_n\in[0,1] \rbrace,\]
% endowed with its Borel $\sigma$-algebra $\sigma(\mathcal{Q})$ and the bi-infinite product measure $\mathbb{P}=\lambda^{\mathbb{Z}}$. We redefine the shift map as 
% \begin{equation} \label{shift_map}
%   \theta q=\theta (q_n)_{n\in\mathbb{Z}} =(q_{n+1})_{n\in \mathbb{Z}}
% \end{equation}
% such that $(\theta^{-n}q)_i = q_{i-n}$, 
% while the cocycle map $\varphi$ remains the same, see \eqref{varphi_short}.

Notice that without loss of generality, we can assume that \eqref{eq:invariance} is satisfied everywhere by restricting ourselves to the full probability set where this is true.
\begin{rem}
\label{REMARK:1_impliesAllN}
    Since we deal with a discrete-time setting, condition \eqref{eq:invariance} is fulfilled as soon as it is satisfied for $n=1$, that is $\varphi_q^1(A_q)=A_{\theta q}$ $\mathbb{P}$-a.s. Indeed, let $\mathcal{Q}_1\subset \mathcal{Q}$ be a measurable set of full probability such that \eqref{eq:invariance} holds for $n=1$. By the invariance of $\mathbb{P}$ under $\theta$, we have that $\mathbb{P}(  \mathcal{Q}_1)=\mathbb{P}(\theta^{-1}\mathcal{Q}_1)=1$. Consider then the full probability set $\mathcal{Q}_2=\theta^{-1}\mathcal{Q}_1\cap \mathcal{Q}_1$, which is the set where \eqref{eq:invariance} holds for $n=1,2$, and inductively for $k\in\mathbb{N}$ the set $\mathcal{Q}_{k+1}=\theta^{-1}\mathcal{Q}_k\cap \mathcal{Q}_k$ for which \eqref{eq:invariance} is satisfied for all $n\leq k+1$. Hence, condition \eqref{eq:invariance} holds on $\tilde{\mathcal{Q}}:=\bigcap_{n=1}^{\infty}\mathcal{Q}_n$, and $\mathbb{P}(\tilde{\mathcal{Q}})=1$.
\end{rem}

%Since $\mathbb{X}=\mathbb{N}_0^L$, ompact sets are simply finite sets. 

%\gom{We give the notion of attraction in terms of a metric.} %Let $d:\mathbb{X}\times \mathbb{X}\to [0,\infty)$ be the Euclidean distance on $\mathbb{X}$ and define $d(x,B):= \inf_{y \in B}d(x,y)$ for non-empty sets $B\subset\mathbb{X}$, as well as the Hausdorff semi-distance for non-empty sets
We give the notion of attraction in terms of the Hausdorff semi-distance for non-empty sets
\begin{equation*}% \label{eq:distance}
    \dist(A,B)= \sup_{x\in A} d(x,B), \quad A,B\subset \mathbb{X},
\end{equation*}
where $d(x,B)$ is given in \eqref{def:d(x,B)}. 
In general, an invariant compact random set $A$ (i.e., $A_q$ is compact for all $q\in\mathcal{Q}$) is called %(cf.~e.g.~\cite{crauel2015})
\begin{itemize}
    \item[(i)] a \textit{(strong) forward attractor}, if for each compact set $B\subset \mathbb{X}$
    \begin{equation}\label{def:forward}
       \lim_{n \to \infty}  \dist(\varphi_{q}^n(B), A_{\theta^n q})=0 \quad \mathbb{P}\text{-a.s.}, 
    \end{equation}
    \item[(ii)]  a \textit{(strong) pullback attractor} if for each compact set $B\subset \mathbb{X}$
    \begin{equation*}%\label{def:pullback}
        \lim_{n \to \infty}  \dist(\varphi_{\theta^{-n}q}^n(B), A_ q)=0 \quad \mathbb{P}\text{-a.s.},
    \end{equation*}
    \item[(iii)]  a \textit{weak attractor} if for each compact set $B\subset \mathbb{X}$
    \begin{equation}\label{def:weak}
       \lim_{n \to \infty}  \dist(\varphi_{q}^n(B), A_{\theta^n q})=0 \ \text{ in probability,} 
    \end{equation}
\end{itemize}
cf.~e.g.~\cite{crauel2015}. 
By the discrete nature of the state space, compact sets in $\mathbb{X}$ are simply finite sets. Moreover, for any $q\in \mathcal{Q}$ and $B \subset \mathbb X$,
 \begin{equation}
    \label{eq:limitIFFlanding}
       \lim_{n\to \infty} \dist\left( \varphi^n_{\theta^{-n}q}(B),A_q \right)= 0  \ \text{ iff } \ \exists N\equiv N(q,B) \text{ s.t. } \forall n\geq N, \ \varphi^{n}_{\theta^{-n}q}(B)\subset A_q.
    \end{equation}
In a similar fashion,  $\lim_{n\rightarrow\infty}\dist(\varphi^n_q(B),A_{\theta^n q})=0$ if and only if there exists $N\equiv N(q,B)$ such that $\varphi^N_{q}(B)\subset A_{\theta^N q}$. Here, the invariance of $A$ guarantees that $\varphi^n_{q}(B)\subset A_{\theta^n q}$ for all $n\geq N$.

It is known that weak attractors are unique, in the sense that two weak attractors $A$ and $\tilde{A}$ coincide $\mathbb{P}$-a.s. \cite[Lemma 1.3]{flandoli2016}. Observe also that in the weak sense ``forward" and ``pullback" attraction are the same since $\theta$ is measure-preserving. In other words, $A$ is a weak attractor if and only if for each compact set $B\subset \mathbb{X}$
$$ \lim_{n \to \infty}  \dist(\varphi_{\theta^{-n}q}^n(B), A_q)=0 \ \text{ in probability.}$$

Since the state space is discrete it is straightforward to see that if the strong pullback convergence holds for all point sets, i.e. for $B=\lbrace x\rbrace$ where $x\in\mathbb{X}$, then it also holds for any finite $B$. Using the terminology from \cite{crauel2015,CrauelFlandoli1994}, this would imply that strong \textit{point attractors} are equivalent to strong (set) attractors as defined above. 
%Indeed, because of \eqref{eq:limitIFFlanding} for any finite set of points $K$, then \eqref{eq:limitIFFlanding} is satisfied for $B=K$ $\mathbb{P}$-a.s.~by taking $N(q,K)=\max\lbrace N(q,x) : x\in K \rbrace$ 
Indeed, assuming that \eqref{eq:limitIFFlanding} holds for all point sets $B=\{x\}$, where $x\in K$ for a given finite set $K$, it also $\mathbb{P}$-a.s.~holds for $B=K$ by taking $N(q,K)=\max\lbrace N(q,x) : x\in K \rbrace$. 
In a similar spirit, we give equivalent conditions for an invariant random finite set to be a weak attractor in the next theorem.  

\begin{Thm}
    \label{THM:weakEquiv}
     Let $A:\mathcal{Q}\rightarrow\mathcal{P}(\mathbb{X})$ be an invariant compact finite set. Then the following are equivalent.
     \begin{itemize}
         \item[I.] $A$ is a weak attractor.
         \item[II.] $A$ is a forward attractor. 
         \item[III.] $A$ is a forward point attractor, that is \eqref{def:forward} is satisfied for all $B=\lbrace x\rbrace$ with $x\in \mathbb{X}$.

         \item[IV.] $A$ is a \emph{weak point attractor}, that is \eqref{def:weak} is satisfied for all $B=\lbrace x\rbrace$ with $x\in \mathbb{X}$.
         \item[V.] $A$ weakly attracts random finite sets, that is for any finite random set $K:\mathcal{Q}\rightarrow\mathcal{P}(\mathbb{X})$ we have that
         \[
            \lim_{n\rightarrow\infty}\dist(\varphi^n_q(K_q),A_{\theta^n q})=0 \ \text{ in probability.}
         \]
             \end{itemize}
\end{Thm}

%\sw{Notation: What about $\mathcal{A}$ or $\mathbb{A}$ or $\boldsymbol{A}$ for the attractor in order to distinguish from the letter $A$ in the italic style of the theorem?} \gom{Would it help to change the bullets for normal numbers?}

\begin{proof}
\noindent \textbf{I implies II.}  Assume that $A$ is a weak attractor. Take $x\in \mathbb{X}$ as an arbitrary initial condition and consider the sets $\mathcal{Q}_n^x:=\lbrace q\in\mathcal{Q} :  \varphi^n_{q}(x)\in A_{\theta^{n}q}  \rbrace$.  Since the state space is discrete, we have 
\[ \mathbb{P}(\mathcal{Q}_n^x) = 1-\mathbb{P}\left( d\left( \varphi^n_{q}(x),A_{\theta^{n}q} \right)\geq r \right) \]
%\[
%    \mathbb{P}\left( d\left( \varphi^n_{q}(x),A_{\theta^{n}q} \right)\geq r \right)= \mathbb{P}\left(  \varphi^n_{q}(x)\not\in A_{\theta^{n}q} \right)
%\]
for any $r\in(0,1]$. %Consider the sets
%\[
%    \Omega_n^x:=\lbrace q\in\mathcal{Q} :  \varphi^n_{q}(x)\in A_{\theta^{n}q}  \rbrace.
%\]
From the convergence in probability to the attractor (cf.~\eqref{def:weak}) it follows that for an arbitrary $\varepsilon>0$ there exists $N\equiv N(x,\varepsilon)$ such that $\mathbb{P}(\mathcal{Q}_n^x)\geq 1-\varepsilon$ for all $n\geq N$. Furthermore, from the $\varphi$-invariance of $A$ we know that $\mathcal{Q}_n^x\subset \mathcal{Q}_{n+1}^x$ for all $n\in\mathbb{N}$. 

Let $\mathcal{Q}^x:=\bigcup_{n\in\mathbb{N}}\mathcal{Q}_n^x=\{q\in \mathcal{Q}: \lim_{n\to \infty} d\left( \varphi^n_{q}(x),A_{\theta^{n}q}\right)=0\}$. Notice that for any $n\geq N$
\[
    \mathbb{P}(\mathcal{Q}^x)\geq \mathbb{P}(\mathcal{Q}_n^x)\geq 1-\varepsilon,
\]
and since $\varepsilon$ was arbitrarily small we have that $\mathbb{P}(\mathcal{Q}^x)=1$.
% Finally, let
% \[
% \hat{\mathcal{Q}}:=\bigcap_{x\in\mathbb{X}} \mathcal{Q}^x = \left\lbrace q\in \mathcal{Q} :
% \text{ for all } x\in\mathbb{X}, \ d(\varphi^n_q(x),A_{\theta^n q}) \stackrel{n\to \infty}{\longrightarrow} 0 \right\rbrace.
% \]
% Since $\mathbb{X}$ is discrete we have that $\mathbb{P}(\hat{\mathcal{Q}})=1$. Hence, for any $q\in\hat{\mathcal{Q}}$ and any $B\subset \mathbb{X}$ finite it follows that $\varphi^n_{q}(B)\subset A_{\theta^{n}q}$ for $n(B,q)=\max_{x\in B}n(x,q)$, and the result follows. 
Now, we set
\[\mathcal{Q}_B:=\bigcap_{x\in B} \mathcal{Q}^x = \left\lbrace q\in \mathcal{Q}: \lim_{n\to \infty}d\left( \varphi^n_{q}(x),A_{\theta^{n}q}\right)= 0 \; \forall x\in B\right\rbrace.\]
As $B$ is finite, we have 
\[ \mathcal{Q}_B = \left\{q\in \mathcal{Q}:\lim_{n\to \infty}\dist\left( \varphi^n_{q}(B),A_{\theta^{n}q}\right)=0 \right\},\]
and, using $\mathbb{P}(\mathcal{Q}^x)=1$ for all $x$, we get $\mathbb{P}(\mathcal{Q}_B)=1$ meaning that $A$ is a forward attractor (cf.~\eqref{def:forward}).

\noindent \textbf{II implies III.} This follows directly by taking $B=\lbrace x \rbrace$.

\noindent \textbf{III implies IV.} This statement is true from the fact that convergence a.s. implies convergence in probability.

    \noindent \textbf{IV implies I.} Assume that \eqref{def:weak} holds for $B=\{x\}$ for any $x\in \mathbb{X}$, which is equivalent to
    \[
        \lim_{n\rightarrow\infty}\mathbb{P}\left( \varphi^n_q(x)\not\in A_{\theta^n q} \right)=0.
    \]
Let $K=\lbrace x_1,\ldots, x_m\rbrace\subset \mathbb{X}$ for some $m\in \mathbb{N}$. Since for any $n\in\mathbb{N}$
    \[
    \left\lbrace q :  \dist(\varphi^n_q(K),A_{\theta^{n}q})\geq 1 \right\rbrace =
\bigcup_{i=1}^m\lbrace q : \varphi^n_{q}(x_i)\not\in A_{\theta^n q} \rbrace,
    \]
it follows that
\[
    \lim_{n\rightarrow\infty}\mathbb{P}\left(  \dist(\varphi^n_q(K),A_{\theta^{n}q})\geq 1 \right) \leq
    \sum_{i=1}^m{\lim_{n\rightarrow\infty}\mathbb{P}\left( \varphi^n_q(x_i)\not\in A_{\theta^nq} \right)}=0,
\]
and the claim follows.

\noindent \textbf{I if and only if V.} Clearly, V implies I, since for each finite set $B\subset \mathbb{X}$ one can take $K_q=B$ for all $q$.
Conversely, if $K$ is a finite random set and $\varepsilon>0$ an arbitrarily small constant, consider $F_\varepsilon>0$ as given in Proposition~\ref{PROP:randomfiniteInsideDeterministic} so that %observe that for a random set  $K\subset \mathbb{X}$ and $F_m:=\lbrace0,1,2,\ldots,m\rbrace^L \subset \mathbb{X}$  
%Conversely, consider $K\subset \mathbb{X}$ to be a finite random set \sw{i.e., $\mathbb{P}\left(  K_q \text{ is finite}\right)=1$?}. Let us denote $F_m:=\lbrace0,1,2,\ldots,m\rbrace^L \subset \mathbb{X}$.
%Observe that
%\[
%    K_q \text{ is finite } \iff \text{ there exists } m\in\mathbb{N} \text{ such that } K_q\subset F_m.
%\]
% Since $\mathcal{Q}_m:=\lbrace q \in\mathcal{Q} : K_q\subset F_m \rbrace$ 
%Since $\lbrace q : K_q\subset \{0,1,\ldots,m\}^L \rbrace$ 
%is a sequence of sets that is increasing in $m$, i.e., $\mathcal{Q}_m\subseteq \mathcal{Q}_{m+1}$ for all $m$, we get that 
%\[ \lim_{m\rightarrow \infty}\mathbb{P}(\mathcal{Q}_m) =\mathbb{P}\left(\bigcup_{m=1}^\infty \mathcal{Q}_m\right)= \mathbb{P}\left(  K_q \text{ is finite}\right),\]
%\[
%    1=\mathbb{P}\left(  K_q \text{ is finite}\right)=\lim_{m\rightarrow \infty}\mathbb{P}(\mathcal{Q}_m).
%\]
%which is equal to one if $K$ is a finite random set almost surely. 
%Consequently, given $\varepsilon>0$ there is $m\in\mathbb{N}$ such that 
%\[
%    1=\mathbb{P}\left(  K_q \text{ is finite}\right)=\lim_{m\rightarrow \infty}\mathbb{P}(\lbrace q : K_q\subset {0,1,\ldots,m}^L \rbrace).
%\]
 %From the limit above, given $\varepsilon>0$ there is $m\in\mathbb{N}$ such that
\begin{equation}
    \label{eq:randomcompactAlmostdeterministic}
   \mathbb{P}\left( K_q\subset F_\varepsilon \right)\geq 1-\varepsilon.
\end{equation}
Note that for any $n,m\in\mathbb{N}$ we have that
\begin{align*}
   \lefteqn{ \mathbb{P}\left( \varphi^n_q(K_q)\not\subset A_{\theta^n}q \right) } \\
   & =
\mathbb{P}\left( \varphi^n_q(K_q)\not\subset A_{\theta^n}q \land K_q\subset F_\varepsilon \right) + \mathbb{P}\left( \varphi^n_q(K_q)\not\subset A_{\theta^n}q \land K_q\not\subset F_\varepsilon \right).
\end{align*}
Hence, it follows from \eqref{eq:randomcompactAlmostdeterministic} and the observation
\[
    \lbrace \varphi^n_q(K_q)\not\subset A_{\theta^n}q \ \land \ K_q\subset F_\varepsilon\rbrace
    \subset
    \lbrace\varphi^n_q(F_\varepsilon)\not\subset A_{\theta^n}q\rbrace, 
\]
that for each $\varepsilon>0$ there is $m\in\mathbb{N}$ such that
\[
\mathbb{P}\left( \varphi^n_q(K_q)\not\subset A_{\theta^n}q \right)\leq \mathbb{P}\left( \varphi^n_q(F_\varepsilon)\not\subset A_{\theta^n}q \right) +\varepsilon.
\]
Since $F_\varepsilon$ is deterministic and finite, statement I implies $\mathbb{P}\left( \varphi^n_q(F_\varepsilon)\not\subset A_{\theta^n}q \right) \to 0$ for $n\to \infty$. As $\varepsilon$ was arbitrarily small, this implies $\mathbb{P}\left( \varphi^n_q(K_q)\not\subset A_{\theta^n}q \right)\rightarrow 0$  for $n\to \infty$.
%Since $F_m$ is deterministic and $\varepsilon$ was arbitrarily small, by taking limits when $n\rightarrow\infty$ we have that $\mathbb{P}\left( \varphi^n_q(K_q)\not\subset A_{\theta^n}q \right)\rightarrow 0$.
\end{proof}
\begin{rem}
Note that, by virtue of this theorem, weak set and point attractors are the same for this discrete-time and discrete-space setting. Hence, the weak set or point attractor being a random point almost surely is equivalent here; in particular, that implies that the distinction between synchronization and weak synchronization defined via an attractor being a singleton, as done in \cite{flandoli2016}, is not necessary here. In addition, note that, again by Theorem~\ref{THM:weakEquiv}, our definition of synchronization (Sec.~\ref{sec:syn}) conincides with the one in \cite{flandoli2016}, if an attractor exists (under extension of $\mathcal Q_+$ to $\mathcal Q$).
\end{rem}

%Assume our (two-sided time) RDS $(\theta, \varphi)$ has a pullback attractor $A$ \gom{We do not need an attractor for defining a periodic cycle}.
Due to their invariance, it becomes relevant to understand the dynamics within the attractors. In particular, the attractor for the birth-death chain admits a periodic behaviour, as we see later in Theorem~\ref{THM:strongAttractor}. 
\begin{Def}
    \label{DEF:randomCycle}
    A random set $A:\mathcal{Q}\rightarrow\mathcal{P}(\mathbb{X})$ is a \emph{random periodic orbit} of period $M$ for the RDS $(\theta,\varphi)$ if for $\mathbb{P}$-a.e.  $q \in \mathcal{Q}$, $A_q=\{a_0(q), \dotsc, a_{M-1}(q)\}$ such that
\begin{equation*}
\varphi_{q}^1(a_i(q))=a_{i+1 (\bmod M)}(\theta q) \quad \text{for } i = 0, \dotsc, M-1.
 \end{equation*}
\end{Def} 
Clearly a random periodic orbit is in particular $\varphi$-invariant. Furthermore, we say that $A:\mathcal{Q}\rightarrow\mathcal{P}(\mathbb{X})$ is a \emph{(weak) attracting random periodic orbit} if it is a random periodic orbit and a weak attractor.
Note that this can be seen as a discrete-time analogue to the continuous-time oriented definition of a random periodic solution \cite{zhao2009} and its generalization \cite{EngelKuehn2021}.
%\Nathalie{We orientate our definition on how random periodic solutions are defined, right? Then we could cite \cite{zhao2009}} \sw{Max: add references}
%Also, would you agree that "random periodic cycle of period $M$" is somewhat redundant? I would just call it "random cycle of period $M$". What do you all think?

\subsection{Existence of weak attractors}
\label{sec:existence_weak_attr}

In this subsection we provide general conditions for an RDS as given in the previous section to admit a weak attractor. 
For this purpose, we combine \cite[Thm. 10]{crauel2015} on the existence of a weak attractor with properties of the stationary distribution 
%of the system 
(in case of its existence) 
for Markov chains on countable state spaces (cf.~\cite[Chapter 6]{durrett2019probability}).
%must decay to zero for $\Vert x\Vert\rightarrow\infty$.
%\sw{Why do we need this sentence in the intro of this section?}.
%For this purpose, we use the fact that whenever the system admits a stationary distribution it must decay as $\Vert x\Vert\rightarrow\infty$ \sw{What must decay?}. 
%We refer to \cite[Chapter 6]{durrett2019probability} for standard definitions and results on Markov chains on countable state spaces \sw{Why do we not already refer to this before?}. 
%
\begin{Thm}
    \label{THM:existenceWeakAttractor}
    Consider an RDS $(\theta,\varphi)$ on $\mathbb{X}$, as defined in Sec.~\ref{sec:attractors_gen}, and
    %induced by the random difference equation \eqref{orbit_x}. 
    assume that for every $x\in \mathbb{X}$ the set $\varphi^1_{\mathcal{Q}}(x):=\lbrace \varphi^1_q(x) : q\in\mathcal{Q}\rbrace \subset \mathbb{X}$ is finite, and that the Markov chain associated to $(\theta,\varphi)$ is irreducible and recurrent. If it admits a stationary distribution, then the RDS admits a weak attractor.
\end{Thm}

\begin{proof}
    In order to prove the theorem we use the following (in fact, equivalent) criterion for the existence of a weak attractor from \cite[Thm. 10]{crauel2015}:
For every $\varepsilon > 0$ there exists a compact set $C_{\varepsilon} \subset \mathbb{X}$ such that for every compact set $K \subset \mathbb{X}$ there is a $n_0 \in \mathbb{N}$ so that for all $n \geq n_0$
$$ \mathbb P  ( \varphi_q^n (K) \subset C_{\varepsilon}) \geq 1 - \varepsilon.$$
%\Nathalie{Wouldn't it be enough to show the statement only for $K=\{x\}$ because of Theorem \ref{THM:weakEquiv}?} \gom{I believe you are right, but in\cite{crauel2015} the conditions are given for set attractors, and not for point attractors. My guess is that if one goes through the proof replacing everything for point sets, then we would get the result, but I think that at this point it is mostly unnecessary.}
Recall that a recurrent state $x\in\mathbb{X}$ has period $\ell\geq 1$ if $\ell$ is the greatest common denominator of the set 
\[
    \left\lbrace n\geq 1 : \mathbb{P}\left(\varphi^n_{\omega}(x)=x  \right)>0 \right\rbrace.
\]
%\sw{How is this related to Def.~\ref{DEF:randomCycle}? Notation: $M$ vs $\ell$ for the period - decide for one of them.}
%\gom{In general there does not need to be a relationship. The periodicity used here comes from the Markov chain structure, and not as part of an attractor.} 
Since the chain is irreducible, every state has the same period. Moreover, the state space admits a \textit{cyclic decomposition} given as $\mathbb{X}=W_0\cup W_1\cup\cdots\cup W_{\ell-1}$, where for $x\in W_i$ and $i\in\lbrace0,1,\ldots,\ell-1\rbrace$
\[
    \mathbb{P}\left(\varphi^1_q(x)\in W_j \right)=
    \begin{cases}
        1 & \text{if } j=i+1 \ (\text{mod } \ell),\\
        0 & \text{otherwise}
    \end{cases},
\]
and the $\ell$-step process $\left( \varphi^{\ell n}_q(x) \right)_{n\in\mathbb{N}_0}$ is aperiodic and irreducible in each $W_i$ (see for instance \cite[Lemma 6.7.1]{durrett2019probability}). Denote by $\rho$ the unique stationary distribution of the system and let $\varepsilon$ be arbitrarily small. Consider $z_\varepsilon\in\mathbb{N}$ large enough such that
\begin{equation}\label{max_rho}
    \max_{0,1,\ldots\ell-1}\frac{\rho((\mathbb{X}\setminus C_\varepsilon) \cap W_i))}{\rho(W_i)}<\frac{\varepsilon}{2}
\end{equation}
for the set $C_{\varepsilon}:=\lbrace 0,1,2,\ldots z_\varepsilon\rbrace^L$. 
%\sw{$\rho$ is defined on subsets of $\mathbb{X}$, so $\rho([z+1,\infty))$ does not make sense. Replace it by $\rho((\mathbb{X}\setminus C_\varepsilon) \cap W_i)$? }

\noindent \textbf{Step 1.} For each $i=0,\ldots,\ell-1$, the $\ell$-step process $(\varphi_q^{\ell n}(x))_{n \in \mathbb{N}_0}$ starting in $x\in W_i$ admits a unique stationary distribution $\tilde{\rho}_i$ supported on $W_i$. 

For $i=0,\ldots,\ell-1$ and $A\subset \mathbb{X}$ set
\begin{equation}\label{tilde_rho}
    \tilde{\rho}_i(A):=\frac{\rho(A \cap W_i)}{\rho(W_i)}.
\end{equation}
Then, for any $n\in \mathbb{N}$ we have
\begin{align*}
    \sum_{x=0}^{\infty}{\mathbb{P}(\varphi_q^{\ell n}(x)\in A) \tilde{\rho}_i(x)} 
     & =
    \frac{1}{\rho(W_i)}\sum_{x \in W_i}{\mathbb{P}(\varphi_q^{\ell n}(x)\in A) \rho(x)} \\
         & \stackrel{(*)}{=}
    \frac{1}{\rho(W_i)}\sum_{x \in W_i}{\mathbb{P}(\varphi_q^{\ell n}(x)\in A \cap W_i) \rho(x)} \\
             & =
    \frac{1}{\rho(W_i)}\sum_{x=0}^\infty{\mathbb{P}(\varphi_q^{\ell n}(x)\in A \cap W_i) \rho(x)} \\
    & =\frac{\rho(A\cap W_i)}{\rho(W_i)}=  \tilde{\rho}_i(A),
\end{align*}
    which implies that $\tilde{\rho}_i$ is a stationary distribution of $(\varphi_q^{\ell n}(x))_{n \in \mathbb{N}_0}$  on $W_i$. In $(*)$ we used the fact that $\mathbb{P}(\varphi_q^{\ell n}(x_0)=y)=0$  $\forall x_0\in W_i, y \notin W_i$. 
    As the $\ell$-step process is irreducible in each $W_i$, the stationary distribution is unique (see \cite[Theorem 6.5.5]{durrett2019probability}).
    
\noindent \textbf{Step 2.} Let $x\in \mathbb{X}$ and $\varepsilon>0$ be given as before.
%\sw{$\varepsilon$ is already given.}. 
Then, there exists $n_1= n_1(x)\in\mathbb{N}$ such that ${\mathbb{P}(\varphi_q^{\ell n}(x)\not\in C_\varepsilon)<\varepsilon}$ for all $n\geq n_1$.

Consider $x\in W_i$ for any $i\in\lbrace0,\ldots,\ell-1\rbrace$. Since the $\ell$-step process $(\varphi_q^{\ell n}(x))_{n \in \mathbb{N}_0}$ is irreducible and admits a stationary distribution, it is then positive recurrent on $W_i$. Furthermore, since it is aperiodic in $W_i$, convergence $\mathbb{P}(\varphi_q^{\ell n}(x)\in \cdot)\rightarrow \tilde{\rho}_i$ in total variation holds as $n \to \infty$, see \cite[Theorem 6.6.4]{durrett2019probability}. 
This implies in particular the convergence $\mathbb{P}(\varphi_q^{\ell n}(x)\not\in C_\varepsilon)\rightarrow \tilde{\rho}_i( \mathbb{X}\setminus C_{\varepsilon} )$. Hence, there exists $n_1= n_1(x)\in\mathbb{N}$ such that for any $n\geq n_1$ we obtain
\begin{equation}\label{P_C_eps}
     \mathbb{P}(\varphi_q^{\ell n}(x)\not\in C_\varepsilon) < \tilde{\rho}_i( \mathbb{X}\setminus C_{\varepsilon} )+\frac{\varepsilon}{2} \stackrel{\eqref{tilde_rho}}{=}\frac{\rho((\mathbb{X}\setminus C_{\varepsilon})\cap W_i)}{\rho(W_i)}+\frac{\varepsilon}{2} \stackrel{\eqref{max_rho}}{<} \varepsilon.
\end{equation}

\noindent \textbf{Step 3.} Let $x\in \mathbb{X}$ be an arbitrary initial condition. Then, there exists $n_2= n_2(x)\in\mathbb{N}$ such that for all $n\geq n_2$ we have ${\mathbb{P}(\varphi_q^{n}(x)\not\in C_{\varepsilon})<\varepsilon}$. 

Since $\varphi^1_{\mathcal{Q}}(x)$ is assumed to be finite, also $\varphi^r_{\mathcal{Q}}(x):=\lbrace \varphi^r_q(x) : q\in\mathcal{Q} \rbrace$ is finite for any $r\in\mathbb{N}_0$. Let $n=k\ell+r$ for some $k\in\mathbb{N}_0$ and $r\in\lbrace0,1,\ldots,\ell-1 \rbrace$. Note that
\[\mathbb{P}\left(\varphi_q^{\ell k +r}(x)\not\in C_\varepsilon\right) = \sum_{y\in\varphi^r_{\mathcal{Q}}(x)} \mathbb{P}\left(\varphi^{\ell k}_q(y)\not\in C_\varepsilon\right)\cdot \mathbb{P}\left(\varphi^r_q(x)=y\right).
\]
Let $n_1(y)$ be given from Step 2 above, and choose $k\geq \max \lbrace n_1(y) : y\in \varphi^r_{\mathcal{Q}}(x)\rbrace$. Then
\[\mathbb{P}(\varphi_q^{\ell k +r}(x)\not\in C_\varepsilon) \stackrel{\eqref{P_C_eps}}{<}\varepsilon \cdot \sum_{y\in\varphi^r_{\mathcal{Q}}(x)} \mathbb{P}(\varphi^r_q(x)=y)\ =\varepsilon.
\]
 So, the claim follows by taking $n_2(x):=\max\lbrace \ell\cdot n_1(y) :   y\in\varphi^r_\mathcal{Q}(x), r\in\lbrace0,1,\ldots,\ell-1 \rbrace\rbrace$.

\noindent \textbf{Step 4.} For each $K\subset \mathbb{N}_0$ compact (i.e.~finite) there exists $n_0= n_0(K)\in\mathbb{N}$ such that 
%$\mathbb{P}(\varphi_q^{2n+1}(x)>z)<\varepsilon$ for all $n\geq n_2$ and $x\in K$.
$\mathbb{P}(\varphi_q^{n}(x)>z)<\varepsilon$ for all $n\geq n_0$ and $x\in K$.

Notice that Step 3 assures that the claim holds when $K=\lbrace x \rbrace$ for any $x\in \mathbb{N}_0$. If we consider any finite set $K=\lbrace x_1,x_2,\ldots,x_l\rbrace$, then the result follows by taking $n_0(K)=\max_{i=1,\ldots, l}n_2( x_i)$.
\end{proof}

%\sw{Is Theorem~\ref{THM:existenceWeakAttractor} for discrete-time, discrete-space RDS in general?}
Note that, by virtue of Theorem~\ref{THM:existenceWeakAttractor}, the existence of the weak attractor can be derived purely by Markov chain arguments and is expected to occur in a large class of RDS derived from reaction jump processes via the embedded Markov chain approach. For instance, since chemical reaction networks are defined via a finite set of reactions, $\varphi^1_{\mathcal{Q}}(x)$ is always finite for any $x\in\mathbb{X}$. In particular, using Theorem~\ref{THM:existenceWeakAttractor} we can directly show that the RDS corresponding to the embedded Markov chains of the birth-death process and the Schl\"ogl model admit a weak attractor, as presented in the next sections.

\subsection{Random periodic orbit of the birth-death chain} \label{sec:birth-death_attr}
The concept of a random attractor makes sense only when considering an invertible dynamical system on $\mathcal{Q_+}$. However, in the RDS formulation of the embedded Markov chain (see Sec.~\ref{Sec:RDS}), the shift map is not invertible since for $q_0\neq \tilde{q}_0$ we have that
\[
    \theta(q_0,q_1,q_2,\ldots)=\theta(\tilde{q}_0,q_1,q_2)=(q_1,q_2,\ldots).
\]
We can come around this inconvenience by redefining our noise space as
\[\mathcal{Q}:=\lbrace q=(q_n)_{n\in\mathbb{Z}} : q_n\in[0,1] \rbrace,\]
endowed with its Borel $\sigma$-algebra $\sigma(\mathcal{Q})$ and the bi-infinite product measure $\lambda^{\mathbb{Z}}$. We redefine the shift map as 
\begin{equation*} %\label{shift_map}
   \theta q=\theta (q_n)_{n\in\mathbb{Z}} =(q_{n+1})_{n\in \mathbb{Z}}
\end{equation*}
such that $(\theta^{-n}q)_i = q_{i-n}$, 
while the cocycle map $\varphi$ remains the same, see \eqref{eq:iterates_cocycles}.
Note that the synchronization results in Section~\ref{sec:synchro} transfer immediately to the invertible setting since they have been formulated independently from the past. By abuse of notation, we will from now on use $\mathbb P=\lambda^{\mathbb{Z}}$.

\subsubsection{Weak attraction}

 In this subsection we provide a full characterization of the weak random attractor for the birth-death process given in Example~\ref{ex:birth-death}. 

\begin{prop}
\label{PROP:stationary distribution}
    The embedded Markov chain of the birth-death process %and the Schl\"ogl model (Example~\ref{ex:schloegl}) 
    admits a unique stationary distribution. Therefore, the associated RDS $(\theta,\varphi)$ admits a (unique) weak attractor.
\end{prop}

\begin{proof}
Recall that the transition probabilities for the embedded Markov chain for the birth-death process are given in \eqref{eq:trans_prob}. %Analogously, the transition probabilities for the embedded Markov chain of the Schl\"ogl model are
%\begin{equation}
%\label{eq:Schloegl_P}
%    P_1(x) =  \frac{\rate_1+\rate_3x(x-1)}{\mu(x)}, \quad
%    P_{-1}(x) =  \frac{\rate_2x+\rate_4 x(x-1)(x-2)}{\mu(x)}
%\end{equation}

    From \cite[p. 304]{durrett2019probability}, we know that any of the chains with transition probabilities given by \eqref{eq:trans_prob} or \eqref{eq:Schloegl_P} admits a unique stationary distribution if and only if
\begin{equation}
    \label{eq:condition_stationary_1}
    \zeta:=\sum_{x=1}^{\infty}\prod_{j=0}^{x-1}\frac{P_1(j)}{P_{-1}(j+1)}<\infty.
\end{equation}
We show that in both cases \eqref{eq:condition_stationary_1} holds.
%\noindent \textbf{Case 1: Birth-death.} 
We consider the transition probabilities \eqref{eq:trans_prob}, and for simplicity let $\alpha:=\frac{\rate_2}{\rate_1} > 0$. Then,
\begin{align}
    \label{eq:condition_stationary_2}
    \zeta&=  \sum_{x=1}^{\infty}\prod_{j=0}^{x-1}\frac{1+\alpha(j+1)}{(1+\alpha j)\alpha(j+1)} =
    \sum_{x=1}^{\infty}\prod_{j=0}^{x-1}\left( 1+\frac{\alpha}{1+\alpha j} \right)\cdot \frac{1}{\alpha(j+1)}  \nonumber\\
    & \quad  \leq  \sum_{x=1}^{\infty}\prod_{j=0}^{x-1}\left( \frac{1+\alpha}{\alpha} \right)\cdot\frac{1}{j+1} =\sum_{x=1}^{\infty} \frac{\beta^x}{x!},
\end{align}
where $\beta=(1+\alpha)/\alpha$. Recall that by Stirling's approximation we have that $x!/(\sqrt{2\pi x}(\frac{x}{e})^x)\rightarrow 1$ as $x\rightarrow \infty$. Then, there exists  $N\in\mathbb{N}$ such that for all $x\geq N$ we have
\begin{equation}
    \label{eq:condition_stationary_3}
    \frac{1}{x!}<\frac{2}{\sqrt{2\pi x}\left( \frac{x}{e}\right)^x}.
\end{equation}
From \eqref{eq:condition_stationary_2} and \eqref{eq:condition_stationary_3}, and assuming without loss of generality that $N\geq 2\pi$, it follows that
\[
    \zeta < S+ \sqrt{2/\pi}\sum_{x=N}^{\infty} \beta^x\cdot \frac{1}{\sqrt{x}\left( \frac{x}{e} \right)^x}\leq 
    S+ \sqrt{2/\pi}\sum_{x=N}^{\infty} \left(\frac{e\beta}{x}\right)^x
\]
for $S:=\sum_{x=1}^{N-1}\beta^x\cdot \frac{1}{x!}$. 
Hence, it suffices to show that
\[
    \sum_{x=N}^{\infty} \left(\frac{e\beta}{x} \right)^x<\infty.
\]
Assume again without loss of generality that $N$ is big enough so that $\frac{e\beta}{x}<\frac{1}{e}$ for all $x\geq N$. Thus,
\[
    \sum_{x=N}^{\infty}{\left( \frac{e\beta}{x}\right)^x}<\sum_{x=N}^{\infty}{e^{-x}}<\infty,
\]
from which we finally deduce $\zeta <\infty$.
% \noindent \textbf{Case 2: Schl\"ogl model.} Consider now the transition probabilities given in \eqref{eq:Schloegl_P}. We show that the quantity $\zeta$ in (\ref{eq:condition_stationary_1}) is finite. First, We bound the product inside of \eqref{eq:condition_stationary_1}:
%     \[
%         \prod_{j=0}^{x-1}\frac{P_1(j)}{P_{-1}(j+1)}=
%         \prod_{j=0}^{x-1}\left(\frac{\gamma_1+\gamma_3 j (j-1)}{\gamma_2(j+1)+\gamma_4 j(j+1)(j-1)}\right)\cdot
%         \left( \frac{\mu(j+1)}{\mu(j)} \right)=
%     \]
%     \[
%         =\frac{1}{x!}\prod_{j=0}^{x-1}\left(\frac{\gamma_1+\gamma_3 j (j-1)}{\gamma_2+\gamma_4 j(j-1)}\right)\cdot
%         \left( \frac{\mu(j+1)}{\mu(j)} \right)
%     \]
% Both terms inside the product are bounded for all $j\in\mathbb{N}_0$. Thus, there exists $\eta>0$ such that
% \[
%     \zeta \leq \sum_{x=1}^{\infty}
%     \frac{1}{x!}\eta^x<\infty.
% \]
% The finiteness of the right-hand side follows from Stirling's approximation analogously to Case 1 above.
\end{proof}

Figure~\ref{fig:BDP_stat_distr} shows the stationary distribution of the one- and two-point motion for the birth-death chain. The state space of the two-point motion splits up into a transient area $\mathbb{X}^2\setminus \mathbb{D}$ and two communication classes $\Delta$ and $\mathbb{D}\setminus \Delta = I_{-1}\cup I_1$ \cite{baxendale1991statistical} where $\Delta$ is the diagonal given in  \eqref{def:diagonal} and $\mathbb{D}$ is the thick diagonal defined in \eqref{D}. The respective stationary distributions $\pi_{\Delta}$ and $\pi_{\mathbb{D}\setminus \Delta}$  are plotted in Figure~\ref{fig:BDP_stat_distr}(b)-(c).

%, as well as the stationary distributions $\pi_{\mathcal{A}}$, $\pi_{\bar{\mathcal{A}}}$ of the associated two-point motion starting in $\mathcal{A}:=(W_0\times W_0)\cup (W_1 \times W_1)$ resp. $\bar{\mathcal{A}}:=(W_0\times W_1)\cup (W_1 \times W_0)=\mathbb{X}^2\setminus \mathcal{A}$. 
%\Nathalie{The two-point motion is a reducible Markov chain: two classes $\mathcal{A}$ and $\bar{\mathcal{A}}$ between which transitions in neither direction are possible, but also into $\mathcal{A}\setminus \Delta$ and $\Delta$ (resp. $\bar{\mathcal{A}}\setminus (I_{-1}\cup I_1)$ and $I_{-1}\cup I_1$), where the two-point is transient on $\mathcal{A}\setminus \Delta$ and persistent on $\Delta$. Essentially we consider the stationary distributions $\pi_{\mathcal{A}}$ and $\pi_{\bar{\mathcal{A}}}$ only for the two-point motions on $\Delta$ and $I_{-1}\cup I_1$. https://math.stackexchange.com/questions/1183298/can-a-reducible-markov-chain-have-an-unique-stationary-distribution, see also "Statistical Equilibrium and Two-Point Motion for a Stochastic Flow of Diffeomorphisms" Baxendale}
%\gom{What would the stationary distribution on the off-diagonals be? To think about it} 

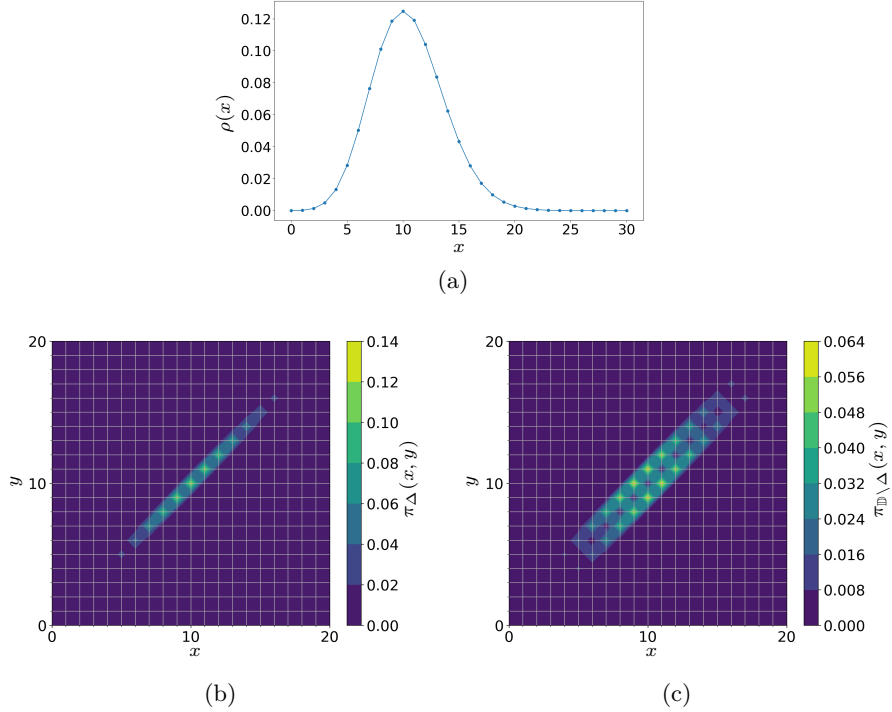
\begin{figure}
    \centering
     \begin{subfigure}[b]{0.52\textwidth}
        \centering
              \begin{overpic}[width=\textwidth]{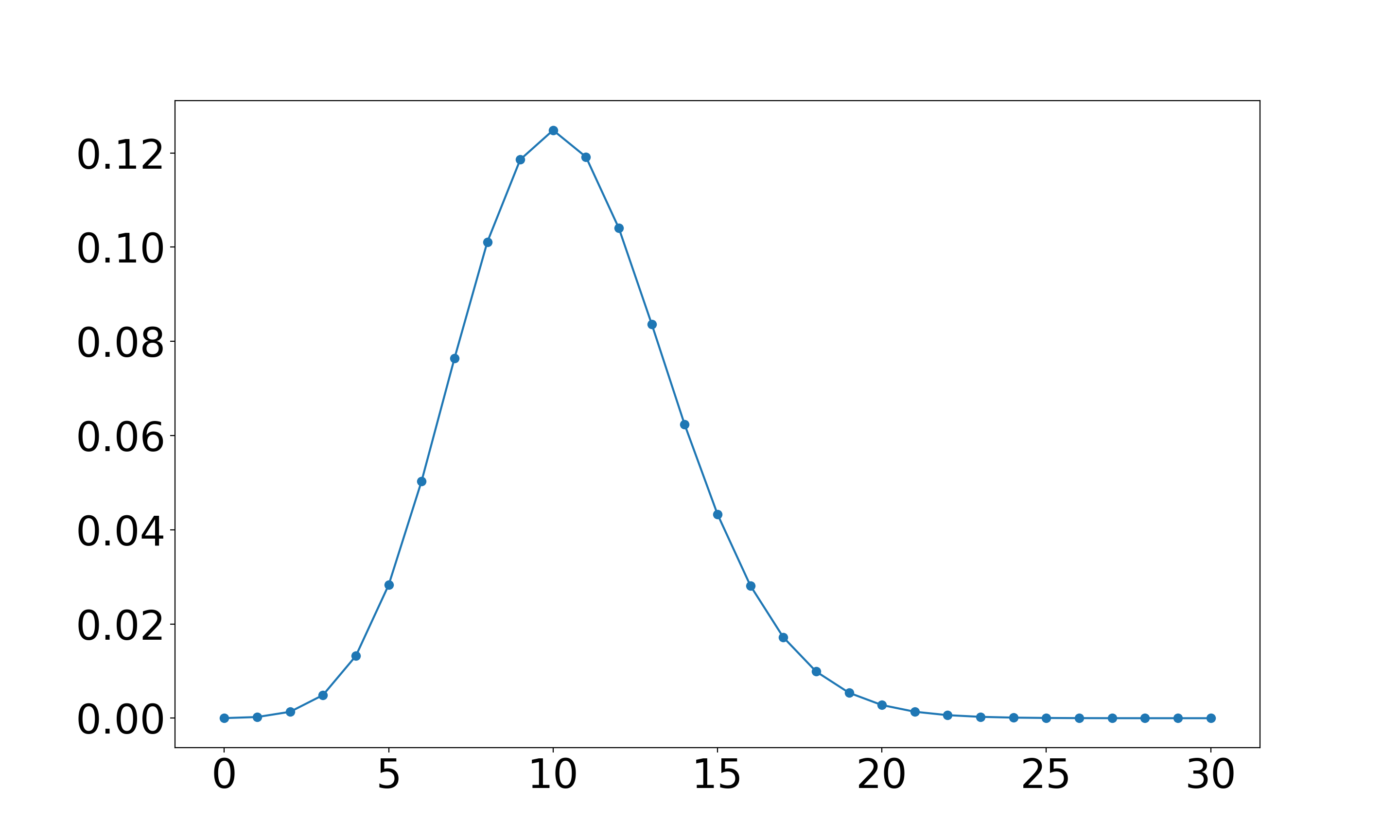}
        \put(0,25){\scriptsize \begin{turn}{90} $\rho(x)$ \end{turn} }
        \put(50,0){\scriptsize $x$}
        \end{overpic}
        \caption{}
    \end{subfigure}
    \hfill
    \begin{subfigure}[b]{0.49\textwidth}
        \centering
                \begin{overpic}[width=\textwidth]{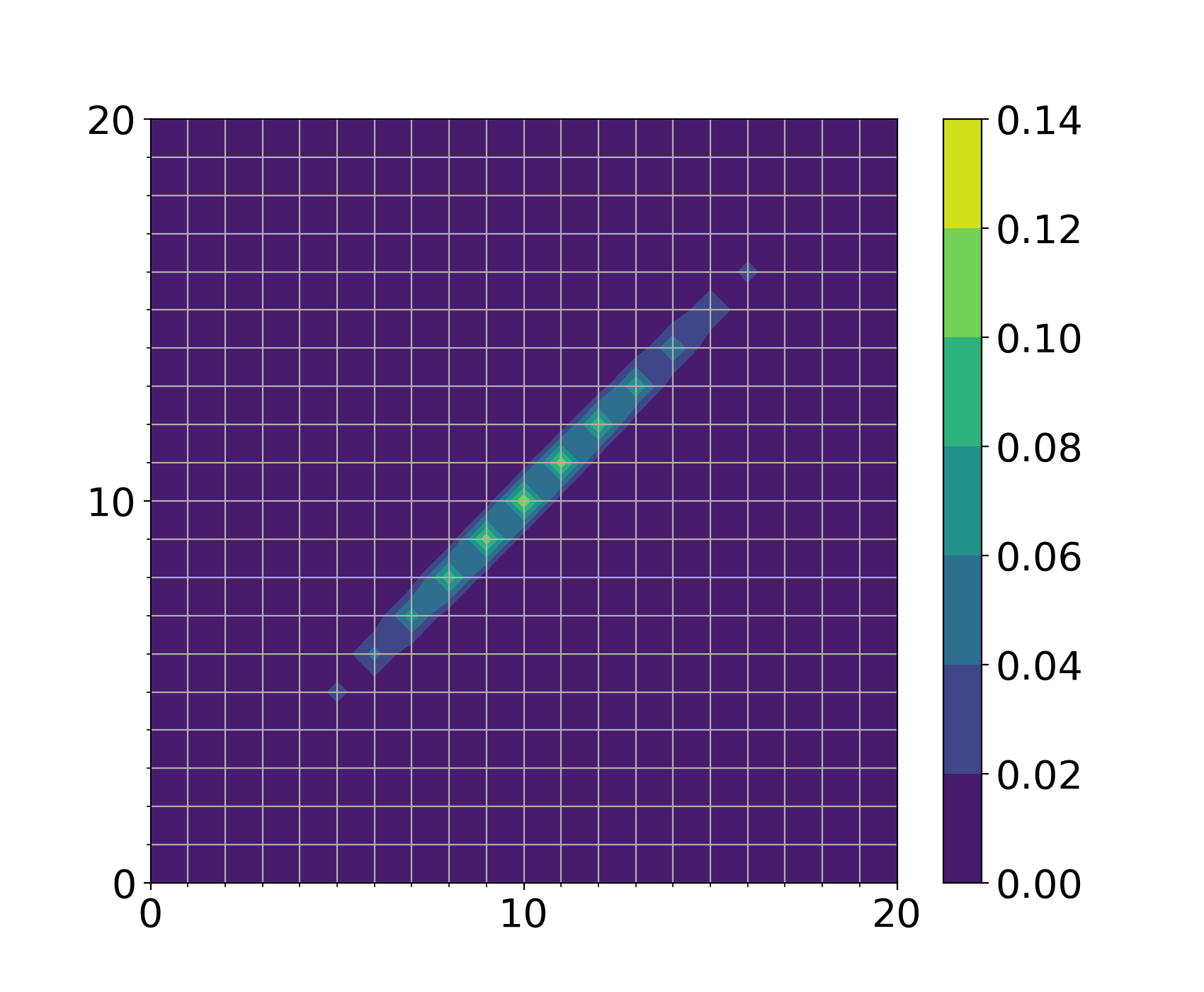}
        \put(3,40){\scriptsize \begin{turn}{90} $y$ \end{turn} }
        \put(43,2){\scriptsize $x$}
        \put(90,33){\scriptsize \begin{turn}{90} $\pi_{\Delta}(x,y)$ \end{turn} }
        \end{overpic}
        \caption{}
    \end{subfigure}
    \hfill
     \begin{subfigure}[b]{0.49\textwidth}
        \centering
        \begin{overpic}[width=\textwidth]{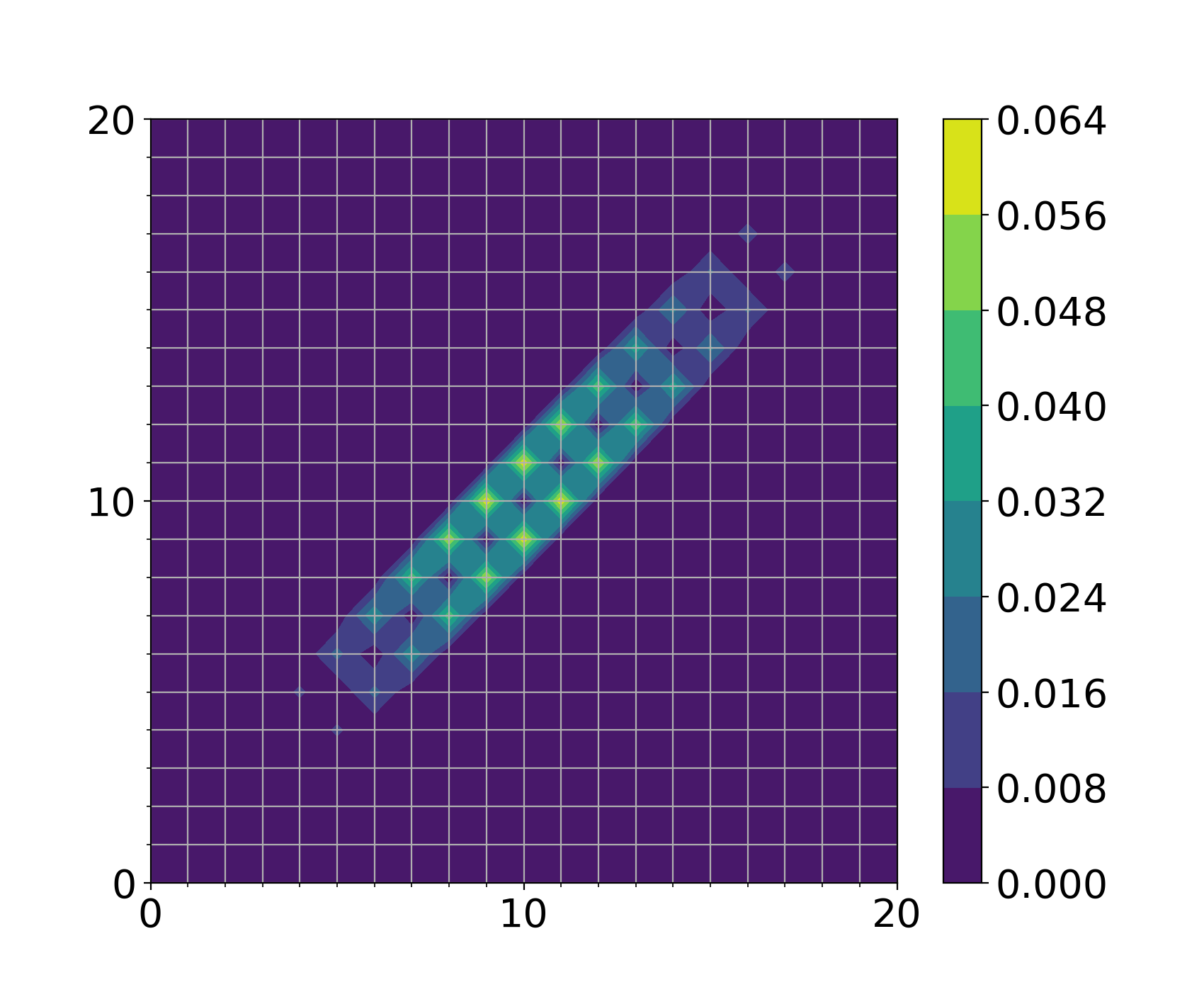}
        \put(3,40){\scriptsize \begin{turn}{90} $y$ \end{turn} }
        \put(43,2){\scriptsize $x$}
        \put(93,33){\scriptsize \begin{turn}{90} $\pi_{\mathbb{D}\setminus\Delta}(x,y)$ \end{turn} }
        \end{overpic}
        \caption{}
    \end{subfigure}
    \hfill
    \caption{Stationary distributions of one- and two-point motion of the birth-death chain. (a) Stationary distribution $\rho$ of the one-point motion and (b), (c) stationary distributions $\pi_{\Delta}$ and $\pi_{\mathbb{D}\setminus \Delta}$ of the two-point motion, respectively.  %$\pi_{\mathcal{A}}$, $\pi_{\bar{\mathcal{A}}}$ of the two-point motion on the sets $\mathcal{A}:=(W_0\times W_0)\cup (W_1 \times W_1)$ resp. $\bar{\mathcal{A}}:=(W_0\times W_1)\cup (W_1 \times W_0)=\mathbb{X}^2\setminus \mathcal{A}$.  
    All stationary distributions are computed as the solution of their eigenvector equation $\rho^TP=\rho^T$ for $P$ the corresponding transition matrix. $\rate_1 =10, \rate_2 = 1$.
    }
    \label{fig:BDP_stat_distr}
\end{figure}

In the following, we investigate the structure of the weak attractor by formulating  Proposition~\ref{PROP:attractor2points} and Corollary~\ref{COR:limit Exists} as preparatory work for Section~\ref{sec:strong_pullback}, where we will show that the weak attractor is in fact a strong pullback attractor consisting of two random points that have distance one. 
%The structure of the random attractor is now specified in the following derivation:  Let $q\mapsto A_q$ be the weak attractor for the birth-death chain. We will show that it is in fact a strong pullback attractor and that it consists of two random points that have distance one. We prove these statements in three steps given by Proposition~\ref{PROP:attractor2points}, Corollary~\ref{COR:limit Exists}, and Theorem~\ref{THM:strongAttractor}. 
The crucial insight is the translation of the partial synchronization results from Proposition~\ref{PROP:synchronisation} and Corollary~\ref{COR:partial synchronization} into the random periodic structure of the random attractor. Depending on the type of partition in such a partial synchronization, the following analysis may well be understood as a blueprint for various forms of augmented Markov chains with more complicated structure.
\begin{prop}
\label{PROP:attractor2points}
    The weak attractor $q\mapsto A_q$ of the birth-death chain has two points on each fiber $\mathbb{P}$-a.s., that is 
    \[
        \mathbb{P}(\#A_q=2)=1.
    \]
Moreover, $A_q\cap W_i\neq \emptyset$ for $i=0,1$ $\mathbb{P}$-a.s., where $W_i$ are given in Corollary \ref{COR:partial synchronization}.
\end{prop}
\begin{proof}
We split the proof in two parts. We first show that $A_q$ has at least two points $\mathbb{P}$-a.s., and then we conclude that $\mathbb{P}(\# A_q\geq 3)=0$.

\noindent \textbf{Step 1.}  $A_q\cap W_i\neq \emptyset$ for $i=0,1$ $\mathbb{P}$-a.s.

For each $x\in\mathbb{N}_0$ consider the pullback limit
\begin{equation*}
    %\label{eq:limsuppullback}
        a_x^+(q):=\limsup_{n\rightarrow\infty}\varphi^{2n}_{\theta^{-2n}q}(x).
\end{equation*}
Let $(n_j)_{j=0}^{\infty}$ be a strictly increasing sequence of natural numbers such that $\varphi^{2n_j}_{\theta^{-2n_j}q}(x)\rightarrow a_x^+(q)$ for $j\to \infty$. Since $A_q$ is a weak attractor we have
\[
    \lim_{j\rightarrow \infty} \mathbb{P}\left( d\left( \varphi^{2n_j}_{\theta^{-2n_j}q}(x), A_q \right)\geq 1 \right)=0.
\]
From the last limit we consider a subsequence $n_{j_k}$ such that $\varphi^{2n_{j_k}}_{\theta^{-2n_{j_k}}q}(x)\in A_q$ for all $k\in\mathbb{N}_0$ $\mathbb{P}$-a.s. We conclude that $a_x^+(q)\in A_q$ for $\mathbb{P}$-a.e. $q$, which in particular implies that $a_x^+<\infty$ $\mathbb{P}$-a.s. Furthermore, since $a_x^+(q)\in W_i$ if and only if $x\in W_i$, the claim follows by taking $x=0,1$, for instance.

\noindent \textbf{Step 2.} $\mathbb{P}(\# A_q\geq 3)=0$. 

Since $\theta$ is $\mathbb{P}$-invariant, we have that for any $n\in\mathbb{N}$
\[
    \mathbb{P}(\# A_q\geq 3)= \mathbb{P}(\# A_{\theta^n q}\geq 3)=\mathbb{P}(\# \varphi^n_q(A_q)\geq 3),
\]
where the last equality follows by the invariance of $A_q$. 
We now partition 
\[\varphi^n_q(A_q)=\varphi^n_q(A_q\cap W_0)\cup \varphi^n_q(A_q\cap W_1).\]
Since it follows from Step 1 that for all $n\in\mathbb{N}$ and $i\in\lbrace 0,1\rbrace$ the sets $\varphi^n_q(A_q\cap W_i)$ contain at least one element, we can combine these observations to give the bound
\[
\mathbb{P}(\# A_q\geq 3)\leq \mathbb{P}(\# \varphi^n_q(A_q\cap W_0)\geq 2) + \mathbb{P}(\# \varphi^n_q(A_q\cap W_1)\geq 2).
\]
By taking the limit $n\rightarrow \infty$, Proposition~\ref{Prop:randomConvProb} implies that the right-hand side tends to $0$ by taking $K^i_q=A_q\cap W_i$. The result follows.
\end{proof}

\begin{cor}
    \label{COR:limit Exists}
        For each $x\in \mathbb{N}_0$
        \begin{equation}
    \label{eq:pullbacklimit}
     a_x(q):=\lim_{n\rightarrow\infty}\varphi^{2n}_{\theta^{-2n}q}(x)
    \end{equation}
    exists $\mathbb{P}$-a.s.. Furthermore, for $x,y\in W_i$, $i=0,1$, we have that $a_x=a_y$ $\mathbb{P}$-a.s..
    Conversely, if $x\in W_0$ and $y \in W_1$, then $a_x\neq a_y$ $\mathbb{P}$-a.s..
\end{cor}
\begin{proof}
    For each $x\in\mathbb{N}_0$ let
    \[
        a_x^-(q):=\liminf_{n\rightarrow\infty} \varphi^{2n}_{\theta^{-2n}q}(x).
    \]
Analogously to Step 1 in the proof of Proposition~\ref{PROP:attractor2points}, we obtain $a_x^-(q)\in A_q$ for almost all $q \in \mathcal Q$. Since it has the same parity as $a_x^+(q)$ and $\#A_q\cap W_i=1$, $i=0,1$, we conclude that $a_x^+=a_x^-$ holds with full probability and, hence, the limit $a_x \in A_q$ exists almost surely.

 Using again that $\#A_q\cap W_i=1$, $i=0,1$, we derive that $a_x$ and $a_y$ are identical (or different, respectively) in a full measure set when $x$ and $y$ are of the same parity (or of different parities, respectively).
\end{proof}

%\begin{rem}
%\label{REM: oddpullback}
%    By the last statement, we have that $\varphi^{2n}_{\theta^{-2n}q}(x)\rightarrow a_i(q)$ as $n\rightarrow\infty$, whenever $x\in W_i$ ($i=0,1$). For the odd pullback iterates we have
%    \[
%        $\varphi^{2n}_{\theta^{-2n}q}(x)
%    \]
%\end{rem}

\subsubsection{Strong pullback attraction to a random period orbit}\label{sec:strong_pullback}

We can now give a full characterization of the pullback attractor in the next theorem. We already know from Proposition~\ref{PROP:attractor2points} and Corollary~\ref{COR:limit Exists} that the weak attractor $A_q$ consists almost surely of the two distinct random points $a_0(q)$ and $a_1(q)$, as given in~\eqref{eq:pullbacklimit}. Now, we identify the strong pullback structure of this attractor.
%\gom{I have changed some of the suggestions made, and modified a bit the presentation of the arguments in the theorem below so they are easier to read. Please have a look and check if you're comfortable with this presentation.}

\begin{Thm}
    \label{THM:strongAttractor}
        The weak attractor $A_q = \{a_0(q), a_1(q)\}$ for the birth-death chain
        \begin{itemize}
            \item[(i)]  is a random periodic orbit of period 2,
            \item[(ii)] is a pullback and a forward attractor, and
            \item[(iii)] satisfies $\left| a_0(q) - a_1(q) \right| =1$ for $\mathbb{P}$-almost all $q \in \mathcal Q$.
        \end{itemize}

\end{Thm}

\begin{proof}
At first, we use the cocycle property in order to show item (i), that is
\begin{equation}
        \label{eq:randomcycle}
            \varphi_{q}^1(a_0(q))=a_1(\theta q) \qquad \text{and} \qquad \varphi_{q}^1(a_1(q))=a_0(\theta q)
        \end{equation}
is satisfied $\mathbb{P}$-a.s. Indeed, let $i=0,1$, $q \in \mathcal{Q}$ be fixed. Then,
\begin{equation*}
        \varphi_{q}^1(a_i(q))=\lim_{n\rightarrow\infty}\varphi^{2n+1}_{\theta^{-2n}q}(i)=\lim_{n\rightarrow\infty}\varphi^{2n}_{\theta^{-2n}\circ\theta q}\left( \varphi_{q}^1(i) \right)=a_{i+1 (\bmod 2)}(\theta q), 
\end{equation*}
where the last equality follows directly from Corollary~\ref{COR:limit Exists} and the fact that $\varphi_{q}^1(i)\in W_{i+1 (\bmod 2)}$.

In order to show (ii), let $q \in \mathcal{Q}$. It follows from the definition \eqref{eq:pullbacklimit} that, for any $x\in W_i$ with $i=0,1$, there exists $N_0= N_0(x)\in\mathbb{N}_0$ such that for all $n\geq N_0$ we have $\varphi_{\theta^{-2n}q}^{2n}(x)=a_i(q)\in A(q)$. On the other hand, by the cocycle property and \eqref{eq:randomcycle} it follows that
\begin{align*}
\lim_{n\rightarrow\infty}\varphi_{\theta^{-2n-1}}^{2n+1}(x)&=\lim_{n\rightarrow\infty}
\varphi^1_{\theta^{-1}q}\left( \varphi^{2n}_{\theta^{-2n}\theta^{-1}q}(x) \right) \\
&=
\varphi^1_{\theta^{-1}q}(a_{i}(\theta^{-1}q))=a_{i+1(\bmod 2)}(q).
\end{align*}
Hence, there is $N_1= N_1(x)\in\mathbb{N}$ such that for all $n\geq N_1$ we have that $\varphi_{\theta^{-2n-1}q}^{2n+1}(x)\in A(q)$. Combining both parts, we obtain $\varphi_{\theta^{-n}q}^{n}(x)\in A(q)$ for all $n\geq N(x):=\max \lbrace 2N_0(x), 2N_1(x)+1\rbrace$. %Therefore, if $B\subset \mathbb{N}_0$ is finite, we get that $\varphi_{\theta^{-n}}^{n}(B)\subset A(q)$ holds for all $n\geq N(B):=\max_{i\in B}{N(x)}$. This proves that $A$ is a pullback attractor.
Recall that point pullback attractors are (set) pullback attractors due to the state space being discrete, and thus $A$ is a pullback attractor. $A$ is also a forward attractor due to Theorem~\ref{THM:weakEquiv}.

Last, we prove (iii). Consider the sets
\[
    \Omega_n:=\left\lbrace q \in\mathcal{Q}: \varphi_{\theta^{-2k}q}^{2k}(0)=a_0(q), \ \varphi_{\theta^{-2k}q}^{2k}(1)=a_1(q) \ \forall k\geq n \right\rbrace.
\]
Since $\Omega_n\subset \Omega_{n+1}$ for all $n\in\mathbb{N}$, and since $\mathbb{P}(\bigcup_{n\in\mathbb{N}} \Omega_n)=1$, for each $\varepsilon>0$ there is $N$ such that for all $k\geq N$ sufficiently large
\begin{align*}
   \mathbb{P}\left(\left| a_0(q) - a_1(q) \right| \geq 2 \right) &\leq  \mathbb{P}\left(\lbrace q\in\mathcal{Q}: \left| a_0(q) - a_1(q) \right| \geq 2 \rbrace \ \cap \Omega_N \right) + \frac{\varepsilon}{2} \\
   &= \mathbb{P}\left( \left| \varphi_{\theta^{-2k}q}^{2k}(0) - \varphi_{\theta^{-2k}q}^{2k}(1) \right| \geq 2  \right) +\frac{\varepsilon}{2} \\
    &= \mathbb{P}\left( \left| \varphi_{q}^{2k}(0) - \varphi_{q}^{2k}(1) \right| \geq 2 
   \right) +\frac{\varepsilon}{2}  < \varepsilon,
\end{align*}
due to Proposition~\ref{PROP:synchronisation}. Since $\varepsilon$ was arbitrarily small, the result follows.
\end{proof}
%\me{We should finally emphasize that here we have, in particular, that $\left|a_0(q) - a_1(q)\right| =1$, as may be derived by applying Proposition~\ref{PROP:synchronisation} to $a_0(q)$ and $a_1(q)$ but still needs a proper argument; something, we probably do not obeserve for Schlögl, even if there is a cycle of period 2?}

\subsubsection{Sample measures supported on the attractor}

In this subsection we briefly describe the statistical importance of the attractor $A$ in terms of the invariant measure for the skew-product map $\Theta:\mathcal{Q}\times \mathbb{X}\rightarrow \mathcal{Q}\times \mathbb{X}$ given by
\[
    \Theta(q,x):=(\theta q, \varphi_q(x)).
\]
Denoting by $ T^*\mu$ the push forward of a measure $\mu$ by a map $T$, i.e. $T^* \mu(\cdot) = \mu (T^{-1}(\cdot))$, we adopt the classical definition of an invariant measure for the RDS (see e.g. \cite[Definition 1.4.1]{arnold1998}):
A probability measure $\mu$ on $\mathcal Q \times \mathbb{X}$ is invariant for the random dynamical system $(\theta, \varphi)$ if
\begin{enumerate}
\item[(i)] $\Theta_t^* \mu = \mu$ for all $ t \in \mathbb \mathbb{N}_0$\,, 
\item[(ii)] the marginal of $\mu$ on $\mathcal Q$ is $\mathbb{P}$, i.e.~$\mu$ can be factorized uniquely into 
$$\mu(\rmd q,  x) = \mu_{q}( x) \mathbb{P}(\rmd q),$$
where $q\mapsto \mu_{q}$ is the \textit{sample measure}  (or \textit{disintegration}) on $\mathbb{X}$, i.e., $\mu_{q}$ is almost surely a probability measure on $\mathbb{X}$  and $q \mapsto \mu_{q}(B)$ is measurable for all $B \subset \mathbb{X}$.
\end{enumerate}
%Recall that we assume the model of the noise to be fixed. Hence, the marginal of $\mu$ on the probability space is demanded to be $\mathbb{P}$.
In particular note that, since $\mathbb P$ is given, the sample measures $\mu_q$ completely determine such an invariant measure.
A specific form of such invariant measures are \emph{Markov measures}, characterized by the sample measures being measurable with respect to the \emph{past}: in our setting, this means that the $\mu_q$ only depend on $q_n, n < 0$ (cf.~e.g.~\cite{LedrappierYoung1988} or \cite{KuksinShirikyan2012}).

The theory of random dynamical systems now gives us the following result on the unique invariant measure for the RDS at hand, relating it to the stationary distribution of the Markov chain:
\begin{prop} \label{prop:samplemeasures}
 The RDS of the birth-death chain possesses a unique invariant Markov measure with sample measures
 $$\mu_{q} = \frac{1}{2} \delta_{a_0(q)} + \frac{1}{2} \delta_{a_1(q)},$$
 such that $\mathbb E[\mu_{q}] = \rho$, where $\rho$ is the unique stationary distribution from Proposition~\ref{PROP:stationary distribution}.
\end{prop}
\begin{proof}
By \cite[Proposition 4.5]{CrauelFlandoli1994}, we know that there exists a Markov measure $\mu$ such that $\mu_q(A(q)) =1$ almost surely, where $A(q)=\lbrace a_0(q),a_1(q)\rbrace$ is the attractor from Theorem~\ref{THM:strongAttractor}.
Its uniqueness and the fact that $\mathbb E[\mu_{q}] = \rho$ follow from the celebrated correspondence theorem, also called \emph{Ledrappier-LeJan-Crauel Theorem} (see \cite[Proposition 1.2.3]{LedrappierYoung1988} for a version that suffices for our situation and 
\cite[Theorem 4.2.9]{KuksinShirikyan2012} for the more general situation).

Using \cite{LeJan1987} (see also \cite[Lemma 2.19]{flandoli2016}) we can directly infer that either $\mu_{q} = \frac{1}{2} \delta_{a_0(q)} + \frac{1}{2} \delta_{a_1(q)}$ almost surely or $\mu_{q} = \delta_{a_i(q)}$ almost surely for $i=0$ or $i=1$ fixed. The latter case can now be excluded by combining \eqref{eq:randomcycle} and the invariance property $(\varphi_q^n)^* \mu_{q} = \mu_{\theta^n q}$ \cite[Proposition 1.3.27]{KuksinShirikyan2012}.
\end{proof}

%It is well known that $\rho$ is a stationary distribution for the cocycle map $\varphi$ if and only if $\mathbb{P}_=\otimes \rho$ is an invariant measure of $\varphi$, that is $\Phi_*(\mathbb{P}\times \rho)=\mathbb{P}\times\rho$. 

\subsection{Attractor structure for the Schlögl model} \label{sec:Schloegl}

For the extended reaction network of Example~\ref{ex:schloegl}, we show the existence of a weak attractor for the corresponding embedded Markov chain via its possession of a unique stationary distribution. This is formulated in the following proposition. As for the structure of the weak attractor, we provide a  conjecture based on numerical experiments whose proof is left as an open problem for future work.

%Following a similar reasoning as in the setting of the birth-death chain, let us first show that the embedded Markov chain of the Schlögl model (Example~\ref{ex:schloegl}) admits a unique stationary distribution.

\begin{prop}
\label{PROP:stationary distribution Schloegl}
    The embedded Markov chain of the Schl\"ogl model admits a unique stationary distribution. Therefore, the associated RDS admits a (unique) weak attractor.
\end{prop}

\begin{proof}
\noindent    Analogously to Proposition~\ref{PROP:stationary distribution}, we show that the quantity $\zeta$ in \eqref{eq:condition_stationary_1} is finite. Indeed,
    \[
        \prod_{j=0}^{x-1}\frac{P_1(j)}{P_{-1}(j+1)}=
       \prod_{j=0}^{x-1}\left(\frac{\gamma_1+\gamma_3 j (j-1)}{\gamma_2(j+1)+\gamma_4 j(j+1)(j-1)}\right)\cdot
        \left( \frac{\mu(j+1)}{\mu(j)} \right)=
    \]
    \[
        =\frac{1}{x!}\prod_{j=0}^{x-1}\left(\frac{\gamma_1+\gamma_3 j (j-1)}{\gamma_2+\gamma_4 j(j-1)}\right)\cdot
        \left( \frac{\mu(j+1)}{\mu(j)} \right).
    \]
Both terms inside the product are bounded for all $j\in\mathbb{N}_0$. Thus, there exists $C>0$ such that
\[
   \zeta \leq \sum_{x=1}^{\infty}
  \frac{1}{x!}C^x<\infty.
\]
The finiteness of the right-hand side follows from Stirling's approximation similarly as in Proposition~\ref{PROP:stationary distribution}.
\end{proof}

The stationary distribution of the Schlögl model is depicted in Figure~\ref{fig:Schloegl_stat_distr}. 
%Combining the results of Prop.~\ref{PROP:stationary distribution Schloegl} and Theorem~\ref{THM:existenceWeakAttractor}, we may follow that the RDS $(\theta,\varphi)$ associated to the embedded Markov chain of the Schlögl model admits a (unique) weak attractor. \me{Why did you put it out of the proposition?}

Numerical explorations suggest the following structure of the weak attractor: similarly to the simple birth-death chain, there is partial synchronization with respect to a partition of $\mathbb N_0$ into odd and even numbers, as exemplified in Figure~\ref{fig:Schloegl_2dim}(a). This suggests that the weak attractor consists again of two random points that most likely form a random periodic orbit. However, as we see in Figure~\ref{fig:Schloegl_2dim}(b), the distance of these two random points will not be one but probably larger; maybe even depending on the random realization. We emphasize that this difference to the simple birth-death chain in terms of the non-synchronizing trajectories reflects the distinction between monostability (the thick diagonal in the two-point motion is absorbing) and bistability (the thick diagonal in the two-point motion has repelling parts) as seen in the large-volume limiting ODEs (cf.~Figure~\ref{fig:ODE}). This fact is also mirrored by the respective stationary distributions of the one- and two-point motion, as illustrated in Figures~\ref{fig:BDP_stat_distr} and~\ref{fig:Schloegl_stat_distr}.

A proof of the associated structure of the random attractor is much more involved than in the simple birth-death case and will be left for future work. 

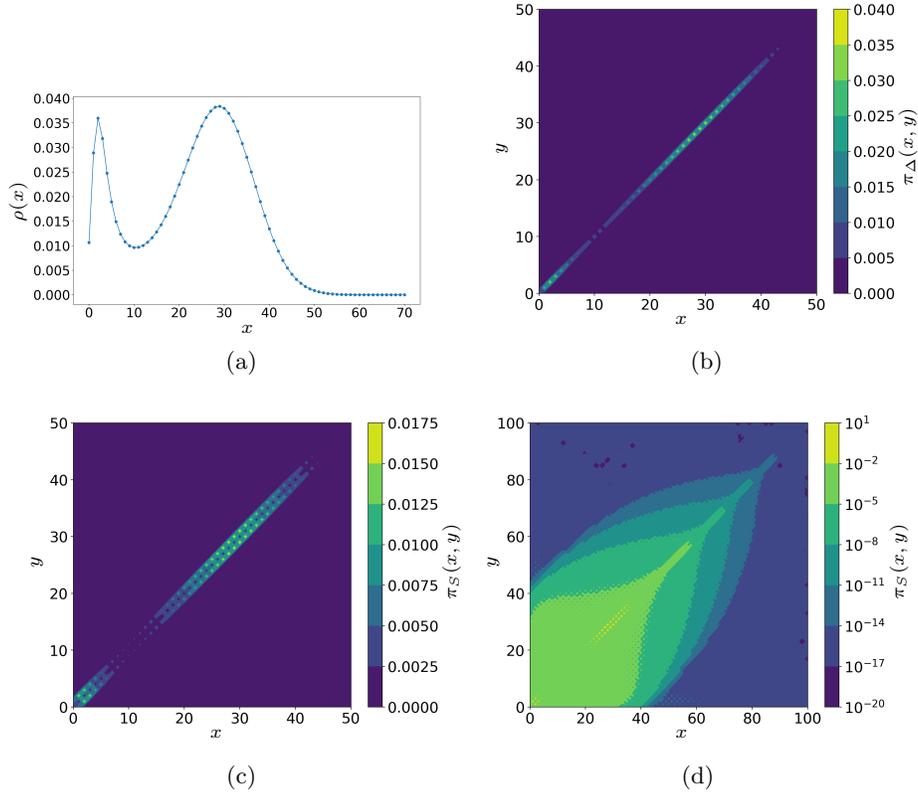
\begin{figure}
    \centering
\begin{subfigure}[b]{0.49\textwidth}
        \centering
      \begin{overpic}[width=\textwidth]{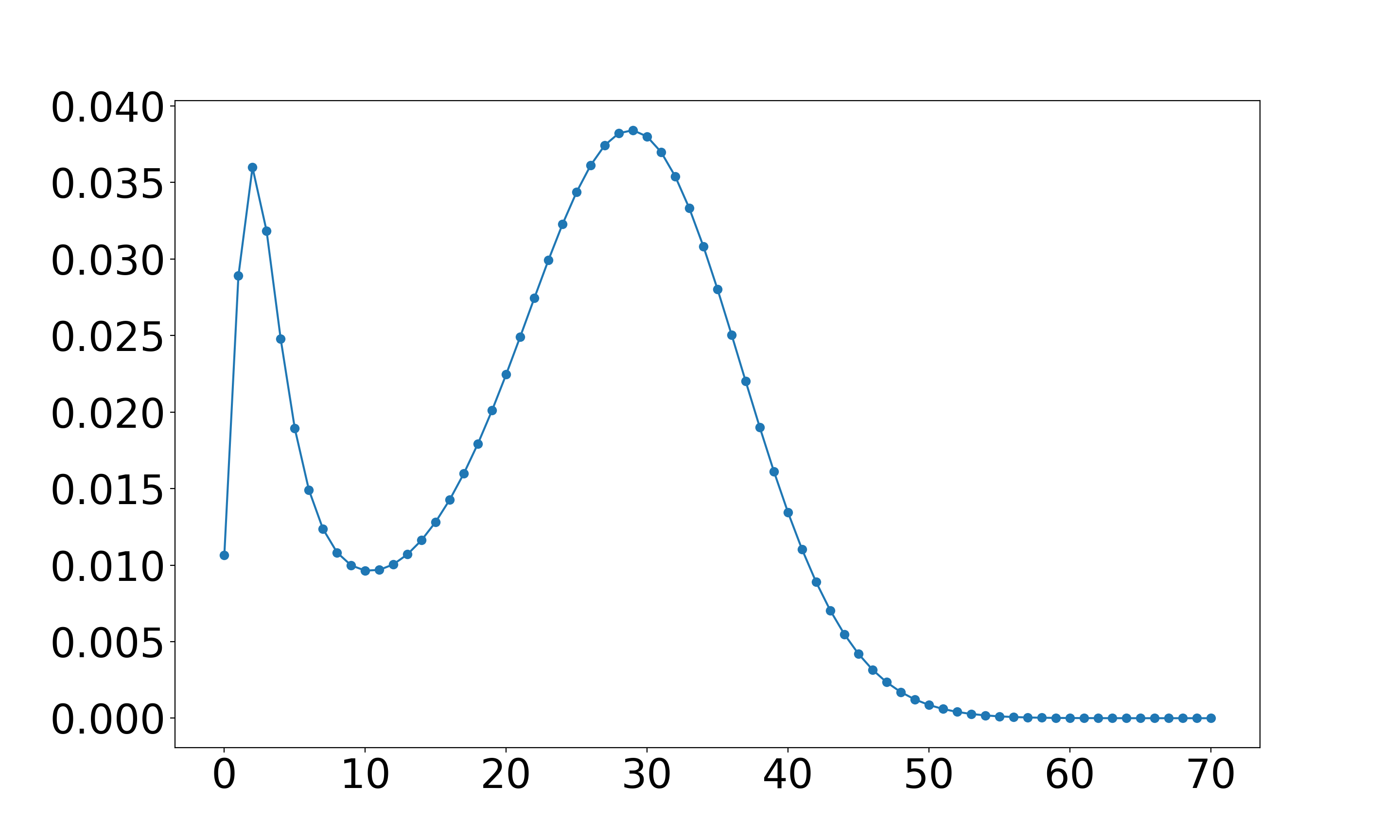}
        \put(-2,25){\scriptsize \begin{turn}{90} $\rho(x)$ \end{turn} }
        \put(50,0){\scriptsize $x$}
        \end{overpic}
        \caption{}
    \end{subfigure}
    \hfill
    \begin{subfigure}[b]{0.49\textwidth}
        \centering
      \begin{overpic}[width=\textwidth]{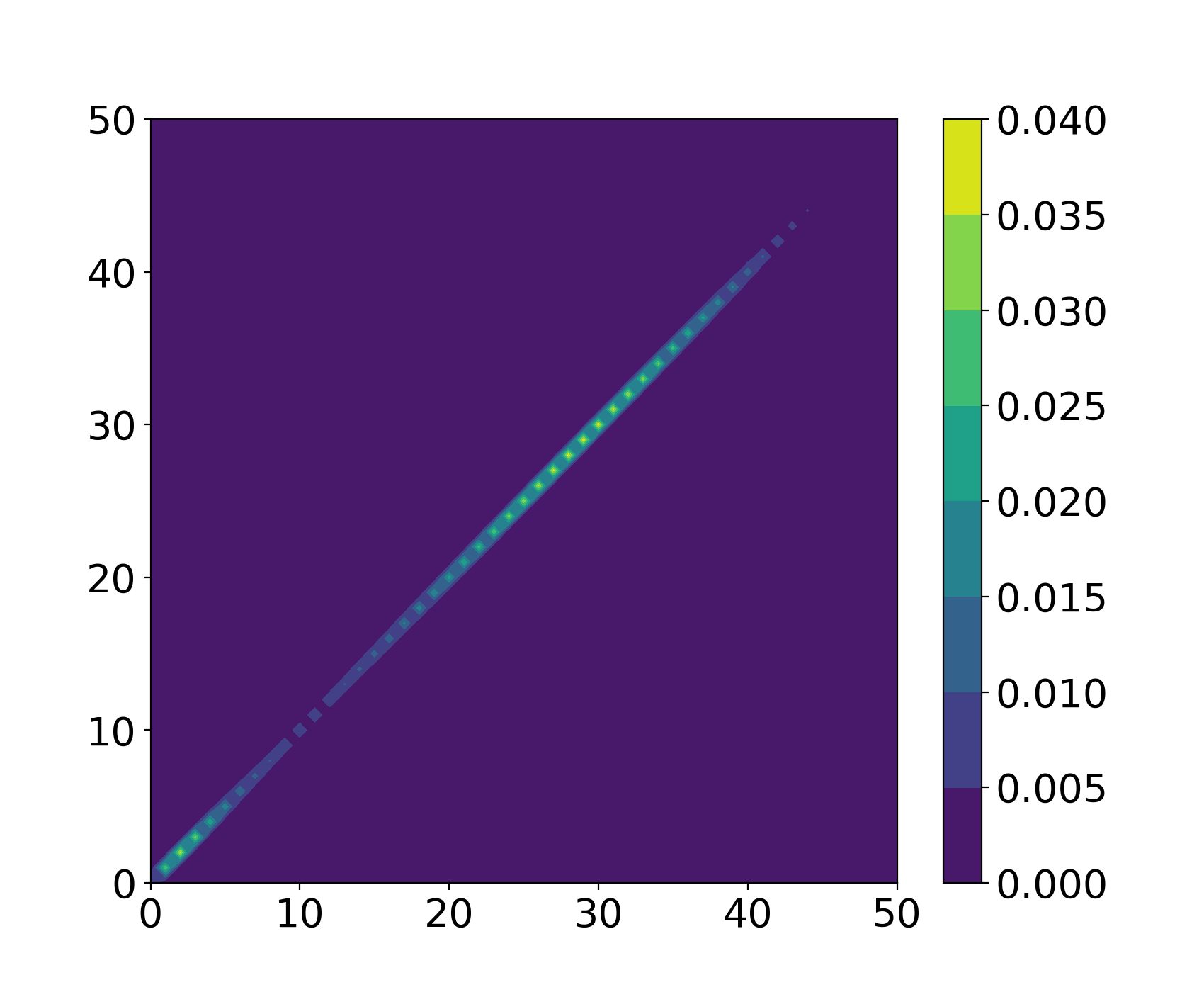}
        \put(3,40){\scriptsize \begin{turn}{90} $y$ \end{turn} }
        \put(43,2){\scriptsize $x$}
        \put(93,33){\scriptsize \begin{turn}{90} $\pi_\Delta(x,y)$ \end{turn} }
        \end{overpic}
        \caption{}
    \end{subfigure}
    \hfill\begin{subfigure}[b]{0.49\textwidth}
        \centering
        \begin{overpic}[width=\textwidth]{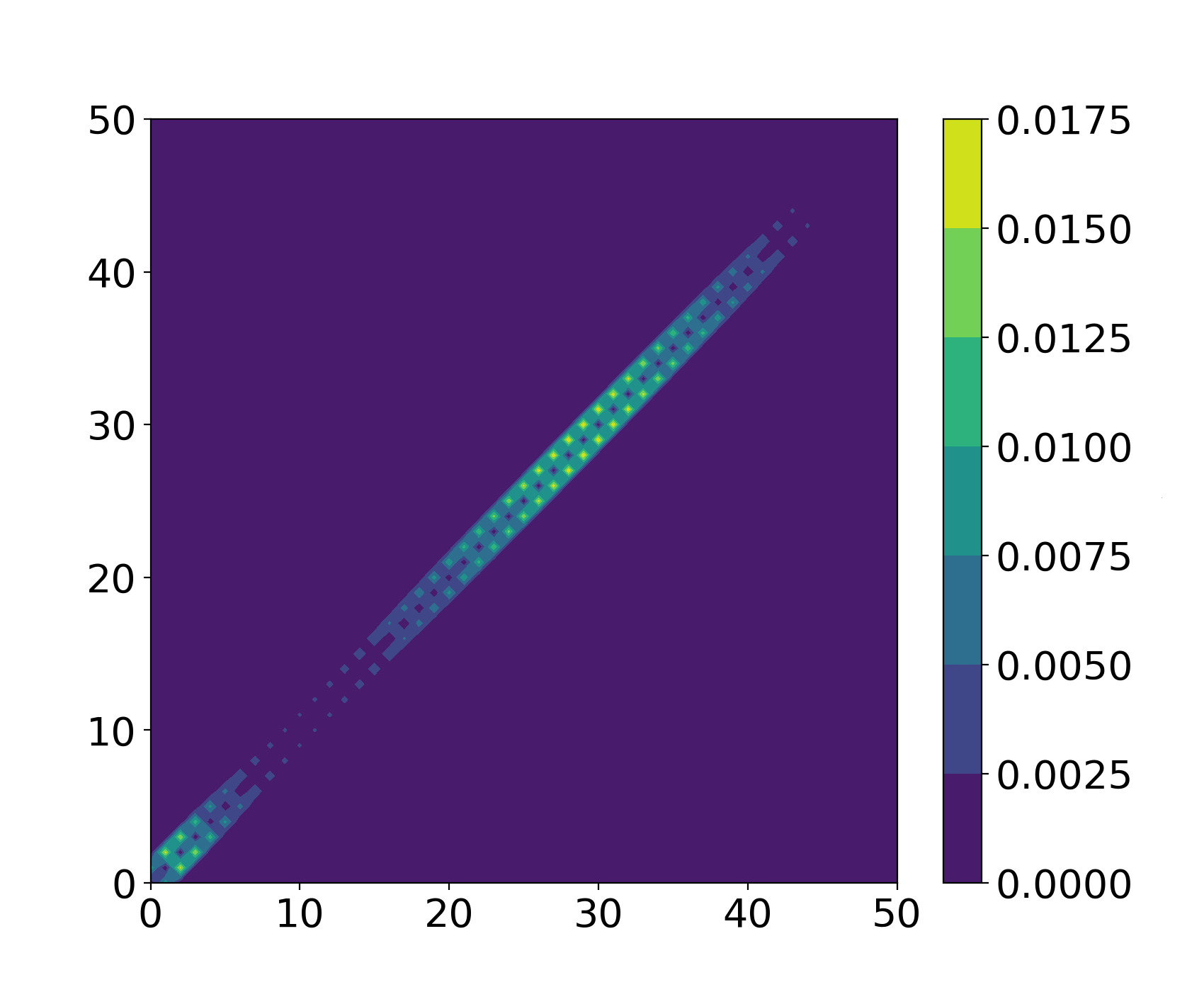}
        \put(3,40){\scriptsize \begin{turn}{90} $y$ \end{turn} }
        \put(43,2){\scriptsize $x$}
        \put(95,33){\scriptsize \begin{turn}{90} $\pi_S(x,y)$ \end{turn} }
        \end{overpic}
        \caption{}
    \end{subfigure}
    \hfill\begin{subfigure}[b]{0.49\textwidth}
        \centering
        \begin{overpic}[width=\textwidth]{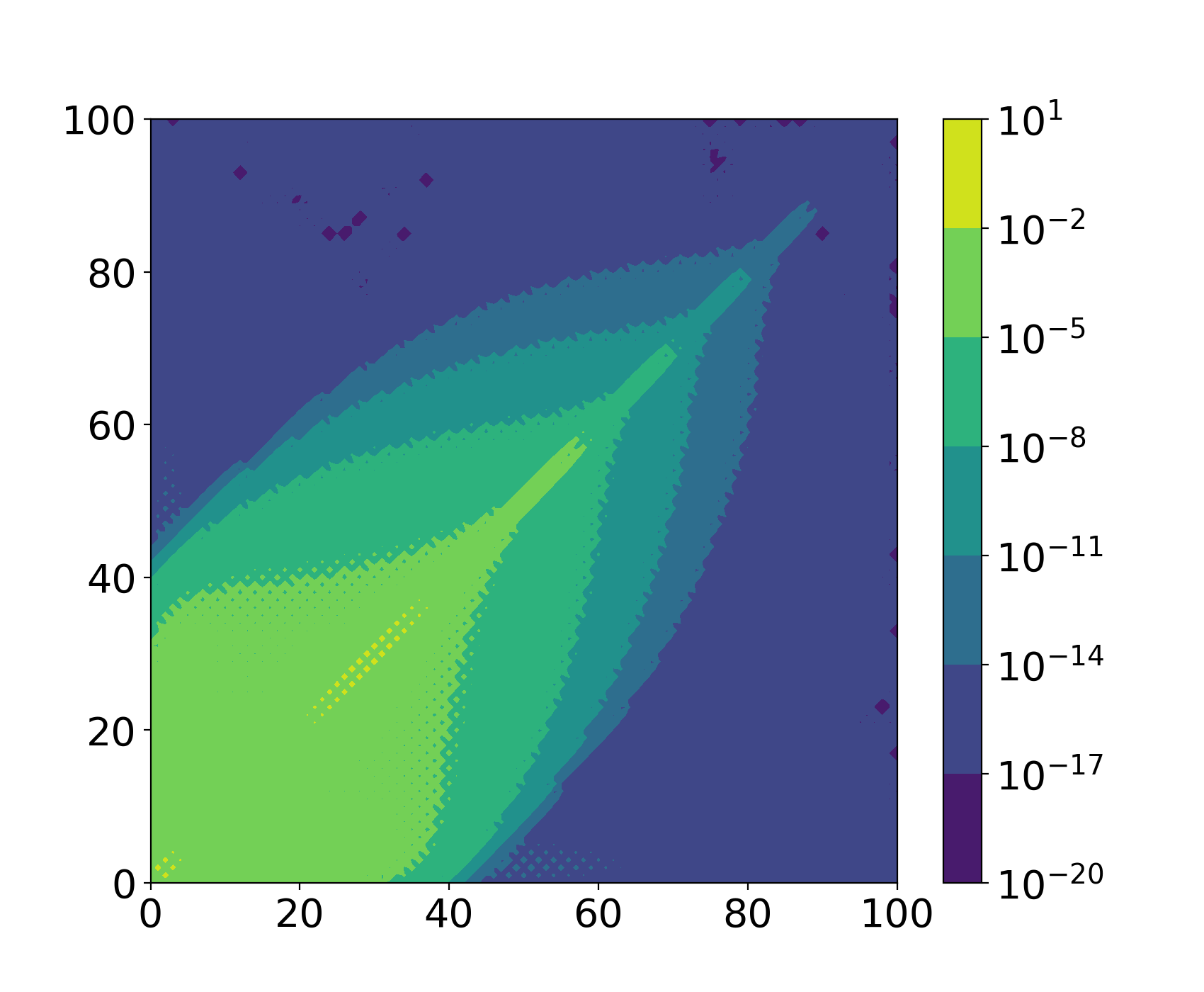}
        \put(3,40){\scriptsize \begin{turn}{90} $y$ \end{turn} }
        \put(45,2){\scriptsize $x$}
        \put(93,33){\scriptsize \begin{turn}{90} $\pi_S(x,y)$ \end{turn} }
        \end{overpic}
        \caption{}
    \end{subfigure}
    \hfill
    \caption{Stationary distributions of one- and two-point motion of the Schlögl model (embedded Markov chain). (a) Stationary distribution $\rho$ of the one-point motion and (b), (c)
   stationary distributions $\pi_{\Delta}$ and $\pi_{S}$ of the two-point motion, respectively. Here, the index $S\subset \mathbb{X}^2 \setminus \Delta$ refers to the support of the stationary distribution for the two-point motion starting with odd distance between the two coordinates $x_0,y_0$ of the initial state.  
    %stationary distributions $\pi_{\mathcal{A}}$, $\pi_{\bar{\mathcal{A}}}$ of the two-point motion on the sets $\mathcal{A}:=(W_0\times W_0)\cup (W_1 \times W_1)$ resp. $\bar{\mathcal{A}}:=(W_0\times W_1)\cup (W_1 \times W_0)=\mathbb{X}^2\setminus \mathcal{A}$.
    (d) the same as (c) only on a logarithmic scale.  $\rate_1=6$, $\rate_2=3.5, \rate_3=0.4, \rate_4=0.0105$. 
    }
    \label{fig:Schloegl_stat_distr}
\end{figure}

%\begin{prop}
% \label{LEMMA: SchloeglWeakAttractor}
%        The embedded Markov chain of the Schl\"ogl model admits a weak attractor.   
%\end{prop}
    %
%\begin{proof}
%    The proof follows the same lines as Lemma~\ref{lemma:weak attractor} for the birth-death chain.
%    \me{At least comment on the steps and why they go through here as well.}
%\end{proof}

\section{Conclusion}
\label{sec:conclusion}
We have introduced the phenomenon of (partial) time-shifted synchronization for reaction jump processes, using their description via the augmented and embedded Markov chain whose properties as a random dynamical system can be specified through the structure of the corresponding random attractor. As a first example we have given a full proof of partial synchronization for the birth-death process, finding the random attractor of the embedded chain to be a random periodic orbit of period $2$. We have demonstrated that for extensions of this basic example, such as the Schl\"{o}gl model, one may expect a similar structure; however, there is an apparent difference in the relation of the two random points.

Depending on the dimensionality of the system, i.e.~the number of different species, or other types of reaction rates, e.g.~Michaelis-Menten,  we expect various forms of (partial) synchronization and random periodic behavior for general reaction systems and leave it as a research direction for the future to work towards a categorization of such processes and their corresponding chains in terms of random dynamical systems theory. The main goal of this paper has been to relate synchronization phenomena for the different formalizations of the chemical reaction process and to give a first complete and rigorous description of a classical example. 
The strategy of understanding the two-point motion, establishing the attractor via a unique stationary distribution with sufficient decay and then specifying the structure of the attractor and the statistical equilibrium supported there,  based on the two-point motion analysis, may well be generalizable. In particular, one may consider time-shifted synchronization for general Markov jump processes, not necessarily given by reaction networks, via the RDS description of the related space-time Markov chains.
Moreover, an intriguing point of more detailed investigation will concern the quantification and statistics of the delay times found for time-shifted synchronization, which may be of high interest also for the applied side of chemical reaction processes.

Furthermore, our work has brought up additional questions that remain open, to our knowledge. Can one find an example coming from a chemical reaction network with no (partial) synchronization at all, i.e.~where each synchronization class is a singleton? May one describe bifurcations of the attractor, for example in an easy model such as Schl\"{o}gl's, via variation of the parameters? How are attractors of the described processes related to the attractors of the corresponding volume-scaled systems, i.e.~the Langevin SDEs or the reaction rate ODEs?
Additionally, from the RDS point of view it will also be intriguing to give general criteria for weak attractors being (strong) pullback attractors in discrete state spaces. 
In summary, we see this work as a first step towards a deeper structural understanding of reaction jump processes via RDS theory and, conversely, a motivation for a broader understanding of random attractors withing the dichotomy between the discrete and the continuous.

\subsection*{Acknowledgements}
We acknowledge the support of Deutsche Forschungsgemeinschaft (DFG) through CRC 1114 and under Germany's Excellence Strategy -- The Berlin Mathematics Research Center MATH+ (EXC-2046/1, project ID 390685689).  
%M.~E. and G. O.-M. have been supported by Germany's Excellence Strategy -- The Berlin Mathematics Research Center MATH+ (EXC-2046/1, project ID: 390685689), in particular for the research project AA1-8: Random Bifurcations in Chemical Reaction Networks. M.~E. additionally thanks the DFG-funded SPP 2298 and CRC 1114 for supporting his research.
M.~E. additionally thanks the DFG-funded SPP 2298 for supporting his research.
G.~O.-M. also thanks FU Berlin for a 3-month Forschungsstipendium. The authors gratefully acknowledge Dennis Chemnitz for fruitful discussions.

%\begin{itemize}
    %\item Explore specific higher-dimensional systems.
    %\item What about other reaction rate types (e.g. Michaelis-Menten type)
    %\item Add a simple example where there is no synchronization at all, that is each synchronization class is a singleton (all maps are invertible). Is this possible in chemical reaction networks?
   % \item Discuss about how the attractor of the Schlögl model changes with the parameters.
    %\item How are attractors of the process related to the attractors of the corresponding volume-scaled systems (like Langevin equations, or ODEs).
    %\item Are weak attractors also pullback attractors in general? Or under mild conditions?
    
%\end{itemize}

\bibliography{bibliography}
\bibliographystyle{abbrvnat}
%\biboptions{numbers}
\end{document}